\documentclass{amsart}

\usepackage{paralist}
\usepackage[]{hyperref}
\usepackage{color}
\usepackage{pgf,tikz, tikzscale}
\usepackage{amsmath, amsthm, amssymb, amsfonts}
\usetikzlibrary{decorations.pathreplacing,decorations.markings}
\usetikzlibrary{arrows, knots, external, positioning}

\usetikzlibrary{arrows.meta}
\tikzstyle arrowstyle=[scale=2]
\tikzstyle mdirected=[postaction={decorate,decoration={markings,
    mark=at position.5 with {\arrow[arrowstyle]{stealth}}}}]
\tikzstyle edirected=[postaction={decorate,decoration={markings,
    mark=at position 1.0 with {\arrow[arrowstyle]{stealth}}}}]

\tikzset{bullet/.style={
shape = circle,fill = black, inner sep = 0pt, outer sep = 0pt, minimum size = 0.35em, line width = 0pt, draw=black!100}}

\usetikzlibrary{circuits.logic.US,circuits.logic.IEC,fit}
\newcommand\addvmargin[1]{
  \node[fit=(current bounding box),inner ysep=#1,inner xsep=0]{};
}
\usepackage{tabularx}
\usepackage{array}
\newcolumntype{C}[1]{>{\centering\arraybackslash}m{#1}}
\renewcommand{\baselinestretch}{1.1}

\newtheorem{thm}{Theorem}[section]

\newtheorem{prop}[thm]{Proposition}

\theoremstyle{definition}

\numberwithin{equation}{section}

\theoremstyle{definition}

\theoremstyle{definition}
\newtheorem{remark}{Remark}[section]
\theoremstyle{definition}

\numberwithin{equation}{section}
\newcommand{\ra}{{\rightarrow}}
\newcommand{\nra}{{\nrightarrow}}
\newcommand{\gen}{\text{gen}}
\newcommand{\spn}{\text{spn}}
\newcommand{\ord}{\text{ord}}
\newcommand{\diag}{\text{diag}}
\newcommand{\z}{{\mathbb Z}}
\newcommand{\q}{{\mathbb Q}}
\newcommand{\f}{{\mathbb F}}
\newcommand{\rank}{\text{rank}}
\newcommand{\n}{{\mathbb N}}
\newcommand{\e}{{\epsilon}}
\newcommand{\RN}[1]{%
  \textup{\uppercase\expandafter{\romannumeral#1}}%
}

\begin{document}

%opening
\title[Rational homology disk smoothings and Lefschetz fibrations]{Rational homology disk smoothings and Lefschetz fibrations}

\author{Hakho Choi}
\address{School of Mathematics, Korea Institute for Advanced Study, 85 Hoegiro, Dongdaemun-gu, Seoul 02455, Republic of Korea }
\email{hakho@kias.re.kr}
\address{Center for Quantum Structures in Modules and Spaces, Seoul National University,
Seoul 08226, Korea}
\email{hako85@snu.ac.kr}

\thanks{}
\date{July 19, 2022}
%\date{February 8, 2018; revised October 31, 2018}

%\begin{abstract} 
%In this article, we construct  genus-$1$ positive allowable Lefschetz fibrations on rational homology disk smoothings of weighted homogeneous surface singularities whose resolution graphs are $3$-legged with central bad vertex. The global monodromies of the Lefschetz fibrations are obtained by applying monodromy substitutions from the global monodromies of the minimal resolutions after introducing appropriate canceling pairs so that the induced contact structure on the boundary is the Milnor fillable contact structure. We also show that some of these rational homology disk smoothings are unique up to diffeomorphism. 
%\end{abstract}

\begin{abstract} 
In this article, we generalize the results discussed in \cite{MR2783383} by introducing a genus to generic fibers of Lefschetz fibrations. That is, we give families of relations in the mapping class groups of genus-1 surfaces with boundaries that represent rational homology disk smoothings of weighted homogeneous surface singularities whose resolution graphs are $3$-legged with a bad central vertex.
\end{abstract}

\maketitle
\hypersetup{linkcolor=black}

\section{Introduction}
Rational blowdown surgery, which was introduced by Fintushel-Stern~\cite{MR1484044} and generalized by J. Park~\cite{MR1490654}, is a surgery operation that replace a linear plumbing $C_{p,q}$ of $2$-spheres with a rational homology ball $B_{p,q}$ (i.e., $H_{*}(B_{p,q}, \mathbb{Q})\cong H_{*}(B^4, \mathbb{Q})$). As rational blowdown surgery reduces the second Betti number and the Seiberg-Witten invariants of the surgered manifold are determined by that of the original manifold under mild conditions, it is one of the most powerful tools in constructing smooth $4$-manifolds with small Euler characteristic~\cite{MR2169045}~\cite{MR2125736} ~\cite{MR2189231}. Further, it can be used to construct simply connected complex surface of general type with $p_g=0$ and $K^2=2, 3, 4$ because $C_{p,q}$ is the minimal resolution of cyclic quotient surfaces singularities $A_{p^2,pq-1}$, and $B_{p,q}$ is the rational homology disk smoothing (i.e., Milnor fiber with vanishing Milnor number) of $A_{p^2,pq-1}$ ~\cite{MR2357500}~\cite{MR2469529}~\cite{MR2496050}.
From these perspectives, researchers attempted to identify other normal surface singularities admitting rational homology disk smoothing ($\mathbb{Q}$HD for short). In particular,
there is a complete classification of resolution graphs admitting $\mathbb{Q}$HD smoothing for the case of weighted homogeneous surface singularities~\cite{MR2399141}, ~\cite{MR2843099}. They are all $3$-legged or $4$-legged graphs. We focus on $3$-legged graphs in this article (refer to Figure~\ref{resolutiongraphs} for the complete list, and for an exhaustive list of $4$-legged cases, refer to Figure 2 in ~\cite{MR2843099}). 
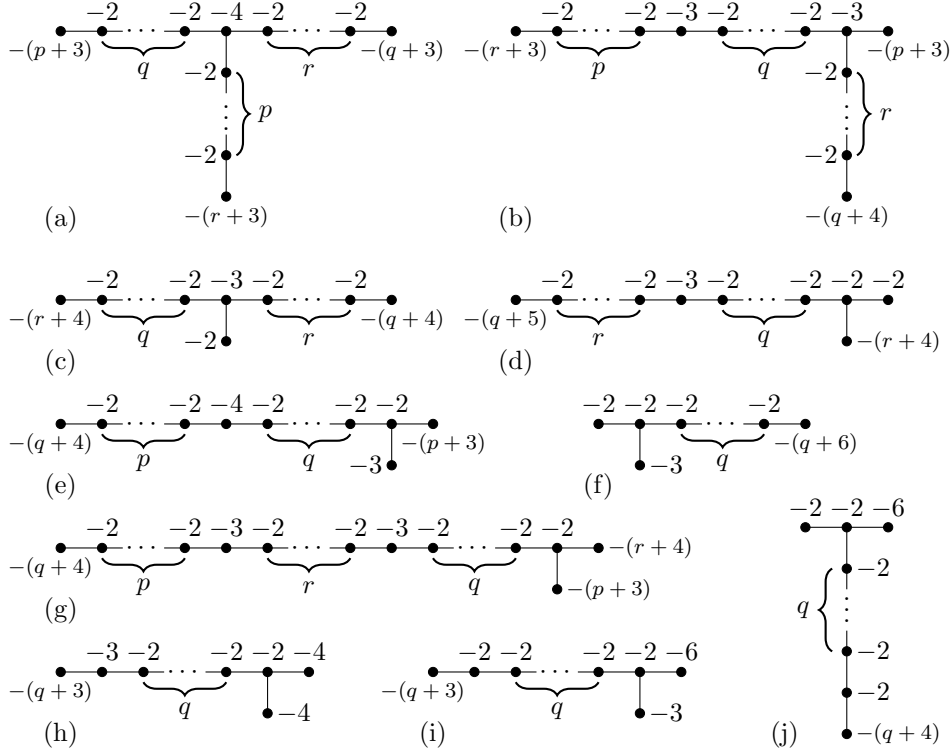
\begin{figure}[h]
\begin{tikzpicture}[scale=0.55]
\begin{scope}
\node[bullet] at (0,0){};
\node[bullet] at (1,0){};
\node[bullet] at (3,0){};
\node[bullet] at (4,0){};
\node[bullet] at (5,0){};
\node[bullet] at (7,0){};
\node[bullet] at (8,0){};

\node[bullet] at (4,-1){};
\node[bullet] at (4,-3){};
\node[bullet] at (4,-4){};

\node[below] at (-0.25,0){\footnotesize{$-(p+3)$}};
\node[above] at (1,0){$-2$};
\node[above] at (3,0){$-2$};
\node[above] at (4,0){$-4$};
\node[above] at (5,0){$-2$};
\node[above] at (7,0){$-2$};
\node[below] at (8.25,0){\footnotesize{$-(q+3)$}};

\node[left] at (4,-1){$-2$};
\node[left] at (4,-3){$-2$};
\node[below] at (4,-4){\footnotesize{$-(r+3)$}};
\node[below] at (0,-4){(a)};

\node at (2,0){$\cdots$};
\node at (6,0){$\cdots$};
\node at (4,-1.9){$\vdots$};

\draw (0,0)--(1.5,0);
\draw (2.5,0)--(5.5,0);
\draw (6.5,0)--(8,0);
\draw (4,0)--(4,-1.5);
\draw (4,-2.5)--(4,-4);

	\draw [thick,decorate,decoration={brace,mirror,amplitude=5pt},xshift=0pt,yshift=-7pt]
	(1,0) -- (3,0) node [black,midway,yshift=-11pt] 
	{$q$};

	\draw [thick,decorate,decoration={brace,mirror,amplitude=5pt},xshift=0pt,yshift=-7pt]
	(5,0) -- (7,0) node [black,midway,yshift=-11pt] 
	{$r$};

	\draw [thick,decorate,decoration={brace,amplitude=5pt},xshift=0pt,xshift=7pt]
	(4,-1) -- (4,-3) node [black,midway,xshift=11pt] 
	{$p$};
\end{scope}

%(b) family
\begin{scope}[shift={(11.,0)}]
\node[bullet] at (0,0){};
\node[bullet] at (1,0){};
\node[bullet] at (3,0){};
\node[bullet] at (4,0){};
\node[bullet] at (5,0){};
\node[bullet] at (7,0){};
\node[bullet] at (8,0){};

\node[bullet] at (8,-1){};
\node[bullet] at (8,-3){};
\node[bullet] at (8,-4){};
\node[bullet] at (9, 0){};

\node[below] at (-.25,0){\footnotesize{$-(r+3)$}};
\node[above] at (1,0){$-2$};
\node[above] at (3,0){$-2$};
\node[above] at (4,0){$-3$};
\node[above] at (5,0){$-2$};
\node[above] at (7,0){$-2$};
\node[above] at (8,0){$-3$};

\node[left] at (8,-1){$-2$};
\node[left] at (8,-3){$-2$};
\node[below] at (8,-4){\footnotesize{$-(q+4)$}};
\node[below] at (9.5,0){\footnotesize{$-(p+3)$}};

\node at (2,0){$\cdots$};
\node at (6,0){$\cdots$};
\node at (8,-1.9){$\vdots$};

\draw (0,0)--(1.5,0);
\draw (8,0)--(9,0);
\draw (2.5,0)--(5.5,0);
\draw (6.5,0)--(8,0);
\draw (8,0)--(8,-1.5);
\draw (8,-2.5)--(8,-4);

	\draw [thick,decorate,decoration={brace,mirror,amplitude=5pt},xshift=0pt,yshift=-7pt]
	(1,0) -- (3,0) node [black,midway,yshift=-11pt] 
	{$p$};

	\draw [thick,decorate,decoration={brace,mirror,amplitude=5pt},xshift=0pt,yshift=-7pt]
	(5,0) -- (7,0) node [black,midway,yshift=-11pt] 
	{$q$};

	\draw [thick,decorate,decoration={brace,amplitude=5pt},xshift=0pt,xshift=7pt]
	(8,-1) -- (8,-3) node [black,midway,xshift=11pt] 
	{$r$};
\node[below] at (0,-4){(b)};
\end{scope}

%(c) family

\begin{scope}[shift={(0,-6.5)}]
\node[bullet] at (0,0){};
\node[bullet] at (1,0){};
\node[bullet] at (3,0){};
\node[bullet] at (4,0){};
\node[bullet] at (5,0){};
\node[bullet] at (7,0){};
\node[bullet] at (8,0){};

\node[bullet] at (4,-1){};

\node[below] at (-0.25,0){\footnotesize{$-(r+4)$}};
\node[above] at (1,0){$-2$};
\node[above] at (3,0){$-2$};
\node[above] at (4,0){$-3$};
\node[above] at (5,0){$-2$};
\node[above] at (7,0){$-2$};
\node[below] at (8.25,0){\footnotesize{$-(q+4)$}};

\node[left] at (4,-1){$-2$};
\node at (0,-1.5){(c)};

\node at (2,0){$\cdots$};
\node at (6,0){$\cdots$};

\draw (0,0)--(1.5,0);
\draw (2.5,0)--(5.5,0);
\draw (6.5,0)--(8,0);
\draw (4,0)--(4,-1.);

	\draw [thick,decorate,decoration={brace,mirror,amplitude=5pt},xshift=0pt,yshift=-7pt]
	(1,0) -- (3,0) node [black,midway,yshift=-11pt] 
	{$q$};

	\draw [thick,decorate,decoration={brace,mirror,amplitude=5pt},xshift=0pt,yshift=-7pt]
	(5,0) -- (7,0) node [black,midway,yshift=-11pt] 
	{$r$};

\end{scope}

%(d) family

\begin{scope}[shift={(11,-6.5)}]
\node[bullet] at (0,0){};
\node[bullet] at (1,0){};
\node[bullet] at (3,0){};
\node[bullet] at (4,0){};
\node[bullet] at (5,0){};
\node[bullet] at (7,0){};
\node[bullet] at (8,0){};
\node[bullet] at (9,0){};

\node[bullet] at (8,-1){};

\node[below] at (-0.25,0){\footnotesize{$-(q+5)$}};
\node[above] at (1,0){$-2$};
\node[above] at (3,0){$-2$};
\node[above] at (4,0){$-3$};
\node[above] at (5,0){$-2$};
\node[above] at (7,0){$-2$};
\node[above] at (8,0){$-2$};
\node[above] at (9,0){$-2$};

\node[right] at (8,-1){\footnotesize{$-(r+4)$}};
\node at (0,-1.5){(d)};

\node at (2,0){$\cdots$};
\node at (6,0){$\cdots$};

\draw (0,0)--(1.5,0);
\draw (2.5,0)--(5.5,0);
\draw (6.5,0)--(9,0) (8,0)--(8,-1);

	\draw [thick,decorate,decoration={brace,mirror,amplitude=5pt},xshift=0pt,yshift=-7pt]
	(1,0) -- (3,0) node [black,midway,yshift=-11pt] 
	{$r$};

	\draw [thick,decorate,decoration={brace,mirror,amplitude=5pt},xshift=0pt,yshift=-7pt]
	(5,0) -- (7,0) node [black,midway,yshift=-11pt] 
	{$q$};

\end{scope}

%(e)family

\begin{scope}[shift={(0,-9.5)}]
\node[bullet] at (0,0){};
\node[bullet] at (1,0){};
\node[bullet] at (3,0){};
\node[bullet] at (4,0){};
\node[bullet] at (5,0){};
\node[bullet] at (7,0){};
\node[bullet] at (8,0){};
\node[bullet] at (9,0){};

\node[bullet] at (8,-1){};

\node[below] at (-0.25,0){\footnotesize{$-(q+4)$}};
\node[above] at (1,0){$-2$};
\node[above] at (3,0){$-2$};
\node[above] at (4,0){$-4$};
\node[above] at (5,0){$-2$};
\node[above] at (7,0){$-2$};
\node[above] at (8,0){$-2$};
\node[below] at (9.25,0){\footnotesize{$-(p+3)$}};

\node[left] at (8,-1){$-3$};
\node at (0,-1.5){(e)};

\node at (2,0){$\cdots$};
\node at (6,0){$\cdots$};

\draw (0,0)--(1.5,0);
\draw (2.5,0)--(5.5,0);
\draw (6.5,0)--(9,0) (8,0)--(8,-1);

	\draw [thick,decorate,decoration={brace,mirror,amplitude=5pt},xshift=0pt,yshift=-7pt]
	(1,0) -- (3,0) node [black,midway,yshift=-11pt] 
	{$p$};

	\draw [thick,decorate,decoration={brace,mirror,amplitude=5pt},xshift=0pt,yshift=-7pt]
	(5,0) -- (7,0) node [black,midway,yshift=-11pt] 
	{$q$};

\end{scope}

%(f) family

\begin{scope}[shift={(13,-9.5)}]
\node[bullet] at (0,0){};
\node[bullet] at (1,0){};
\node[bullet] at (1,-1){};

\node[bullet] at (2,0){};
\node[bullet] at (4,0){};
\node[bullet] at (5,0){};

\node[above] at (-0.,0){$-2$};
\node[above] at (1,0){$-2$};
\node[above] at (2,0){$-2$};
\node[above] at (4,0){$-2$};
\node[right] at (1,-1){$-3$};

\node[below] at (5.25,0){\footnotesize{$-(q+6)$}};
\node at (0,-1.5){(f)};

\node at (3,0){$\cdots$};

\draw (0,0)--(2.5,0);
\draw (3.5,0)--(5,0) (1,0)--(1,-1);

	\draw [thick,decorate,decoration={brace,mirror,amplitude=5pt},xshift=0pt,yshift=-7pt]
	(2,0) -- (4,0) node [black,midway,yshift=-11pt] 
	{$q$};

\end{scope}

%(g) family

\begin{scope}[shift={(0,-12.5)}]
\node[bullet] at (0,0){};
\node[bullet] at (1,0){};
\node[bullet] at (3,0){};
\node[bullet] at (4,0){};
\node[bullet] at (5,0){};
\node[bullet] at (7,0){};
\node[bullet] at (8,0){};
\node[bullet] at (9,0){};
\node[bullet] at (11,0){};
\node[bullet] at (12,0){};
\node[bullet] at (13,0){};

\node[bullet] at (12,-1){};

\node[below] at (-0.25,0){\footnotesize{$-(q+4)$}};
\node[above] at (1,0){$-2$};
\node[above] at (3,0){$-2$};
\node[above] at (4,0){$-3$};
\node[above] at (5,0){$-2$};
\node[above] at (7,0){$-2$};
\node[above] at (8,0){$-3$};
\node[above] at (9.,0){$-2$};
\node[above] at (11.,0){$-2$};
\node[above] at (12.,0){$-2$};

\node at (14.25,0){\footnotesize{$-(r+4)$}};

\node[right] at (12,-1){\footnotesize{$-(p+3)$}};
\node at (0,-1.5){(g)};

\node at (2,0){$\cdots$};
\node at (6,0){$\cdots$};
\node at (10,0){$\cdots$};

\draw (0,0)--(1.5,0);
\draw (2.5,0)--(5.5,0);
\draw (6.5,0)--(9.5,0) (10.5,0)--(13,0) (12,0)--(12,-1);

	\draw [thick,decorate,decoration={brace,mirror,amplitude=5pt},xshift=0pt,yshift=-7pt]
	(1,0) -- (3,0) node [black,midway,yshift=-11pt] 
	{$p$};

	\draw [thick,decorate,decoration={brace,mirror,amplitude=5pt},xshift=0pt,yshift=-7pt]
	(5,0) -- (7,0) node [black,midway,yshift=-11pt] 
	{$r$};

\draw [thick,decorate,decoration={brace,mirror,amplitude=5pt},xshift=0pt,yshift=-7pt]
	(9,0) -- (11,0) node [black,midway,yshift=-11pt] 
	{$q$};

\end{scope}

%(h) family

\begin{scope}[shift={(0,-15.5)}]
\node[bullet] at (0,0){};
\node[bullet] at (1,0){};
\node[bullet] at (2,0){};
\node[bullet] at (4,0){};
\node[bullet] at (5,0){};
\node[bullet] at (6,0){};

\node[bullet] at (5,-1){};

\node[below] at (-0.25,0){\footnotesize{$-(q+3)$}};
\node[above] at (1,0){$-3$};
\node[above] at (2,0){$-2$};
\node[above] at (4,0){$-2$};
\node[above] at (5,0){$-2$};
\node[above] at (6,0){$-4$};

\node[right] at (5,-1){$-4$};
\node at (0,-1.5){(h)};

\node at (3,0){$\cdots$};

\draw (0,0)--(2.5,0);
\draw (3.5,0)--(6,0);
\draw (5,0)--(5,-1);

	\draw [thick,decorate,decoration={brace,mirror,amplitude=5pt},xshift=0pt,yshift=-7pt]
	(2,0) -- (4,0) node [black,midway,yshift=-11pt] 
	{$q$};

\end{scope}

%(i) family

\begin{scope}[shift={(9,-15.5)}]
\node[bullet] at (0,0){};
\node[bullet] at (1,0){};
\node[bullet] at (2,0){};
\node[bullet] at (4,0){};
\node[bullet] at (5,0){};
\node[bullet] at (6,0){};

\node[bullet] at (5,-1){};

\node[below] at (-0.25,0){\footnotesize{$-(q+3)$}};
\node[above] at (1,0){$-2$};
\node[above] at (2,0){$-2$};
\node[above] at (4,0){$-2$};
\node[above] at (5,0){$-2$};
\node[above] at (6,0){$-6$};

\node[right] at (5,-1){$-3$};
\node at (0,-1.5){(i)};

\node at (3,0){$\cdots$};

\draw (0,0)--(2.5,0);
\draw (3.5,0)--(6,0);
\draw (5,0)--(5,-1);

	\draw [thick,decorate,decoration={brace,mirror,amplitude=5pt},xshift=0pt,yshift=-7pt]
	(2,0) -- (4,0) node [black,midway,yshift=-11pt] 
	{$q$};

\end{scope}

%(j) family

\begin{scope}[shift={(18,-12.)}]
\node[bullet] at (0,0){};
\node[bullet] at (1,0){};
\node[bullet] at (2,0){};
\node[bullet] at (1,-1){};
\node[bullet] at (1,-3){};
\node[bullet] at (1,-4){};
\node[bullet] at (1,-5){};

\node[below] at (-0.25,0){};
\node[above] at (0,0){$-2$};
\node[above] at (1,0){$-2$};
\node[above] at (2,0){$-6$};
\node[right] at (1,-1){$-2$};
\node[right] at (1,-3){$-2$};
\node[right] at (1,-4){$-2$};
\node[right] at (1,-5){\footnotesize{$-(q+4)$}};

\node at (-0.5,-5){(j)};

\node at (1,-1.8){$\vdots$};

\draw (0,0)--(2,0);
\draw (1,0)--(1,-1.50);
\draw (1,-2.5)--(1,-5);

	\draw [thick,decorate,decoration={brace,mirror,amplitude=5pt},xshift=-10pt,yshift=0pt]
	(1,-1) -- (1,-3) node [black,midway,xshift=-11pt] 
	{$q$};

\end{scope}

\end{tikzpicture}
\caption{$3$-legged resolution graphs admitting $\mathbb{Q}$HD smoothing $(p,q,r\geq0)$} 
\label{resolutiongraphs}
\end{figure}

In this article, we aim to interpret $\mathbb{Q}$HD smoothings in terms of Lefschetz fibrations. A $\mathbb{Q}$HD smoothing is a Stein filling of the link of a weighted homogeneous surface singularity with the Milnor fillable contact structure. As the existence of positive allowable Lefschetz fibration(PALF for short) on a Stein filling is well known in general ~\cite{MR1825664}~\cite{MR1835390}, an explicit monodromy description of the filling is of great interest. The simplest example of this is the famous lantern relation $abcd=xyz$ in the mapping class group of $4$-holed sphere, where each letter stands for right-handed Dehn twists of curves, as depicted in Figure~\ref{lantern}. Here, a Lefschetz fibration $X$ with monodromy $abcd$ is diffeomorphic to the minimal resolution of $A_{4,1}$ singularity whose link is diffeomorphic to a lens space $L(4,1)$ while a Lefschetz fibration $Y$ with monodromy $xyz$ is diffeomorphic to the $\mathbb{Q}$HD smoothing of the singularity~\cite{MR2566578}. Furthermore, the equality in the relation implies that the boundary of $X$ and $Y$ are diffeomorphic and the induced contact structures on the boundary are isotopic to each other, which is isotopic to the Milnor fillable contact structure.
\begin{figure}[h]
\begin{center}
\begin{tikzpicture}[scale=1.1]
\begin{scope}
\draw (3,-3) ellipse (2 and 1.2);
\draw (3,-3) circle (0.15);
\draw (4,-3) circle (0.15);
\draw (2,-3) circle (0.15);
\draw[red] (3,-3) circle (0.35);
\draw[red] (4,-3) circle (0.35);
\draw[red] (2,-3) circle (0.35);
\draw[red] (3,-3) ellipse (1.8 and 0.8);
\draw (2,-3.35) node[below] {$a$};
\draw (3,-3.35) node[below] {$b$};
\draw (4,-3.35) node[below] {$c$};
\draw (3.45,-4) node {$d$};
\end{scope}
\begin{scope}[shift={(5,0)}]
\draw (3,-3) ellipse (2 and 1.2);
\draw (3,-3) circle (0.15);
\draw (4.25,-3) circle (0.15);
\draw (1.75,-3) circle (0.15);
\draw[red] (1.5,-3) arc (180:360: 0.25);
\draw[red] (4,-3) arc (180:360: 0.25);
\draw[red] (4.5,-3) to [out=up, in=up] (1.5,-3);
\draw[red] (2,-3) to [out=up, in=up] (4,-3);
\draw[red] (3.675,-3) ellipse (1.1 and 0.4); 
\draw[red] (2.375,-3) ellipse (1.1 and 0.4); 
\draw (1.75,-3.5) node {$x$};
\draw (3,-2) node {$z$};
\draw (4.25,-3.5) node {$y$};

\end{scope}
\end{tikzpicture}
\end{center}
\caption{Lantern relation}
\label{lantern}
\end{figure}
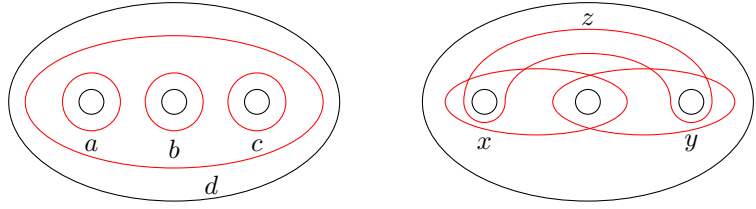
Therefore, asking whether other realtions that  describe $\mathbb{Q}$HD smoothings exist is natural. In ~\cite{MR2783383}, relations in the mapping class group of planar surfaces corresponding to $\mathbb{Q}$HD smoothings of  $A_{p^2,pq-1}$ and weighted homogeneous surface singularities with resolution graphs belonging to (a), (b) and (c) families in Figure~\ref{resolutiongraphs} were given by Endo-Mark-Van Horn-Morris. In the resolution graphs depicted in Figure~\ref{resolutiongraphs}, each vertex $v$ corresponds to an irreducible component $E_v$ of the exceptional divisor $E$, which is topologically $2$-sphere, and each edge corresponds to an intersection between the irreducible components $E_v$. We denote the number of edges connected to a vertex $v$ as the \emph{valence} of $v$, and the self-intersection of $E_v$ as the \emph{degree} of $v$. If the absolute value of the degree of $v$ is strictly less than the valence of $v$, we call the vertex $v$ a \emph{bad vertex}. Note that central vertices in (a), (b), and (c) families in Figure~\ref{resolutiongraphs} are not bad vertices, while central vertices in other families are bad. In this article, we construct genus-$1$ Lefschetz fibrations on $\mathbb{Q}$HD smoothings containing bad central vertices in their resolution graphs.
\begin{thm}
For each resolution graph $\Gamma$ in  Figure~\ref{resolutiongraphs} with bad central vertex, there is a relation $W_{\Gamma}=W_{\Gamma}'$ between words of right-handed Dehn twists in mapping class group of a genus-$1$ surface with boundaries such that Lefschetz fibration $X_{\Gamma}$ with monodromy $W_{\Gamma}$ is diffeomorphic to the minimal resolution of corresponding singularity $S_{\Gamma}$ and Lefschetz fibration $Y_{\Gamma}$ with monodromy $W_{\Gamma}'$ is a rational homology ball.
\label{Mainthm}
\end{thm}
To prove Theorem~\ref{Mainthm}, we proceed as follows: For each resolution graph $\Gamma$ in Figure~\ref{resolutiongraphs} with a bad central vertex, we construct a genus-$1$ PALF $X_{\Gamma}$ on the minimal resolution of the singularity $S_{\Gamma}$ corresponding to $\Gamma$ and verify wheter the induced contact structure on the boundary is the Milnor fillable contact structure by computing the first Chern class.
Then, starting from the global monodromy $W_{\Gamma}$ of the $X_{\Gamma}$, we get another positive word $W_{\Gamma}'=W_{\Gamma}$ of right-handed Dehn twists by monodromy substitutions after introducing appropriate canceling pairs so that PALF $Y_{\Gamma}$ with global monodromy $W_{\Gamma}'$ is rationally homology ball filling of the link of $S_{\Gamma}$. 
\begin{remark}
From the Lefschetz fibration $Y_{\Gamma}$ we constructed, we obtain a rational homology ball filling of the link of $S_{\Gamma}$. Hence, one may ask whether the total space of $Y_{\Gamma}$ is symplectic deformation equivalent or diffeomorphic to a $\mathbb{Q}$HD smoothing of $S_{\Gamma}$, which is given by complement of the compactifying divisor $K_{\Gamma}$ in a rational surface. By analyzing the method of constructing $\mathbb{Q}$HD smoothings, M. Bhupal and A. Stipsicz demonstrated that if $\Gamma$ is one of the resolution graphs in Figure~\ref{resolutiongraphs} (a), (b), (c), (d), (e), (f) or (g), then rational homology ball filling of the link of $S_{\Gamma}$ with the Milnor fillable contact structure is symplectically unique \cite[Theorem 1.1]{MR3203575}. 
We expect this result to be valid for the (h),(i) and (j) families. However, the uniqueness of the symplectic deformation or diffeomorphism type of rational homology ball filling is unknwon for those families.
\end{remark}
\subsection*{Acknowledgements}
The author thanks Kyungbae Park, Jongil Park and Ki-Heon Yun for their interests and valuable comments. The author
is supported by a KIAS Individual Grant (MG071002) at Korea Institute for Advanced Study

\section{Monodromy relations}
As the first step of the proof of Theorem~\ref{Mainthm}, we construct genus-$1$ Lefschetz fibrations on the minimal resolutions. If there is no bad vertex in the resolution graph $\Gamma$, there is well-known genus-$0$ PALF of Gay-Mark on the minimal resolution~\cite{MR3100798} (See also ~\cite{MR2272094}): We consider the $2$-sphere $\Sigma_{i}$ with $b_i$ holes for each vertex $v_i$ with degree $-b_i$. Then, the generic fiber $\Sigma$ is obtained by gluing $\Sigma_{i}$ along their boundaries according to $\Gamma$, and the global monodromy is given by the product of right-handed Dehn twists on curves parallel to the boundary of each $\Sigma_{i}$. We end up with only one right-handed Dehn twist on the connecting neck.  
For the resolution graphs in Figure~\ref{resolutiongraphs} with bad central vertex, we construct PALFs on the minimal resolutions by introducing a genus on the generic fibers, as in ~\cite{MR4125685}.  
\begin{figure} [h]
\begin{tikzpicture}[scale=0.7]
\draw (0,0) node[below]{$-2$};
\draw (1,0) node[below]{$-b_1$};
\draw (-1,0) node[below]{$-a_i$};
\draw (0,1) node[right]{$-c_1$};

\draw (-3.5,0) node[below]{$-a_1$};
\draw (3.5,0) node[below]{$-b_j$};
\draw (0,3.5) node[right]{$-c_k$};

\draw (-3.5,0)--(-3,0) (-1.5,0)--(1.5,0) (3,0)--(3.5,0) (0,0)--(0,1.5);
\draw (0,3)--(0,3.5);
\draw (0,0 ) node[circle, fill, inner sep=1.5pt, black]{};
\draw (1,0 ) node[circle, fill, inner sep=1.5pt, black]{};
\draw (-1,0 ) node[circle, fill, inner sep=1.5pt, black]{};
\draw (2.25,0) node{$\cdots$};
\draw (-2.25,0) node{$\cdots$};
\draw (0,2.25) node{$\vdots$};

\draw (0,1 ) node[circle, fill, inner sep=1.5pt, black]{};
\draw (0,3.5 ) node[circle, fill, inner sep=1.5pt, black]{};

\draw (3.5,0 ) node[circle, fill, inner sep=1.5pt, black]{};
\draw (-3.5,0 ) node[circle, fill, inner sep=1.5pt, black]{};

\end{tikzpicture}
\caption{$3$-legged plumbing graph $\Gamma$ with bad central vertex}
\label{resolutiongraph}
\end{figure}
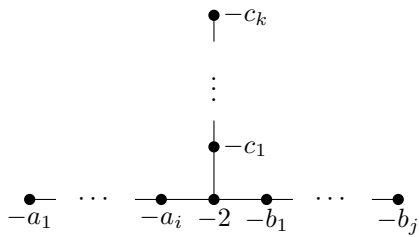
First, we construct genus-$0$ Lefschetz fibrations on horizontal and vertical part of a plumbing graph $\Gamma$ given in Figure~\ref{resolutiongraph} as illsutrated in Figure~\ref{horizontal} and Figure~\ref{vertical}: Let $\Sigma_1$ and $\Sigma_2$ be the generic fibers for horizontal and vertical parts, respectively. We denote a simple closed curve in $\Sigma_1$ enclosing $i^{\text{th}}$ hole by $\alpha_{i}$ and a simple closed curve in $\Sigma_1$ enclosing all holes from the first to $i^{\text{th}}$ hole by $\gamma_{i}$. Further, we denote a simple closed curve in $\Sigma_2$ enclosing $i^{\text{th}}$ hole by $\beta_i$ and a simple closed curve in $\Sigma_2$ enclosing all holes from the first to $i^{\text{th}}$ hole by $\delta_{i}$. The global monodromy of horizontal part can be written as $W_a\gamma_{N}^2W_b$, where $W_a$ is a word of right-handed Dehn twists along curves from degree $-a_n$ vertices, $W_b$ is a word of right-handed Dehn twists along curves from degree $-b_m$ vertices, and $N=(a_1+\cdots+a_i)-2i+1$. Similarly, the global monodromy corresponding to the vertical part can be written as $\beta_1\cdots\beta_{c_1-1}\delta_{c_1-1}W_c$, where $W_c$ is a word of right-handed Dehn twists along curves from degree $-c_l$ vertices with $l=2,\dots, k$. 
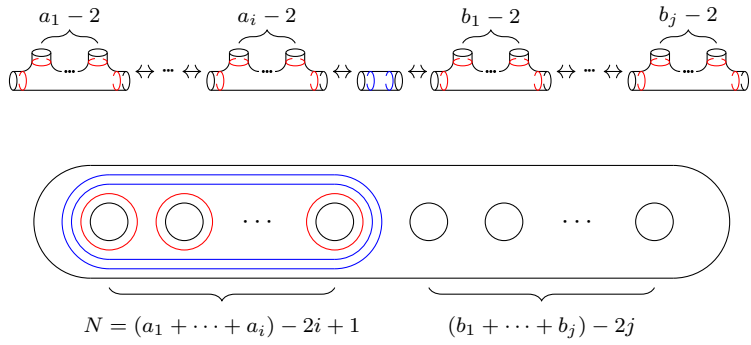
\begin{figure}
\begin{tikzpicture}[xscale=0.25, yscale=0.25]

\draw [decorate,decoration={brace,amplitude=5pt},xshift=0pt,yshift=0]
	(4.4,2.5) -- (7.4,2.5) node [black,midway,yshift=10pt] 
	{\footnotesize $a_1-2$};

%\draw[rounded corners] (-0.5,0)--(1.5,0) (-0.5,1)--(1.5,1) ;
%\draw[<->](1.8,0.5)--(2.6,0.5);
\draw[rounded corners] (2.9,0)--(8.9,0) (2.9,1)--(3.9,1)--(3.9,2) (4.9,2)--(4.9,1)--(5.4,1) (6.4,1)--(6.9,1)--(6.9,2) (7.9,2)--(7.9,1)--(8.9,1);

\draw (5.7,1) node[circle, fill, inner sep=.4pt, black]{};
\draw (5.9,1) node[circle, fill, inner sep=.4pt, black]{};
\draw (6.1,1) node[circle, fill, inner sep=.4pt, black]{};

\draw (2.9,0) arc (-90:-270:0.2 and 0.5);
\draw (2.9,0) arc (-90:90:0.2 and 0.5);

\draw[red] (3.4,0) arc (-90:90:0.2 and 0.5);
\draw[dashed, red] (3.4,0) arc (-90:-270:0.2 and 0.5);

%\draw (5.9,0) arc (-90:90:0.2 and 0.5);
%\draw[dashed] (5.9,0) arc (-90:-270:0.2 and 0.5);

%\draw[red] (5.4,0) arc (-90:90:0.2 and 0.5);
%\draw[dashed, red] (5.4,0) arc (-90:-270:0.2 and 0.5);

%\draw (7.3,0) arc (-90:-270:0.2 and 0.5);
%\draw (7.3,0) arc (-90:90:0.2 and 0.5);

%\draw[red] (7.8,0) arc (-90:90:0.2 and 0.5);
%\draw[dashed, red] (7.8,0) arc (-90:-270:0.2 and 0.5);

\draw (8.9,0) arc (-90:90:0.2 and 0.5);
\draw[dashed] (8.9,0) arc (-90:-270:0.2 and 0.5);

\draw[red] (8.4,0) arc (-90:90:0.2 and 0.5);
\draw[dashed, red] (8.4,0) arc (-90:-270:0.2 and 0.5);

\draw (3.9,2) arc (180:360:0.5 and 0.2);
\draw (3.9,2) arc (-180:-360:0.5 and 0.2);

\draw[red] (3.9,1.5) arc (180:360:0.5 and 0.2);
\draw[red,dashed] (3.9,1.5) arc (-180:-360:0.5 and 0.2);

\draw (6.9,2) arc (180:360:0.5 and 0.2);
\draw (6.9,2) arc (-180:-360:0.5 and 0.2);

\draw[red] (6.9,1.5) arc (180:360:0.5 and 0.2);
\draw[red,dashed] (6.9,1.5) arc (-180:-360:0.5 and 0.2);

\draw[<->] (9.4,1)--(10.4,1);
\draw (11.15,1) node {$\cdot$};
\draw (11.35,1) node {$\cdot$};
\draw (10.95,1) node {$\cdot$};

\draw[<->] (11.9,1)--(12.9,1);

\begin{scope}[shift={(10.5,0)}]

\draw [decorate,decoration={brace,amplitude=5pt},xshift=0pt,yshift=0]
	(4.4,2.5) -- (7.4,2.5) node [black,midway,yshift=10pt] 
	{\footnotesize $a_i-2$};

\draw[rounded corners] (2.9,0)--(8.9,0) (2.9,1)--(3.9,1)--(3.9,2) (4.9,2)--(4.9,1)--(5.4,1) (6.4,1)--(6.9,1)--(6.9,2) (7.9,2)--(7.9,1)--(8.9,1);

\draw (5.7,1) node[circle, fill, inner sep=.4pt, black]{};
\draw (5.9,1) node[circle, fill, inner sep=.4pt, black]{};
\draw (6.1,1) node[circle, fill, inner sep=.4pt, black]{};

\draw (2.9,0) arc (-90:-270:0.2 and 0.5);
\draw (2.9,0) arc (-90:90:0.2 and 0.5);

\draw[red] (3.4,0) arc (-90:90:0.2 and 0.5);
\draw[dashed, red] (3.4,0) arc (-90:-270:0.2 and 0.5);

%\draw (5.9,0) arc (-90:90:0.2 and 0.5);
%\draw[dashed] (5.9,0) arc (-90:-270:0.2 and 0.5);

%\draw[red] (5.4,0) arc (-90:90:0.2 and 0.5);
%\draw[dashed, red] (5.4,0) arc (-90:-270:0.2 and 0.5);

%\draw (7.3,0) arc (-90:-270:0.2 and 0.5);
%\draw (7.3,0) arc (-90:90:0.2 and 0.5);

%\draw[red] (7.8,0) arc (-90:90:0.2 and 0.5);
%\draw[dashed, red] (7.8,0) arc (-90:-270:0.2 and 0.5);

\draw (8.9,0) arc (-90:90:0.2 and 0.5);
\draw[dashed] (8.9,0) arc (-90:-270:0.2 and 0.5);

\draw[red] (8.4,0) arc (-90:90:0.2 and 0.5);
\draw[dashed, red] (8.4,0) arc (-90:-270:0.2 and 0.5);

\draw (3.9,2) arc (180:360:0.5 and 0.2);
\draw (3.9,2) arc (-180:-360:0.5 and 0.2);

\draw[red] (3.9,1.5) arc (180:360:0.5 and 0.2);
\draw[red,dashed] (3.9,1.5) arc (-180:-360:0.5 and 0.2);

\draw (6.9,2) arc (180:360:0.5 and 0.2);
\draw (6.9,2) arc (-180:-360:0.5 and 0.2);

\draw[red] (6.9,1.5) arc (180:360:0.5 and 0.2);
\draw[red,dashed] (6.9,1.5) arc (-180:-360:0.5 and 0.2);
\draw[<->] (9.4,1)--(10.4,1);
\end{scope}

%-2 연결
\begin{scope}[shift={(10.5+11.4,0)}]
\draw[rounded corners] (-0.5,0)--(1.5,0) (-0.5,1)--(1.5,1) ;
\draw (-0.50,0) arc (-90:90:0.2 and 0.5);
\draw (-0.50,0) arc (-90:-270:0.2 and 0.5);

\draw[blue] (0,0) arc (-90:90:0.2 and 0.5);
\draw[dashed, blue] (0,0) arc (-90:-270:0.2 and 0.5);

\draw (1.50,0) arc (-90:90:0.2 and 0.5);
\draw (1.50,0) arc (-90:-270:0.2 and 0.5);

\draw[blue] (1,0) arc (-90:90:0.2 and 0.5);
\draw[dashed, blue] (1,0) arc (-90:-270:0.2 and 0.5);
\draw[<->] (2,1)--(3,1);

\end{scope}

\begin{scope}[shift={(21.9+.5,0)}]

\draw [decorate,decoration={brace,amplitude=5pt},xshift=0pt,yshift=0]
	(4.4,2.5) -- (7.4,2.5) node [black,midway,yshift=10pt] 
	{\footnotesize $b_1-2$};

\draw[rounded corners] (2.9,0)--(8.9,0) (2.9,1)--(3.9,1)--(3.9,2) (4.9,2)--(4.9,1)--(5.4,1) (6.4,1)--(6.9,1)--(6.9,2) (7.9,2)--(7.9,1)--(8.9,1);

\draw (5.7,1) node[circle, fill, inner sep=.4pt, black]{};
\draw (5.9,1) node[circle, fill, inner sep=.4pt, black]{};
\draw (6.1,1) node[circle, fill, inner sep=.4pt, black]{};

\draw (2.9,0) arc (-90:-270:0.2 and 0.5);
\draw (2.9,0) arc (-90:90:0.2 and 0.5);

\draw[red] (3.4,0) arc (-90:90:0.2 and 0.5);
\draw[dashed, red] (3.4,0) arc (-90:-270:0.2 and 0.5);

%\draw (5.9,0) arc (-90:90:0.2 and 0.5);
%\draw[dashed] (5.9,0) arc (-90:-270:0.2 and 0.5);

%\draw[red] (5.4,0) arc (-90:90:0.2 and 0.5);
%\draw[dashed, red] (5.4,0) arc (-90:-270:0.2 and 0.5);

%\draw (7.3,0) arc (-90:-270:0.2 and 0.5);
%\draw (7.3,0) arc (-90:90:0.2 and 0.5);

%\draw[red] (7.8,0) arc (-90:90:0.2 and 0.5);
%\draw[dashed, red] (7.8,0) arc (-90:-270:0.2 and 0.5);

\draw (8.9,0) arc (-90:90:0.2 and 0.5);
\draw[dashed] (8.9,0) arc (-90:-270:0.2 and 0.5);

\draw[red] (8.4,0) arc (-90:90:0.2 and 0.5);
\draw[dashed, red] (8.4,0) arc (-90:-270:0.2 and 0.5);

\draw (3.9,2) arc (180:360:0.5 and 0.2);
\draw (3.9,2) arc (-180:-360:0.5 and 0.2);

\draw[red] (3.9,1.5) arc (180:360:0.5 and 0.2);
\draw[red,dashed] (3.9,1.5) arc (-180:-360:0.5 and 0.2);

\draw (6.9,2) arc (180:360:0.5 and 0.2);
\draw (6.9,2) arc (-180:-360:0.5 and 0.2);

\draw[red] (6.9,1.5) arc (180:360:0.5 and 0.2);
\draw[red,dashed] (6.9,1.5) arc (-180:-360:0.5 and 0.2);

\draw[<->] (9.4,1)--(10.4,1);
\draw (11.15,1) node {$\cdot$};
\draw (11.35,1) node {$\cdot$};
\draw (10.95,1) node {$\cdot$};

\draw[<->] (11.9,1)--(12.9,1);
\end{scope}

\begin{scope}[shift={(22.4+10.5,0)}]

\draw [decorate,decoration={brace,amplitude=5pt},xshift=0pt,yshift=0]
	(4.4,2.5) -- (7.4,2.5) node [black,midway,yshift=10pt] 
	{\footnotesize $b_j-2$};

\draw[rounded corners] (2.9,0)--(8.9,0) (2.9,1)--(3.9,1)--(3.9,2) (4.9,2)--(4.9,1)--(5.4,1) (6.4,1)--(6.9,1)--(6.9,2) (7.9,2)--(7.9,1)--(8.9,1);

\draw (5.7,1) node[circle, fill, inner sep=.4pt, black]{};
\draw (5.9,1) node[circle, fill, inner sep=.4pt, black]{};
\draw (6.1,1) node[circle, fill, inner sep=.4pt, black]{};

\draw (2.9,0) arc (-90:-270:0.2 and 0.5);
\draw (2.9,0) arc (-90:90:0.2 and 0.5);

\draw[red] (3.4,0) arc (-90:90:0.2 and 0.5);
\draw[dashed, red] (3.4,0) arc (-90:-270:0.2 and 0.5);

%\draw (5.9,0) arc (-90:90:0.2 and 0.5);
%\draw[dashed] (5.9,0) arc (-90:-270:0.2 and 0.5);

%\draw[red] (5.4,0) arc (-90:90:0.2 and 0.5);
%\draw[dashed, red] (5.4,0) arc (-90:-270:0.2 and 0.5);

%\draw (7.3,0) arc (-90:-270:0.2 and 0.5);
%\draw (7.3,0) arc (-90:90:0.2 and 0.5);

%\draw[red] (7.8,0) arc (-90:90:0.2 and 0.5);
%\draw[dashed, red] (7.8,0) arc (-90:-270:0.2 and 0.5);

\draw (8.9,0) arc (-90:90:0.2 and 0.5);
\draw[dashed] (8.9,0) arc (-90:-270:0.2 and 0.5);

\draw[red] (8.4,0) arc (-90:90:0.2 and 0.5);
\draw[dashed, red] (8.4,0) arc (-90:-270:0.2 and 0.5);

\draw (3.9,2) arc (180:360:0.5 and 0.2);
\draw (3.9,2) arc (-180:-360:0.5 and 0.2);

\draw[red] (3.9,1.5) arc (180:360:0.5 and 0.2);
\draw[red,dashed] (3.9,1.5) arc (-180:-360:0.5 and 0.2);

\draw (6.9,2) arc (180:360:0.5 and 0.2);
\draw (6.9,2) arc (-180:-360:0.5 and 0.2);

\draw[red] (6.9,1.5) arc (180:360:0.5 and 0.2);
\draw[red,dashed] (6.9,1.5) arc (-180:-360:0.5 and 0.2);
\end{scope}

\begin{scope}[shift={(0,-4)}]
\draw (7,0)--(38,0);
\draw (7,-6)--(38,-6);
\draw (7,0) arc (90:270:3);
\draw (38,0) arc (90:-90:3);
\draw (8,-3) circle (1);
\draw[red] (8,-3) circle (1.5);
\draw (12,-3) circle (1);
\draw[red] (12,-3) circle (1.5);

\draw (16,-3) node{$\cdots$};
\draw (20,-3) circle (1);
\draw[red] (20,-3) circle (1.5);

\draw (25,-3) circle (1);
\draw (29,-3) circle (1);
\draw (33,-3) node{$\cdots$};
\draw (37,-3) circle (1);

\draw[blue] (20,-1)--(8,-1);
\draw[blue] (20,-5)--(8,-5);
\draw[blue] (8,-1) arc (90:270:2);
\draw[blue] (20,-1) arc (90:-90:2);
\draw[blue] (20,-.5)--(8,-.5);
\draw[blue] (20,-5.5)--(8,-5.5);
\draw[blue] (8,-.5) arc (90:270:2.5);
\draw[blue] (20,-.5) arc (90:-90:2.5);
\draw [decorate,decoration={brace,mirror,amplitude=5pt},xshift=0pt,yshift=0]
	(8,-6.5) -- (20,-6.5) node [black,midway,yshift=-15pt] 
	{\footnotesize $N=(a_1+\cdots+a_i)-2i+1$};
\draw [decorate,decoration={brace,mirror,amplitude=5pt},xshift=0pt,yshift=0]
	(25,-6.5) -- (37,-6.5) node [black,midway,yshift=-15pt] 
	{\footnotesize $(b_1+\cdots+b_j)-2j$};

\end{scope}

\end{tikzpicture}
\caption{Global monodromy: $W_a \textcolor{blue}{\gamma_N^2} W_b$
}
\label{horizontal}
\end{figure}

\begin{figure}
\begin{tikzpicture}[xscale=0.25, yscale=0.25]

\begin{scope}[rotate=90]
\draw [decorate,decoration={brace,amplitude=5pt},xshift=0pt,yshift=0]
	(4.4,2.5) -- (7.4,2.5) node [black,midway,xshift=-20pt] 
	{\footnotesize $c_1-2$};

%\draw[rounded corners] (-0.5,0)--(1.5,0) (-0.5,1)--(1.5,1) ;
%\draw[<->](1.8,0.5)--(2.6,0.5);
\draw[rounded corners] (2.9,0)--(8.9,0) (2.9,1)--(3.9,1)--(3.9,2) (4.9,2)--(4.9,1)--(5.4,1) (6.4,1)--(6.9,1)--(6.9,2) (7.9,2)--(7.9,1)--(8.9,1);

\draw (5.7,1) node[circle, fill, inner sep=.4pt, black]{};
\draw (5.9,1) node[circle, fill, inner sep=.4pt, black]{};
\draw (6.1,1) node[circle, fill, inner sep=.4pt, black]{};

\draw (2.9,0) arc (-90:-270:0.2 and 0.5);
\draw (2.9,0) arc (-90:90:0.2 and 0.5);

\draw[red] (3.4,0) arc (-90:90:0.2 and 0.5);
\draw[dashed, red] (3.4,0) arc (-90:-270:0.2 and 0.5);

%\draw (5.9,0) arc (-90:90:0.2 and 0.5);
%\draw[dashed] (5.9,0) arc (-90:-270:0.2 and 0.5);

%\draw[red] (5.4,0) arc (-90:90:0.2 and 0.5);
%\draw[dashed, red] (5.4,0) arc (-90:-270:0.2 and 0.5);

%\draw (7.3,0) arc (-90:-270:0.2 and 0.5);
%\draw (7.3,0) arc (-90:90:0.2 and 0.5);

%\draw[red] (7.8,0) arc (-90:90:0.2 and 0.5);
%\draw[dashed, red] (7.8,0) arc (-90:-270:0.2 and 0.5);

\draw (8.9,0) arc (-90:90:0.2 and 0.5);
\draw[dashed] (8.9,0) arc (-90:-270:0.2 and 0.5);

\draw[red] (8.4,0) arc (-90:90:0.2 and 0.5);
\draw[dashed, red] (8.4,0) arc (-90:-270:0.2 and 0.5);

\draw (3.9,2) arc (180:360:0.5 and 0.2);
\draw (3.9,2) arc (-180:-360:0.5 and 0.2);

\draw[red] (3.9,1.5) arc (180:360:0.5 and 0.2);
\draw[red,dashed] (3.9,1.5) arc (-180:-360:0.5 and 0.2);

\draw (6.9,2) arc (180:360:0.5 and 0.2);
\draw (6.9,2) arc (-180:-360:0.5 and 0.2);

\draw[red] (6.9,1.5) arc (180:360:0.5 and 0.2);
\draw[red,dashed] (6.9,1.5) arc (-180:-360:0.5 and 0.2);

\draw[<->] (9.4,1)--(10.4,1);
\draw (11.15,1) node {$\cdot$};
\draw (11.35,1) node {$\cdot$};
\draw (10.95,1) node {$\cdot$};

\draw[<->] (11.9,1)--(12.9,1);

\begin{scope}[shift={(10.5,0)}]

\draw [decorate,decoration={brace,amplitude=5pt},xshift=0pt,yshift=0]
	(4.4,2.5) -- (7.4,2.5) node [black,midway,xshift=-20pt] 
	{\footnotesize $c_k-2$};

\draw[rounded corners] (2.9,0)--(8.9,0) (2.9,1)--(3.9,1)--(3.9,2) (4.9,2)--(4.9,1)--(5.4,1) (6.4,1)--(6.9,1)--(6.9,2) (7.9,2)--(7.9,1)--(8.9,1);

\draw (5.7,1) node[circle, fill, inner sep=.4pt, black]{};
\draw (5.9,1) node[circle, fill, inner sep=.4pt, black]{};
\draw (6.1,1) node[circle, fill, inner sep=.4pt, black]{};

\draw (2.9,0) arc (-90:-270:0.2 and 0.5);
\draw (2.9,0) arc (-90:90:0.2 and 0.5);

\draw[red] (3.4,0) arc (-90:90:0.2 and 0.5);
\draw[dashed, red] (3.4,0) arc (-90:-270:0.2 and 0.5);

%\draw (5.9,0) arc (-90:90:0.2 and 0.5);
%\draw[dashed] (5.9,0) arc (-90:-270:0.2 and 0.5);

%\draw[red] (5.4,0) arc (-90:90:0.2 and 0.5);
%\draw[dashed, red] (5.4,0) arc (-90:-270:0.2 and 0.5);

%\draw (7.3,0) arc (-90:-270:0.2 and 0.5);
%\draw (7.3,0) arc (-90:90:0.2 and 0.5);

%\draw[red] (7.8,0) arc (-90:90:0.2 and 0.5);
%\draw[dashed, red] (7.8,0) arc (-90:-270:0.2 and 0.5);

\draw (8.9,0) arc (-90:90:0.2 and 0.5);
\draw[dashed] (8.9,0) arc (-90:-270:0.2 and 0.5);

\draw[red] (8.4,0) arc (-90:90:0.2 and 0.5);
\draw[dashed, red] (8.4,0) arc (-90:-270:0.2 and 0.5);

\draw (3.9,2) arc (180:360:0.5 and 0.2);
\draw (3.9,2) arc (-180:-360:0.5 and 0.2);

\draw[red] (3.9,1.5) arc (180:360:0.5 and 0.2);
\draw[red,dashed] (3.9,1.5) arc (-180:-360:0.5 and 0.2);

\draw (6.9,2) arc (180:360:0.5 and 0.2);
\draw (6.9,2) arc (-180:-360:0.5 and 0.2);

\draw[red] (6.9,1.5) arc (180:360:0.5 and 0.2);
\draw[red,dashed] (6.9,1.5) arc (-180:-360:0.5 and 0.2);
\end{scope}
\begin{scope}[shift={(0,-8)}]
\draw (5-0.5,0)--(18,0);
\draw (5-0.5,-6)--(18,-6);
\draw (5-0.5,0) arc (90:270:3);
\draw (18-0.5,0) arc (90:-90:3);
\draw (5-0.5,-3) circle (1);
\draw[red] (5-0.5,-3) circle (1.5);
\draw (10-0.5,-3) circle (1);
\draw[red] (10-0.5,-3) circle (1.5);
\draw (14-0.5,-3) circle (1);
%\draw[red] (14,-3) circle (1.5);

\draw (19-0.5,-3) circle (1);
%\draw[red] (19,-3) circle (1.5);

\draw (17.-0.5,-3) node[circle, fill, inner sep=0.5pt, black]{};
\draw (16.5-0.5,-3) node[circle, fill, inner sep=0.5pt, black]{};
\draw (16.-0.5,-3) node[circle, fill, inner sep=0.5pt, black]{};

\draw (8.-0.5,-3) node[circle, fill, inner sep=0.5pt, black]{};
\draw (7.5-0.5,-3) node[circle, fill, inner sep=0.5pt, black]{};
\draw (7.-0.5,-3) node[circle, fill, inner sep=0.5pt, black]{};

\draw[red] (5-0.5,-1)--(10-0.5,-1);
\draw[red] (5-0.5,-5)--(10-0.5,-5);
\draw[red] (5-0.5,-1) arc (90:270:2);
\draw[red] (10-0.5,-1) arc (90:-90:2);
\draw [decorate,decoration={brace,mirror,amplitude=5pt},xshift=0pt,yshift=0]
	(5-0.5,-6.5) -- (10-0.5,-6.5) node [black,midway,xshift=25pt] 
	{\footnotesize $c_1-1$};
\draw [decorate,decoration={brace,mirror,amplitude=5pt},xshift=0pt,yshift=0]
	(14-0.5,-6.5) -- (19-0.5,-6.5) node [black,midway,xshift=40pt,align=center] 
	{\footnotesize $(c_2+\cdots+c_k)$\\
	\footnotesize$-2(k-1)$};

\end{scope}
\end{scope}

\end{tikzpicture}
\caption{Global monodromy: $\beta_1\cdots \beta_{c_1-1} \delta_{c_1-1}W_c$
}
\label{vertical}
\end{figure}
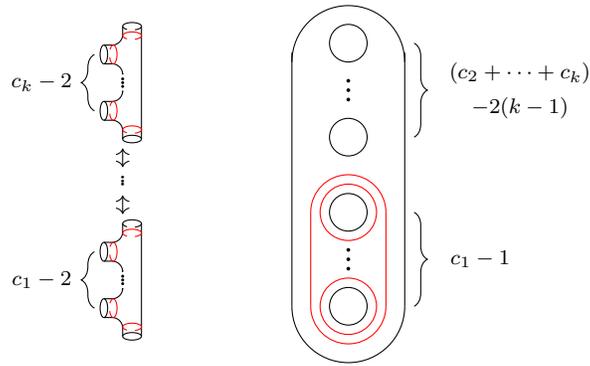

Now, we consider a genus-$1$ surface $\Sigma_\Gamma$ obtained from $\Sigma_1$ by attaching $b_1(\Sigma_2)$ $1$-handles as in Figure~\ref{genericfiber}. Then, we can naturally consider simple closed curves in $\Sigma_i$ as simple closed curves in $\Sigma_\Gamma$.
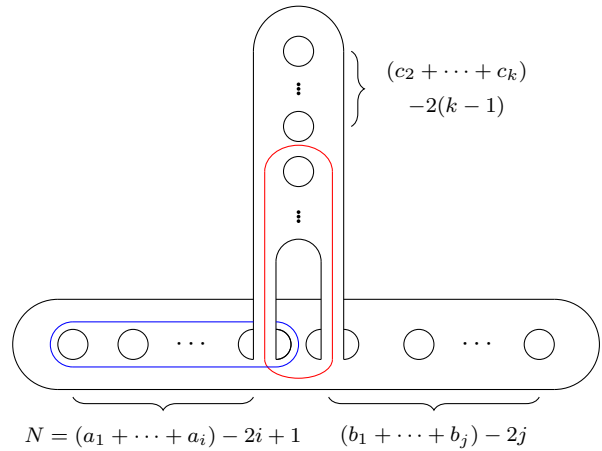
\begin{figure}
\begin{tikzpicture}[scale=0.2]

\draw (7,0)--(20,0) (21.5,0)--(24.5,0) (26,0)--(40,0);
\draw (7,-6)--(40,-6);
\draw (7,0) arc (90:270:3);
\draw (38+2,0) arc (90:-90:3);
\draw (8,-3) circle (1);
%\draw[red] (8,-3) circle (1.5);
\draw (12,-3) circle (1);
%\draw[red] (12,-3) circle (1.5);

\draw (16,-3) node{$\cdots$};
\draw (20,-2) arc (90:270:1);
\draw (21.5,-2) arc (90:-90:1);

%\draw[red] (20,-1.5) arc (1.5);
\draw (21.5,-2) arc (90:-90:1);
%vertical part
\draw (20,-4)--(20,16.5);
\draw (21.5,-4)--(21.5,2.5);
\draw[red] (20.75,-4)--(20.75,8.5);
\draw[red] (20.75,8.5) arc (180:0:2.25 and 1.75);
%\draw[red] (20.75,-4)--(23,-5.4)--(25.25,-4);
\draw[red] (20.75,-4) arc (180:360:2.25 and 1.25);

\draw (21.5,2.5) arc (180:0:1.5);
\draw (20,16.5) arc (180:0:3);

\draw (23, 8.5) circle (1);
\draw (23, 11.5) circle (1);
\draw (23, 16.5) circle (1);

\draw (23, 5.5) node[circle,fill, inner sep=.5pt,black]{};
\draw (23, 5.2) node[circle,fill, inner sep=.5pt,black]{};
\draw (23, 5.8) node[circle,fill, inner sep=.5pt,black]{};

\draw (23, 13.7) node[circle,fill, inner sep=.5pt,black]{};
\draw (23, 14) node[circle,fill, inner sep=.5pt,black]{};
\draw (23, 14.3) node[circle,fill, inner sep=.5pt,black]{};

\draw (24+2,-2) arc (90:-90:1);
\draw (24.5,-2) arc (90:270:1);

\draw (26,-4)--(26,16.5);
\draw (24.5,-4)--(24.5,2.5);
\draw[red] (25.25,-4)--(25.25,8.5);
\draw [decorate,decoration={brace,mirror,amplitude=5pt},xshift=0pt,yshift=0]
	(26.5,11.5) -- (26.5,16.5) node [black,midway,xshift=40pt,align=center] 
	{\footnotesize $(c_2+\cdots+c_k)$\\
	\footnotesize$-2(k-1)$};

%vertical part 끝

\draw (29+2,-3) circle (1);
\draw (33+2,-3) node{$\cdots$};
\draw (37+2,-3) circle (1);

\draw[blue] (20,-1.5)--(8,-1.5);
\draw[blue] (21.5,-4.5)--(8,-4.5);
\draw[blue] (8,-1.5) arc (90:270:1.5);
%\draw[blue] (20,-1) arc (90:-90:2);
%\draw[blue] (20,-.5)--(8,-.5);
%\draw[blue] (20,-5.5)--(8,-5.5);

\draw[blue] (21.5,-1.5) arc (90:-90:1.5);
%\draw[blue] (20,-.5) arc (90:-90:2.5);
\draw [decorate,decoration={brace,mirror,amplitude=5pt},xshift=0pt,yshift=0]
	(8,-6.5) -- (20,-6.5) node [black,midway,yshift=-15pt] 
	{\footnotesize $N=(a_1+\cdots+a_i)-2i+1$};
\draw [decorate,decoration={brace,mirror,amplitude=5pt},xshift=0pt,yshift=0]
	(25,-6.5) -- (39,-6.5) node [black,midway,yshift=-15pt] 
	{\footnotesize $(b_1+\cdots+b_j)-2j$};

\end{tikzpicture}
\caption{Global monodromy: $\beta_1\cdots\beta_{c_1-1}W_a\textcolor{blue}{\gamma_N}\textcolor{red}{\delta_{c_1-1}}W_c\textcolor{blue}{\gamma_N}W_b$
}
\label{genericfiber}
\end{figure}
\begin{prop}
Let $X_\Gamma$ be a positive allowable Lefschetz fibration with generic fiber $\Sigma_\Gamma$, and global monodromy $\beta_1\cdots\beta_{c_1-1}W_a{\gamma_N}{\delta_{c_1-1}}W_c{\gamma_N}W_b$. Then, total space of $X_\Gamma$ is diffeomorphic to the plumbing of $2$-spheres according to $\Gamma$. Furthermore, the first Chern class of $X_\Gamma$ satisfies the adjunction equality for each vertex in $\Gamma$.
\label{prop1}
\end{prop}

\begin{proof}
We first verify that a genus-$0$ PALF $X_L$ of Gay-Mark on a linear plumbing $L$ (see Figure~\ref{linear}) is actually diffeomorphic to the plumbing of spheres. From the Lefschetz fibration structure of $X_L$, we obtain a Kirby diagram as in Figure~\ref{linear}: As the generic fiber of $W$ is $N=(a_1+\cdots+a_i)-2i+1$ holed disk, we have one $0$-handle, $N$ $1$-handles and a $2$-handle for each vanishing cycle as in Figure~\ref{linear}. Note that all the framings of $2$-handles are $-1$ with respect to the blackboard framing.
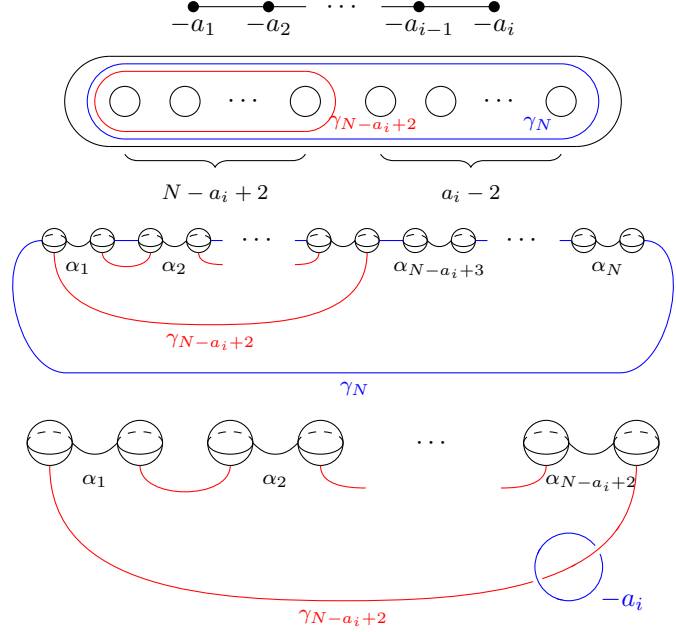
\begin{figure}[h]
 \begin{tikzpicture}[scale=1.]
\draw (0,0) node[below]{$-a_1$};

\draw (1,0) node[below]{$-a_2$};
\draw (3,0) node[below]{$-a_{i-1}$};
\draw (4,0) node[below]{$-a_i$};

\draw (0,0)--(1.5,0) (2.5,0)--(4,0);
\draw (0,0 ) node[circle, fill, inner sep=1.5pt, black]{};
\draw (1,0 ) node[circle, fill, inner sep=1.5pt, black]{};
\draw (2,0) node{$\cdots$};
\draw (3,0 ) node[circle, fill, inner sep=1.5pt, black]{};
\draw (4,0 ) node[circle, fill, inner sep=1.5pt, black]{};
\draw (0,-.5) node{};

\end{tikzpicture}

\begin{tikzpicture}[scale=0.20]
\draw (7,0)--(38,0);
\draw (7,-6)--(38,-6);
\draw (7,0) arc (90:270:3);
\draw (38,0) arc (90:-90:3);
\draw (8,-3) circle (1);
%\draw[red] (8,-3) circle (1.5);
\draw (12,-3) circle (1);
%\draw[red] (12,-3) circle (1.5);

\draw (16,-3) node{$\cdots$};
\draw (20,-3) circle (1);
%\draw[red] (20,-3) circle (1.5);

\draw (25,-3) circle (1);
\draw (29,-3) circle (1);
\draw (33,-3) node{$\cdots$};
\draw (37,-3) circle (1);

\draw[red] (20,-1)--(8,-1);
\draw[red] (20,-5)--(8,-5);
\draw[red] (8,-1) arc (90:270:2);
\draw[red] (20,-1) arc (90:-90:2);
\draw[red] (24.5,-4.5) node {\footnotesize{$\gamma_{N-a_i+2}$}};
\draw[blue] (37,-.5)--(8,-.5);
\draw[blue] (37,-5.5)--(8,-5.5);
\draw[blue] (8,-.5) arc (90:270:2.5);
\draw[blue] (37,-.5) arc (90:-90:2.5);
\draw[blue] (35.5,-4.5) node {\footnotesize{$\gamma_{N}$}};

\draw [decorate,decoration={brace,mirror,amplitude=5pt},xshift=0pt,yshift=0]
	(8,-6.5) -- (20,-6.5) node [black,midway,yshift=-15pt] 
	{\footnotesize $N-a_i+2$};
\draw [decorate,decoration={brace,mirror,amplitude=5pt},xshift=0pt,yshift=0]
	(25,-6.5) -- (37,-6.5) node [black,midway,yshift=-15pt] 
	{\footnotesize $a_i-2$};

\end{tikzpicture}

\begin{tikzpicture}[scale=0.16]
\begin{scope}
\draw (0,1.5) node[]{};
\draw (0,0) circle (1);
\draw (1,0) arc (0:-180:1 and 0.5);
\draw[dashed] (1,0) arc (0:180:1 and 0.5);
\draw (4,0) circle (1);
\draw (5,0) arc (0:-180:1 and 0.5);
\draw[dashed] (5,0) arc (0:180:1 and 0.5);
\draw (1,0) to [out=-45, in=left] (2,-.5) to [out=right,  in=225] (3,0);
\draw (2,-1) node[below]{\footnotesize{$\alpha_1$}};

\draw (8,0) circle (1);
\draw (9,0) arc (0:-180:1 and 0.5);
\draw[dashed] (9,0) arc (0:180:1 and 0.5);
\draw (12,0) circle (1);
\draw (13,0) arc (0:-180:1 and 0.5);
\draw[dashed] (13,0) arc (0:180:1 and 0.5);

\draw (9,0) to [out=-45, in=left] (10,-.5) to [out=right,  in=225] (11,0);
\draw (10,-1) node[below]{\footnotesize{$\alpha_2$}};

\draw (17,0) node[]{$\cdots$};

\draw (22,0) circle (1);
\draw (23,0) arc (0:-180:1 and 0.5);
\draw[dashed] (23,0) arc (0:180:1 and 0.5);
\draw (26,0) circle (1);
\draw (27,0) arc (0:-180:1 and 0.5);
\draw[dashed] (27,0) arc (0:180:1 and 0.5);
\draw (23,0) to [out=-45, in=left] (24,-.5) to [out=right,  in=225] (25,0);
%\draw (24,2) node[above]{\footnotesize{$\alpha_{N-a_i+2}$}};

\draw (30,0) circle (1);
\draw (31,0) arc (0:-180:1 and 0.5);
\draw[dashed] (31,0) arc (0:180:1 and 0.5);
\draw (34,0) circle (1);
\draw (35,0) arc (0:-180:1 and 0.5);
\draw[dashed] (35,0) arc (0:180:1 and 0.5);
\draw (31,0) to [out=-45, in=left] (32,-.5) to [out=right,  in=225] (33,0);
\draw (32,-1) node[below]{\footnotesize{$\alpha_{N-a_i+3}$}};

\draw (44,0) circle (1);
\draw (45,0) arc (0:-180:1 and 0.5);
\draw[dashed] (45,0) arc (0:180:1 and 0.5);
\draw (48,0) circle (1);
\draw (49,0) arc (0:-180:1 and 0.5);
\draw[dashed] (49,0) arc (0:180:1 and 0.5);
\draw (45,0) to [out=-45, in=left] (46,-.5) to [out=right,  in=225] (47,0);
\draw (46,-1) node[below]{\footnotesize{$\alpha_{N}$}};

\draw (39,0) node[]{$\cdots$};

\draw[red] (0,-1) to [out=down, in=left] (13,-7) to [out=right, in =down] (26,-1);
\draw[red] (13,-7) node[below]{\footnotesize{$\gamma_{N-a_i+2}$}};
\draw[red] (4,-1) to [out=down, in=down] (8,-1);
\draw[red] (12,-1) to [out=down, in=left] (14,-2);
\draw[red] (20,-2) to [out=right, in=down] (22,-1);

\draw[blue] (-1,0) to [out=left, in=left] (1,-11) to (25,-11) to (47,-11) to [out=right, in=right] (49,0);
\draw[blue] (25,-11) node[below]{\footnotesize{$\gamma_{N}$}};

\draw[blue] (5,0)--(7,0)  (13,0)--(14,0) (42,0)--(43,0);
\draw[blue] (20,0)--(21,0) (27,0)--(29,0) (35,0)--(36,0);

\end{scope}
\begin{scope}[shift={(-5.5,18)},scale=1.3]
\end{scope}

\begin{scope}[shift={(7,24)}, scale=8]

\end{scope}
\end{tikzpicture}

\begin{tikzpicture}[scale=0.3]
\draw (0,1.2) node[]{};
\draw (0,0) circle (1);
\draw (1,0) arc (0:-180:1 and 0.5);
\draw[dashed] (1,0) arc (0:180:1 and 0.5);
\draw (4,0) circle (1);
\draw (5,0) arc (0:-180:1 and 0.5);
\draw[dashed] (5,0) arc (0:180:1 and 0.5);
\draw (1,0) to [out=-45, in=left] (2,-.5) to [out=right,  in=225] (3,0);
\draw (2,-1) node[below]{\footnotesize{$\alpha_1$}};

\draw (8,0) circle (1);
\draw (9,0) arc (0:-180:1 and 0.5);
\draw[dashed] (9,0) arc (0:180:1 and 0.5);
\draw (12,0) circle (1);
\draw (13,0) arc (0:-180:1 and 0.5);
\draw[dashed] (13,0) arc (0:180:1 and 0.5);

\draw (9,0) to [out=-45, in=left] (10,-.5) to [out=right,  in=225] (11,0);
\draw (10,-1) node[below]{\footnotesize{$\alpha_2$}};

\draw (17,0) node[]{$\cdots$};

\draw (22,0) circle (1);
\draw (23,0) arc (0:-180:1 and 0.5);
\draw[dashed] (23,0) arc (0:180:1 and 0.5);
\draw (26,0) circle (1);
\draw (27,0) arc (0:-180:1 and 0.5);
\draw[dashed] (27,0) arc (0:180:1 and 0.5);
\draw (23,0) to [out=-45, in=left] (24,-.5) to [out=right,  in=225] (25,0);
\draw (24,-1) node[below]{\footnotesize{$\alpha_{N-a_i+2}$}};

\begin{knot}
\strand[red] (0,-1) to [out=down, in=left] (13,-7) to [out=right, in =down] (26,-1);
\strand[blue] (23,-5.5) circle (1.5);
\flipcrossings{2}
\draw[blue] (24.,-7) node[right]{$-a_i$};
\draw[red] (13,-7) node[below]{\footnotesize{$\gamma_{N-a_i+2}$}};
\draw[red] (4,-1) to [out=down, in=down] (8,-1);
\draw[red] (12,-1) to [out=down, in=left] (14,-2);
\draw[red] (20,-2) to [out=right, in=down] (22,-1);
\end{knot}

\end{tikzpicture}
\caption{Linear plumbing of $L$ and part of Kirby diagram of $W$}
\label{linear}
\end{figure}
First, we slide a $2$-handle corresponding to $\gamma_N$ over $2$-handles corresponding to $\alpha_n$ with $n=N-a_i+3, \dots, N$, and $\gamma_{N-a_i+2}$ to unlink from the $1$-handles. 
After cancelling $1$-handles with $\alpha_n$ $(n=N-a_i+3,\dots,N)$ $2$-handles, we obtain the last diagram in Figure~\ref{linear}. A $2$-handle represented by unknot in the last diagram in Figure~\ref{linear} corresponds to degree $-a_i$ vertex $v_i$ in $L$. Thus,  the homology class of $v_i$ can be represented by $\gamma_N-\gamma_{N-a_i+2}-\alpha_{N-a_i+3}-\cdots-\alpha_N$. We also verify that the first Chern class $c_1(X_L)$ satisfies the adjunction equality on $v_i$: The first Chern class $c_1(X_L)$ is represented by a co-cycle whose value on the $2$-handle corresponding to a vanishing cycle is the rotation number of the vanishing cycle~\cite{MR1668563}. And if we fix a trivialization of the tangent bundle of fiber as a natural extension of a trivialization of the tangent bundle of $\mathbb{R}^2$, the rotation numbers of all vanishing cycles of $X_L$ are $1$. Then a simple computation shows that the adjunction equality is satisfied on $v_i$. The aforementioned process can be repeated until a Kirby diagram of the plumbing of $2$-spheres according to $L$ is obtained. 

Subsequently, we verify that the total space of $X_\Gamma$ is diffeomorphic to the plumbing of $2$-spheres according to $\Gamma$. From the Lefschetz fibration structure of $X_\Gamma$, we get a Kirby diagram of $X_\Gamma$ as in Figure~\ref{minimalresolution}. The white and gray $1$-handles correspond to the horizontal and vertical parts of $\Gamma$, respectively, and all the framings of $2$-handles are $-1$ with respect to the blackboard framing. As the Kirby diagrams for the horizontal and vertical parts are embedded in the diagram for $X_\Gamma$, the plumbing of unknots is obtained with respect to the horizontal and vertical parts by sliding and canceling handles as described previously. 
 The linking of horizontal and vertical parts is derived from linkings between $2$-handles corresponding to $\beta_1$, $\delta_{c_1-1}$ and two $\gamma_N$. The homology class of each vertex $v$ is represented by same vanishing cycles derived from linear plumbings. Thus, the first Chern class of $X_\Gamma$ satisfies adjunction equality on $v$.
\begin{figure}[h]
\begin{tikzpicture}[scale=0.18]
\draw (0,0) circle (1);
\draw (1,0) arc (0:-180:1 and 0.5);
\draw[dashed] (1,0) arc (0:180:1 and 0.5);
\draw (4,0) circle (1);
\draw (5,0) arc (0:-180:1 and 0.5);
\draw[dashed] (5,0) arc (0:180:1 and 0.5);
\draw (1,0) to [out=-45, in=left] (2,-.5) to [out=right,  in=225] (3,0);
\draw (2,-.5) node[below]{\footnotesize{$\alpha_1$}};

\draw (8,0) circle (1);
\draw (9,0) arc (0:-180:1 and 0.5);
\draw[dashed] (9,0) arc (0:180:1 and 0.5);
\draw (12,0) circle (1);
\draw (13,0) arc (0:-180:1 and 0.5);
\draw[dashed] (13,0) arc (0:180:1 and 0.5);

\draw (9,0) to [out=-45, in=left] (10,-.5) to [out=right,  in=225] (11,0);
\draw (10,-.5) node[below]{\footnotesize{$\alpha_2$}};

\draw (15,0) node[]{$\cdots$};

\draw[blue, thick] (5,0) to [out=-20, in=200] (7,0);
\draw[blue, thick] (13,0) to [out=-40, in=left] (14,-.5);
\draw[blue, thick] (16,-.5) to [out=right, in=220] (17,0);

\draw (18,0) circle (1);
\draw (19,0) arc (0:-180:1 and 0.5);
\draw[dashed] (19,0) arc (0:180:1 and 0.5);
\draw (26,0) circle (1);
\draw (27,0) arc (0:-180:1 and 0.5);
\draw[dashed] (27,0) arc (0:180:1 and 0.5);
%\draw (23,0) to [out=-45, in=left] (24,-.5) to [out=right,  in=225] (25,0);
%\draw (24,2) node[above]{\footnotesize{$\alpha_{N-a_i+2}$}};

\draw[fill=gray!50] (22,0) circle (1);
\draw (23,0) arc (0:-180:1 and 0.5);
\draw[dashed] (23,0) arc (0:180:1 and 0.5);

\draw[fill=gray!50] (28,3) circle (1);
\draw (29,3) arc (0:-180:1 and 0.5);
\draw[dashed] (29,3) arc (0:180:1 and 0.5);

\draw[olive] (28,4)--(28,6);

\draw[fill=gray!50] (28,7) circle (1);
\draw (29,7) arc (0:-180:1 and 0.5);
\draw[dashed] (29,7) arc (0:180:1 and 0.5);

\draw (28, 10.5) node[]{$\vdots$};

\draw[fill=gray!50] (28,13) circle (1);
\draw (29,13) arc (0:-180:1 and 0.5);
\draw[dashed] (29,13) arc (0:180:1 and 0.5);

\draw[olive] (28,14)--(28,16);

\draw[fill=gray!50] (28,17) circle (1);
\draw (29,17) arc (0:-180:1 and 0.5);
\draw[dashed] (29,17) arc (0:180:1 and 0.5);

\draw (28,20.5) node[]{$\vdots$};

\draw[fill=gray!50] (28,23) circle (1);
\draw (29,23) arc (0:-180:1 and 0.5);
\draw[dashed] (29,23) arc (0:180:1 and 0.5);

\draw[olive] (28,24)--(28,26);

\draw[fill=gray!50] (28,27) circle (1);
\draw (29,27) arc (0:-180:1 and 0.5);
\draw[dashed] (29,27) arc (0:180:1 and 0.5);

\draw[fill=gray!50] (34,0) circle (1);
\draw (35,0) arc (0:-180:1 and 0.5);
\draw[dashed] (35,0) arc (0:180:1 and 0.5);

\begin{knot}[
%	draft mode=crossings,
	clip width=5,
	clip radius = 2pt,
	end tolerance = 1pt,
]
\strand[olive] (23,0) to [out=-90, in=270] (33,0);
\strand (19,0) to [out=down, in=down] (25,0);

\strand[blue] (0,-1) to [out=-40, in=left] (13,-5) to [out=right, in=220](26,-1);

\strand[red] (21,0) to [out=down, in=left] (28,-6) to [out=right, in=down] (36,0) to [out=up,in=-20] (29,17);
\strand (31,0) to [out=down, in=down] (37,0);

\strand[red] (35,0) to [out=up, in=0] (29,3);
\strand[red] (29,7) to [out=0, in=-90] (30,9);
\strand[red] (29,13) to [out=0, in=90] (30,11);
\strand[blue, thick] (-1,0) to [out=-90, in=-90] (27,0);

\end{knot}

\draw[blue,thick] (13,-3) node[below]{\textbf{\footnotesize{$\gamma_N$}}};

\draw[blue,thick] (13,-6) node[below]{\textbf{\footnotesize{$\gamma_N$}}};

\draw[olive] (28,-3) node[below]{\footnotesize{$\beta_{1}$}};
\draw[olive] (28,5) node[left]{\footnotesize{$\beta_{2}$}};
\draw[olive] (28,15) node[left]{\footnotesize{$\beta_{c_1-1}$}};

\draw[red] (28,-6) node[below]{\footnotesize{$\delta_{c_1-1}$}};

\draw (30,0) circle (1);
\draw (31,0) arc (0:-180:1 and 0.5);
\draw[dashed] (31,0) arc (0:180:1 and 0.5);
\draw (38,0) circle (1);
\draw (39,0) arc (0:-180:1 and 0.5);
\draw[dashed] (39,0) arc (0:180:1 and 0.5);
%\draw (31,0) to [out=-45, in=left] (32,-.5) to [out=right,  in=225] (33,0);
%\draw (32,-1) node[below]{\footnotesize{$\alpha_{N-a_i+3}$}};

\draw (44,0) circle (1);
\draw (45,0) arc (0:-180:1 and 0.5);
\draw[dashed] (45,0) arc (0:180:1 and 0.5);
\draw (48,0) circle (1);
\draw (49,0) arc (0:-180:1 and 0.5);
\draw[dashed] (49,0) arc (0:180:1 and 0.5);
\draw (45,0) to [out=-45, in=left] (46,-.5) to [out=right,  in=225] (47,0);
%\draw (46,-1) node[below]{\footnotesize{$\alpha_{N}$}};

\draw (41,0) node[]{$\cdots$};

%\draw[red] (0,-1) to [out=down, in=left] (13,-7) to [out=right, in =down] (26,-1);
%\draw[red] (13,-7) node[below]{\footnotesize{$\gamma_{N}$}};

\draw[blue] (4,-1) to [out=down, in=down] (8,-1);
\draw[blue] (12,-1) to [out=down, in=left] (14,-2);
\draw[blue] (16,-2) to [out=right, in=down] (18,-1);
%\draw[red] (18,-2) to [out=right, in=down] (20,-1);

\end{tikzpicture}

\begin{tikzpicture}[scale=0.25]
\draw (0,3) node []{};
\draw (0,0) circle (1);
\draw (1,0) arc (0:-180:1 and 0.5);
\draw[dashed] (1,0) arc (0:180:1 and 0.5);
\draw (4,0) circle (1);
\draw (5,0) arc (0:-180:1 and 0.5);
\draw[dashed] (5,0) arc (0:180:1 and 0.5);
\draw (1,0) to [out=-45, in=left] (2,-.5) to [out=right,  in=225] (3,0);
\draw (2,-.5) node[below]{\footnotesize{$-1$}};

\draw (8,0) circle (1);
\draw (9,0) arc (0:-180:1 and 0.5);
\draw[dashed] (9,0) arc (0:180:1 and 0.5);
\draw (12,0) circle (1);
\draw (13,0) arc (0:-180:1 and 0.5);
\draw[dashed] (13,0) arc (0:180:1 and 0.5);

\draw (9,0) to [out=-45, in=left] (10,-.5) to [out=right,  in=225] (11,0);
\draw (10,-.5) node[below]{\footnotesize{$-1$}};

\draw (15,0) node[]{$\cdots$};

\draw[blue, thick] (5,0) to [out=-20, in=200] (7,0);
\draw[blue, thick] (13,0) to [out=-40, in=left] (14,-.5);
\draw[blue, thick] (16,-.5) to [out=right, in=220] (17,0);

\draw (18,0) circle (1);
\draw (19,0) arc (0:-180:1 and 0.5);
\draw[dashed] (19,0) arc (0:180:1 and 0.5);
\draw (26,0) circle (1);
\draw (27,0) arc (0:-180:1 and 0.5);
\draw[dashed] (27,0) arc (0:180:1 and 0.5);
%\draw (23,0) to [out=-45, in=left] (24,-.5) to [out=right,  in=225] (25,0);
%\draw (24,2) node[above]{\footnotesize{$\alpha_{N-a_i+2}$}};

\draw[fill=gray!50] (22,0) circle (1);
\draw (23,0) arc (0:-180:1 and 0.5);
\draw[dashed] (23,0) arc (0:180:1 and 0.5);

\draw[fill=gray!50] (34,0) circle (1);
\draw (35,0) arc (0:-180:1 and 0.5);
\draw[dashed] (35,0) arc (0:180:1 and 0.5);

\begin{knot}[
%	draft mode=crossings,
	clip width=5,
	clip radius = 2pt,
	end tolerance = 1pt,
]
\strand[olive] (23,0) to [out=-90, in=270] (33,0);
\strand (19,0) to [out=down, in=down] (25,0);

\strand[blue] (0,-1) to [out=-40, in=left] (13,-5) to [out=right, in=220](26,-1);

\strand[red] (21,0) to [out=down, in=left] (28,-6) to [out=right, in=down] (35,0);

\strand[blue, thick] (-1,0) to [out=-90, in=-90] (27,0);
%속도떄문에 일단 %
\strand[red] (34,-4.5) circle (1.5);
\strand[red] (36.5,-3) arc (90:270:1.5);
\strand[red] (39.5,-3) arc (90:-90:1.5);
\strand[red] (42,-4.5) circle (1.5);

\strand[blue] (18.5,-8.5) circle (1.5);
\strand[blue] (20.,-11) arc (0:180:1.5);
\strand[blue] (20.,-14) arc (0:-180:1.5);
\strand[blue] (18.5,-16.5) circle (1.5);

\flipcrossings{10,15,18,8,12,14}
\end{knot}

\draw[blue,thick] (13,-3) node[below]{\textbf{\footnotesize{$-1$}}};

\draw[blue,thick] (13,-6) node[below]{\textbf{\footnotesize{$-1$}}};

\draw[blue] (20,-8.5) node[right]{\footnotesize{${-b_1}$}};
\draw[blue] (20,-16.5) node[right]{\footnotesize{${-b_j}$}};

\draw[blue] (18.5,-12.5) node[]{$\vdots$};

\draw[olive] (28,-3) node[below]{\footnotesize{$-1$}};

\draw[red] (28,-6) node[below]{\footnotesize{${-c_1+1}$}};

\draw[red] (34,-6) node[below]{\footnotesize{${-c_2}$}};
\draw[red] (42,-6) node[below]{\footnotesize{${-c_k}$}};

\draw[red] (38,-4.5) node[]{$\cdots$};

%\draw[red] (0,-1) to [out=down, in=left] (13,-7) to [out=right, in =down] (26,-1);
%\draw[red] (13,-7) node[below]{\footnotesize{$\gamma_{N}$}};

\draw[blue] (4,-1) to [out=down, in=down] (8,-1);
\draw[blue] (12,-1) to [out=down, in=left] (14,-2);
\draw[blue] (16,-2) to [out=right, in=down] (18,-1);
%\draw[red] (18,-2) to [out=right, in=down] (20,-1);

\end{tikzpicture}
\caption{Part of Kirby diagram of $X_\Gamma$}
\label{minimalresolution}
\end{figure}
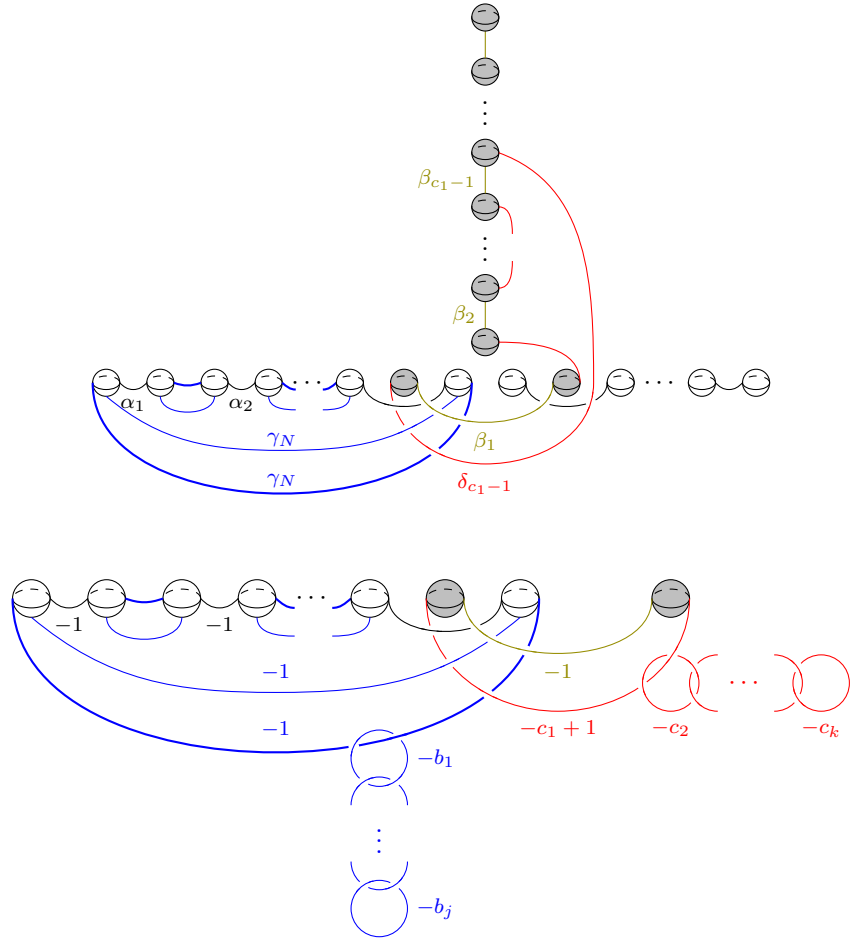
\end{proof}

If $\Gamma$ is one of the resolution graphs with bad central vertices in Figure~\ref{resolutiongraphs}, then the PALF $X_\Gamma$ induces the Milnor fillable contact structure on the boundary: Let $Y$ be the $3$-manifold diffeomorphic to the boundary of $X_\Gamma$. Then $Y$ is a small Seifert $3$-manifold $Y(-2; (\alpha_1, \beta_1), (\alpha_2, \beta_2), (\alpha_3, \beta_3))$. If $Y$ is an $L$-space, there is a classification of tight contact structures on $Y$ given by Ghiggini ( Theorem 1.3 in ~\cite{MR2443233}): A tight contact structure $\xi$ of $L$ is determined by $Spin^c$ structure $t_{\xi}$ induced by $\xi$ and filled by the Stein manifolds described via Legendrian surgery on all possible Legendrian realizations of the link corresponding to $\Gamma$. In particular, the Milnor fillable contact structure is filled by the Stein manifold whose first Chern class satisfies the adjunction equality on each vertex in $\Gamma$. 
On the other hand, the singularity corresponding to $\Gamma$ is a rational singularity, because it admits $\mathbb{Q}$HD smoothing. Therefore, the link of the singularity, which is diffeomorphic to the boundary of $X_\Gamma$, is $L$-space. Hence, by the theorem of Ghiggini, the PALF $X_\Gamma$ we constructed induces the Milnor fillable contact structure on the boundary of $X_\Gamma$. 
\begin{remark}
Using a technique similar to that described in  Proposition~\ref{prop1}, we can construct a genus-1 PALF structure on a $4$-legged resolution graph with a central $-3$ vertex whose first Chern class satisfies the adjunction equality on each vertex. But it is unknown whether the induced contact structure of the PALF is the Milnor fillable due to the lack of a classification of tight contact structures. 
\end{remark}

We now give an explicit relation $W_{\Gamma}=W_{\Gamma}'$ between two words of right-handed Dehn twists on simple closed curves in genus-$1$ surface $\Sigma_{\Gamma}$ where $W_{\Gamma}$ is the global monodromy of $X_\Gamma$ while the PALF $Y_{\Gamma}$ with the global monodromy $W_{\Gamma}'$ is a rational homology ball filling. Because of the homology condition on $Y_{\Gamma}$, the length of $W_{\Gamma}'$ must be equal to $b_1(\Sigma_{\Gamma})$. Conversely, if the length of $W_{\Gamma}'$ is equal to $b_1(\Sigma_{\Gamma})$, then $Y_{\Gamma}$ is a rational homology ball since the boundary of $X_\Gamma$ is rational homology $3$-sphere. We denote the right-handed Dehn twist on a curve $\alpha$ by $\alpha$ and also by $t_\alpha$ and use functional notation for the products of Dehn twists.
\subsection{Relations for (d) and (f) family}
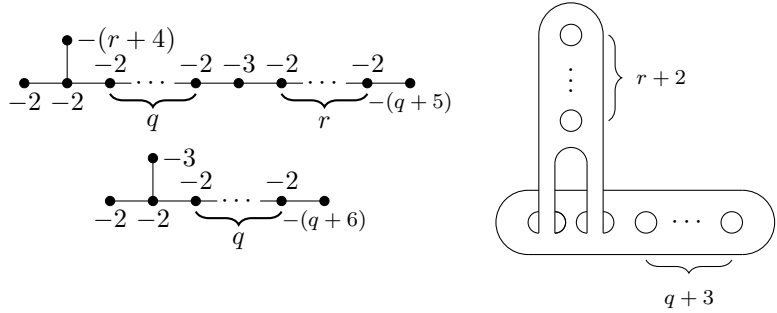
\begin{figure}[h]
\begin{tikzpicture}[scale=0.57]
\begin{scope}
\node[bullet] at (0,0){};
\node[bullet] at (1,0){};
\node[bullet] at (1,1){};

\node[bullet] at (2,0){};
\node[bullet] at (4,0){};
\node[bullet] at (5,0){};
\node[bullet] at (6,0){};
\node[bullet] at (8,0){};
\node[bullet] at (9,0){};

\node[below] at (-0,0){$-2$};
\node[below] at (1,0){$-2$};
\node[above] at (2,0){$-2$};
\node[above] at (4,0){$-2$};
\node[above] at (5,0){$-3$};
\node[above] at (6,0){$-2$};
\node[above] at (8,0){$-2$};

\node[below] at (9,0){\footnotesize{$-(q+5)$}};

\node[right] at (1,1){$-(r+4)$};

\node at (3,0){$\cdots$};
\node at (7,0){$\cdots$};

\draw (0,0)--(2.5,0);
\draw (3.5,0)--(6.5,0);
\draw (7.5,0)--(9,0);
\draw (1,0)--(1,1);
	\draw [thick,decorate,decoration={brace,mirror,amplitude=5pt},xshift=0pt,yshift=-7pt]
	(2,0) -- (4,0) node [black,midway,yshift=-11pt] 
	{$q$};
	\draw [thick,decorate,decoration={brace,mirror,amplitude=5pt},xshift=0pt,yshift=-7pt]
	(6,0) -- (8,0) node [black,midway,yshift=-11pt] 
	{$r$};

\end{scope}

\begin{scope}[shift={(2,-2.75)}]
\node[bullet] at (0,0){};
\node[bullet] at (1,0){};
\node[bullet] at (1,1){};

\node[bullet] at (2,0){};
\node[bullet] at (4,0){};
\node[bullet] at (5,0){};

\node[below] at (-0,0){$-2$};
\node[below] at (1,0){$-2$};
\node[above] at (2,0){$-2$};
\node[above] at (4,0){$-2$};

\node[below] at (5,0){\footnotesize{$-(q+6)$}};

\node[right] at (1,1){$-3$};

\node at (3,0){$\cdots$};

\draw (0,0)--(2.5,0);
\draw (3.5,0)--(5,0);

\draw (1,0)--(1,1);
	\draw [thick,decorate,decoration={brace,mirror,amplitude=5pt},xshift=0pt,yshift=-7pt]
	(2,0) -- (4,0) node [black,midway,yshift=-11pt] 
	{$q$};

\end{scope}

\begin{scope}[shift={(7.,-2.5)}, scale=0.25]
\draw (19,0)--(20,0) (21.5,0)--(24.5,0) (26,0)--(39,0);
\draw (19,-6)--(39,-6);
\draw (19,0) arc (90:270:3);
\draw (39,0) arc (90:-90:3);

\draw (20,-2) arc (90:270:1);
\draw (21.5,-2) arc (90:-90:1);

%\draw[red] (20,-1.5) arc (1.5);
\draw (21.5,-2) arc (90:-90:1);
%vertical part
\draw (20,-4)--(20,14.5);
\draw (21.5,-4)--(21.5,2.5);
%\draw[red] (20.75,-4)--(20.75,8.5);
%\draw[red] (20.75,8.5) arc (180:0:2.25 and 1.75);
%\draw[red] (20.75,-4)--(23,-5.4)--(25.25,-4);
%\draw[red] (20.75,-4) arc (180:360:2.25 and 1.25);

\draw (21.5,2.5) arc (180:0:1.5);
\draw (20,14.5) arc (180:0:3);

\draw (23, 6.5) circle (1);

\draw (23,11.) node{$\vdots$};

\draw (23, 14.5) circle (1);

\draw (23,18) node[]{};
\draw (24+2,-2) arc (90:-90:1);
\draw (24.5,-2) arc (90:270:1);

\draw (26,-4)--(26,14.5);
\draw (24.5,-4)--(24.5,2.5);
\draw [decorate,decoration={brace,mirror,amplitude=5pt},xshift=0pt,yshift=0]
	(26.5,6.5) -- (26.5,14.5) node [black,midway,xshift=20pt,align=center] 
	{\footnotesize $r+2$};

%vertical part 끝

\draw (30,-3) circle (1);
\draw (34,-3) node{$\cdots$};
\draw (38,-3) circle (1);

%\draw[blue] (20,-1.5)--(8,-1.5);
%\draw[blue] (21.5,-4.5)--(8,-4.5);
%\draw[blue] (8,-1.5) arc (90:270:1.5);
%\draw[blue] (21.5,-1.5) arc (90:-90:1.5);
\draw [decorate,decoration={brace,mirror,amplitude=5pt},xshift=0pt,yshift=0]
	(30,-6.5) -- (38,-6.5) node [black,midway,yshift=-15pt] 
	{\footnotesize $q+3$};
\end{scope}
\end{tikzpicture}

\caption{Resolution graph $\Gamma_{q,r}$ and generic fiber $\Sigma_{\Gamma_{q,r}}$ for (d) and (f) family}
\label{df}
\end{figure}
Let $\Gamma_{q,r}$(with $r\geq0$) be a resolution graph of (d) family and $\Gamma_{q,-1}$ be a resolution graph of (f) family as in Figure~\ref{df}. Then the generic fiber for $X_{\Gamma_{q,r}}$ is $\Sigma_{\Gamma_{q,r}}$ as in Figure~\ref{df} and the global monodromy  of $X_{\Gamma_{q,r}}$ is given by
\begin{eqnarray*}
&\phantom{0}& \hspace{-2 em} 
\beta_1
\cdots
\beta_{r+3}
\alpha_1^2
\delta_{r+3}
\alpha_1^{q+1}
\alpha_2
\gamma_{2}^{r+1}
\alpha_3
\cdots
\alpha_{q+5}
\gamma_{q+5}
\end{eqnarray*}
We introduce a cancelling pair $\delta_{r+3}^{-1}\cdot\delta_{r+3}$ and rearrange the word using Hurwitz moves.
\begin{eqnarray*}
&\phantom{0}& \hspace{-2 em} 
\beta_1
\cdots
\beta_{r+3}
\alpha_1^2
\delta_{r+3}
\alpha_1^{q+1}
\alpha_2
\gamma_{2}^{r+1}
\alpha_3
\cdots
\alpha_{q+5}
\gamma_{q+5}
\\
&=&
\beta_1
\cdots
\beta_{r+3}
\cdot
\textcolor{red}{
\delta_{r+3}^{-1}
\cdot
\delta_{r+3}
}
\cdot
\alpha_1^2
\delta_{r+3}
\alpha_1^{q+1}
\alpha_2
\gamma_{2}^{r+1}
\alpha_3
\cdots
\alpha_{q+5}
\gamma_{q+5}
\\
&=&
\beta_1
\cdots
\beta_{r+3}
\gamma_{2}^{r+1}
\cdot
\delta_{r+3}^{-1}
\cdot
\alpha_1^{q+3}
\alpha_2
\alpha_3
\cdots
\alpha_{q+5}
\gamma_{q+5}
\\
&&
\cdot
(t_{\alpha_{1}}^{-(q+3)}\cdot t_{\alpha_2}^{-1})(\delta_{r+3})
\cdot
(t_{\alpha_{1}}^{-(q+1)}\cdot t_{\alpha_2}^{-1})(\delta_{r+3})
\end{eqnarray*}
Let $c_{r}=t_{\alpha_2}^{-1}(\delta_{r+3})$. Then $c_r$ and $\alpha_1$ intersect geometrically once. Because of the braid relation $c_r\cdot\alpha_1\cdot c_r=\alpha_{1}\cdot c_r \cdot\alpha_{1}$, we have
\begin{eqnarray*}
&\phantom{0}& \hspace{-2 em} 
(t_{\alpha_{1}}^{-(q+3)}\cdot t_{\alpha_2}^{-1})(\delta_{r+3})
\cdot
(t_{\alpha_{1}}^{-(q+1)}\cdot t_{\alpha_2}^{-1})(\delta_{r+3})
\\
&=&
\alpha_1^{-(q+3)}
\cdot
c_{r}
\cdot
\alpha_1^{2}
\cdot
c_{r}
\cdot
\alpha_1^{(q+1)}
\\
&=&
\alpha_1^{-(q+3)}
\cdot
c_{r}
\cdot
\alpha_1
\cdot
c_{r}
\cdot
\alpha_1
\cdot
c_{r}
\cdot
\alpha_1^{q}
\\
&=&
\alpha_1^{-(q+2)}
\cdot
c_{r}
\cdot
\alpha_1^{2}
\cdot
c_{r}
\cdot
\alpha_1^{q}
\\
&&
\phantom{000000}
\vdots
\\
&=&
\alpha_1^{-2}
\cdot
c_{r}
\cdot
\alpha_1^{2}
\cdot
c_{r}
\\
&=&
\alpha_1^{-2}
\cdot
c_{r}
\cdot
\alpha_1^{2}
\cdot
c_{r}
\cdot
\alpha_1
\cdot
\alpha_1^{-1}
\cdot
\\
&=&
\alpha_1^{-2}
\cdot
c_{r}
\cdot
\alpha_1
\cdot
c_{r}
\cdot
\alpha_1
\cdot
c_{r}
\cdot
\alpha_1^{-1}
\\
&=&
\alpha_1^{-1}
\cdot
c_{r}
\cdot
\alpha_1
\cdot
\alpha_1
\cdot
c_{r}
\cdot
\alpha_1^{-1}
\\
&=&
(t_{\alpha_1}^{-1})(c_r)
\cdot
t_{\alpha_1}(c_r)
\end{eqnarray*}
On the other hand, we have a daisy relation of the form
$$\alpha_1^{q+3}\alpha_2\cdots\alpha_{q+5}\gamma_{q+5}=y_1y_2\cdots y_{q+5}$$
with $y_1=\gamma_2$ as in Figure~\ref{df1} to Figure~\ref{df2}.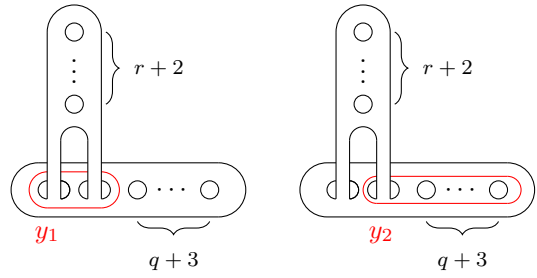
\begin{figure}[h]
\begin{tikzpicture}[scale=0.12]
\begin{scope}
\draw (19,0)--(20,0) (21.5,0)--(24.5,0) (26,0)--(39,0);
\draw (19,-6)--(39,-6);
\draw (19,0) arc (90:270:3);
\draw (39,0) arc (90:-90:3);

\draw[red] (20,-1) arc (90:270:2); 
\draw[red] (26,-1) arc (90:-90:2);
\draw[red] (21.5,-1)--(24.5,-1);
\draw[red] (20,-5)--(26,-5);
\draw[red] (20,-6) node[below]{$y_1$};

\draw (20,-2) arc (90:270:1);
\draw (21.5,-2) arc (90:-90:1);

%\draw[red] (20,-1.5) arc (1.5);
\draw (21.5,-2) arc (90:-90:1);
%vertical part
\draw (20,-4)--(20,14.5);
\draw (21.5,-4)--(21.5,2.5);
%\draw[red] (20.75,-4)--(20.75,8.5);
%\draw[red] (20.75,8.5) arc (180:0:2.25 and 1.75);
%\draw[red] (20.75,-4)--(23,-5.4)--(25.25,-4);
%\draw[red] (20.75,-4) arc (180:360:2.25 and 1.25);

\draw (21.5,2.5) arc (180:0:1.5);
\draw (20,14.5) arc (180:0:3);

\draw (23, 6.5) circle (1);

\draw (23,11.) node{$\vdots$};

\draw (23, 14.5) circle (1);

\draw (23,18) node[]{};
\draw (24+2,-2) arc (90:-90:1);
\draw (24.5,-2) arc (90:270:1);

\draw (26,-4)--(26,14.5);
\draw (24.5,-4)--(24.5,2.5);
\draw [decorate,decoration={brace,mirror,amplitude=5pt},xshift=0pt,yshift=0]
	(26.5,6.5) -- (26.5,14.5) node [black,midway,xshift=20pt,align=center] 
	{\footnotesize $r+2$};

%vertical part 끝

\draw (30,-3) circle (1);
\draw (34,-3) node{$\cdots$};
\draw (38,-3) circle (1);

%\draw[blue] (20,-1.5)--(8,-1.5);
%\draw[blue] (21.5,-4.5)--(8,-4.5);
%\draw[blue] (8,-1.5) arc (90:270:1.5);
%\draw[blue] (21.5,-1.5) arc (90:-90:1.5);
\draw [decorate,decoration={brace,mirror,amplitude=5pt},xshift=0pt,yshift=0]
	(30,-6.5) -- (38,-6.5) node [black,midway,yshift=-15pt] 
	{\footnotesize $q+3$};
\end{scope}

\begin{scope}[shift={(32,0)}]
\draw (19,0)--(20,0) (21.5,0)--(24.5,0) (26,0)--(39,0);
\draw (19,-6)--(39,-6);
\draw (19,0) arc (90:270:3);
\draw (39,0) arc (90:-90:3);

\draw (20,-2) arc (90:270:1);
\draw (21.5,-2) arc (90:-90:1);

%\draw[red] (20,-1.5) arc (1.5);
\draw (21.5,-2) arc (90:-90:1);
%vertical part
\draw (20,-4)--(20,14.5);
\draw (21.5,-4)--(21.5,2.5);
%\draw[red] (20.75,-4)--(20.75,8.5);
%\draw[red] (20.75,8.5) arc (180:0:2.25 and 1.75);
%\draw[red] (20.75,-4)--(23,-5.4)--(25.25,-4);
%\draw[red] (20.75,-4) arc (180:360:2.25 and 1.25);

\draw (21.5,2.5) arc (180:0:1.5);
\draw (20,14.5) arc (180:0:3);

\draw[red] (24.5,-1.5) arc (90:270:1.5); 
\draw[red] (39,-1.5) arc (90:-90:1.5);
\draw[red] (26,-1.5)--(39,-1.5);
\draw[red] (24.5,-4.5)--(39,-4.5);
\draw[red] (25,-6) node[below]{$y_2$};

\draw (23, 6.5) circle (1);

\draw (23,11.) node{$\vdots$};

\draw (23, 14.5) circle (1);

\draw (23,18) node[]{};
\draw (24+2,-2) arc (90:-90:1);
\draw (24.5,-2) arc (90:270:1);

\draw (26,-4)--(26,14.5);
\draw (24.5,-4)--(24.5,2.5);
\draw [decorate,decoration={brace,mirror,amplitude=5pt},xshift=0pt,yshift=0]
	(26.5,6.5) -- (26.5,14.5) node [black,midway,xshift=20pt,align=center] 
	{\footnotesize $r+2$};

%vertical part 끝

\draw (30,-3) circle (1);
\draw (34,-3) node{$\cdots$};
\draw (38,-3) circle (1);

%\draw[blue] (20,-1.5)--(8,-1.5);
%\draw[blue] (21.5,-4.5)--(8,-4.5);
%\draw[blue] (8,-1.5) arc (90:270:1.5);
%\draw[blue] (21.5,-1.5) arc (90:-90:1.5);
\draw [decorate,decoration={brace,mirror,amplitude=5pt},xshift=0pt,yshift=0]
	(30,-6.5) -- (38,-6.5) node [black,midway,yshift=-15pt] 
	{\footnotesize $q+3$};
\end{scope}
\end{tikzpicture}
\caption{$y_1$ and $y_2$}
\label{df1}
\end{figure}
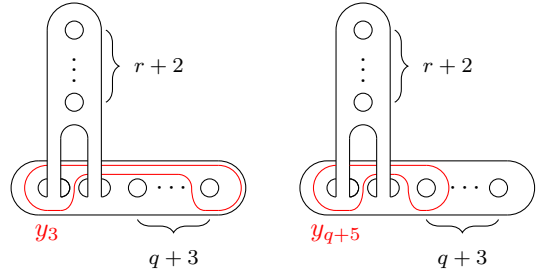
\begin{figure}[h]
\begin{tikzpicture}[scale=0.12]
\begin{scope}
\draw (19,0)--(20,0) (21.5,0)--(24.5,0) (26,0)--(39,0);
\draw (19,-6)--(39,-6);
\draw (19,0) arc (90:270:3);
\draw (39,0) arc (90:-90:3);

\draw[red] (20,-.5) arc (90:270:2.5);

\draw[red] (39,-.5) arc (90:-90:2.5);
\draw[red] (39,-5.5)--(37.5,-5.5);
\draw[red] (37.5,-5.5) to [out=left, in=down] (36,-3) to [out=up,in=right] (34.5,-1.5);
\draw[red] (26,-1.5)-- (34.5,-1.5);

\draw[red] (21.5,-.5)--(24.5,-.5);
\draw[red] (26,-.5)--(39,-.5);

\draw[red] (20,-5.5)--(21.5,-5.5);
\draw[red] (21.5,-5.5) to [out=right, in=down] (23,-3) to [out=up,in=left] (24.5,-1.5);

\draw[red] (20,-6) node[below]{$y_3$};

\draw (20,-2) arc (90:270:1);
\draw (21.5,-2) arc (90:-90:1);

%\draw[red] (20,-1.5) arc (1.5);
\draw (21.5,-2) arc (90:-90:1);
%vertical part
\draw (20,-4)--(20,14.5);
\draw (21.5,-4)--(21.5,2.5);
%\draw[red] (20.75,-4)--(20.75,8.5);
%\draw[red] (20.75,8.5) arc (180:0:2.25 and 1.75);
%\draw[red] (20.75,-4)--(23,-5.4)--(25.25,-4);
%\draw[red] (20.75,-4) arc (180:360:2.25 and 1.25);

\draw (21.5,2.5) arc (180:0:1.5);
\draw (20,14.5) arc (180:0:3);

\draw (23, 6.5) circle (1);

\draw (23,11.) node{$\vdots$};

\draw (23, 14.5) circle (1);

\draw (23,18) node[]{};
\draw (24+2,-2) arc (90:-90:1);
\draw (24.5,-2) arc (90:270:1);

\draw (26,-4)--(26,14.5);
\draw (24.5,-4)--(24.5,2.5);
\draw [decorate,decoration={brace,mirror,amplitude=5pt},xshift=0pt,yshift=0]
	(26.5,6.5) -- (26.5,14.5) node [black,midway,xshift=20pt,align=center] 
	{\footnotesize $r+2$};

%vertical part 끝

\draw (30,-3) circle (1);
\draw (34,-3) node{$\cdots$};
\draw (38,-3) circle (1);

%\draw[blue] (20,-1.5)--(8,-1.5);
%\draw[blue] (21.5,-4.5)--(8,-4.5);
%\draw[blue] (8,-1.5) arc (90:270:1.5);
%\draw[blue] (21.5,-1.5) arc (90:-90:1.5);
\draw [decorate,decoration={brace,mirror,amplitude=5pt},xshift=0pt,yshift=0]
	(30,-6.5) -- (38,-6.5) node [black,midway,yshift=-15pt] 
	{\footnotesize $q+3$};
\end{scope}

\begin{scope}[shift={(32,0)}]
\draw (19,0)--(20,0) (21.5,0)--(24.5,0) (26,0)--(39,0);
\draw (19,-6)--(39,-6);
\draw (19,0) arc (90:270:3);
\draw (39,0) arc (90:-90:3);

\draw (20,-2) arc (90:270:1);
\draw (21.5,-2) arc (90:-90:1);

%\draw[red] (20,-1.5) arc (1.5);
\draw (21.5,-2) arc (90:-90:1);
%vertical part
\draw (20,-4)--(20,14.5);
\draw (21.5,-4)--(21.5,2.5);
%\draw[red] (20.75,-4)--(20.75,8.5);
%\draw[red] (20.75,8.5) arc (180:0:2.25 and 1.75);
%\draw[red] (20.75,-4)--(23,-5.4)--(25.25,-4);
%\draw[red] (20.75,-4) arc (180:360:2.25 and 1.25);

\draw (21.5,2.5) arc (180:0:1.5);
\draw (20,14.5) arc (180:0:3);

\draw[red] (20,-.5) arc (90:270:2.5);

\draw[red] (30,-.5) arc (90:-90:2.5);

\draw[red] (30,-5.5)--(29.5,-5.5) to [out=left, in=down] (28,-3) to [out=up,in=right] (26.5,-1.5);

\draw[red] (21.5,-.5)--(24.5,-.5);
\draw[red] (26,-.5)--(30,-.5);
\draw[red] (26,-1.5)--(26.5,-1.5);

\draw[red] (20,-5.5)--(21.5,-5.5);
\draw[red] (21.5,-5.5) to [out=right, in=down] (23,-3) to [out=up,in=left] (24.5,-1.5);

\draw[red] (20,-6) node[below]{$y_{q+5}$};

\draw (23, 6.5) circle (1);

\draw (23,11.) node{$\vdots$};

\draw (23, 14.5) circle (1);

\draw (23,18) node[]{};
\draw (24+2,-2) arc (90:-90:1);
\draw (24.5,-2) arc (90:270:1);

\draw (26,-4)--(26,14.5);
\draw (24.5,-4)--(24.5,2.5);
\draw [decorate,decoration={brace,mirror,amplitude=5pt},xshift=0pt,yshift=0]
	(26.5,6.5) -- (26.5,14.5) node [black,midway,xshift=20pt,align=center] 
	{\footnotesize $r+2$};

%vertical part 끝

\draw (30,-3) circle (1);
\draw (34.5,-3) node{$\cdots$};
\draw (38,-3) circle (1);

%\draw[blue] (20,-1.5)--(8,-1.5);
%\draw[blue] (21.5,-4.5)--(8,-4.5);
%\draw[blue] (8,-1.5) arc (90:270:1.5);
%\draw[blue] (21.5,-1.5) arc (90:-90:1.5);
\draw [decorate,decoration={brace,mirror,amplitude=5pt},xshift=0pt,yshift=0]
	(30,-6.5) -- (38,-6.5) node [black,midway,yshift=-15pt] 
	{\footnotesize $q+3$};
\end{scope}
\end{tikzpicture}
\caption{$y_i$ for $i=3,\dots,q+5$}
\label{df2}
\end{figure}Hence after a daisy substitution and Hurwitz moves, we have
\begin{eqnarray*}
&&
\beta_1
\cdots
\beta_{r+3}
\gamma_{2}^{r+1}
\cdot
\delta_{r+3}^{-1}
\cdot
y_1
y_2
\cdots
y_{q+5}
\cdot
(t_{\alpha_1}^{-1})(c_r)
\cdot
t_{\alpha_1}(c_r)
\\
&=&
\beta_1
\cdots
\beta_{r+3}
\gamma_{2}^{r+2}
\cdot
\delta_{r+3}^{-1}
\cdot
y_2
\cdots
y_{q+5}
\cdot
(t_{\alpha_1}^{-1})(c_r)
\cdot
t_{\alpha_1}(c_r)
\\
&=&
Y_{2,r}
\cdots
Y_{q+5,r}
\cdot
(t_{\beta_1}\cdot t_{\delta_{r+3}}^{-1}\cdot t_{\alpha_1}^{-1})(c_r)
\cdot
\beta_1
\cdots
\beta_{r+3}
\cdot
t_{\alpha_1}(c_r)
\cdot
\gamma_{2}^{r+2}
\cdot
\delta_{r+3}^{-1}
\end{eqnarray*}
where $Y_{i,r}=(t_{\beta_1}\cdot t_{\delta_{r+3}}^{-1})(y_i)$. Again, we have a daisy relation of the form
$$\beta_1
\cdots
\beta_{r+3}
\cdot
t_{\alpha_1}(c_r)
\cdot
\gamma_{2}^{r+2}
=
z_1\cdot z_2 \cdots z_{r+3}\cdot z_{r+4}
$$
with $z_{r+4}=\delta_{r+3}$. See Figure~\ref{df3} for corresponding curves in planar sufrace.
By performing a daisy substitution and cancelling $z_{r+4}$ with $\delta_{r+3}^{-1}$, we get a  monodromy factorization $W_{\Gamma_{q,r}}'$ whose length is $b_1(\Sigma_{\Gamma_{q,r}})$.
$$Y_{2,r}
\cdots
Y_{q+5,r}
\cdot
(t_{\beta_1}\cdot t_{\delta_{r+3}}^{-1}\cdot t_{\alpha_1}^{-1})(c_r)
\cdot
z_{1}
\cdots
z_{r+3}
$$

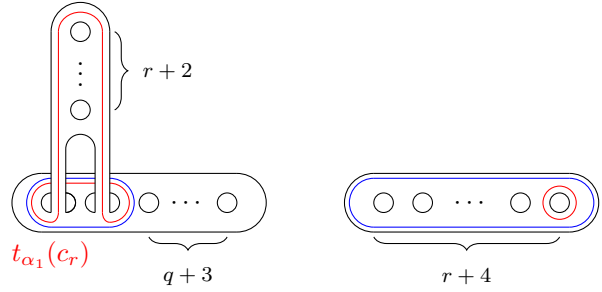
\begin{figure}[h]
\begin{tikzpicture}[scale=0.13]
\begin{scope}
\draw (19,0)--(20,0) (21.5,0)--(24.5,0) (26,0)--(39,0);
\draw (19,-6)--(39,-6);
\draw (19,0) arc (90:270:3);
\draw (39,0) arc (90:-90:3);

\draw[red] (20,-1) arc (90:270:2); 
\draw[red] (26,-1) arc (90:-90:2);
\draw[red] (20,-6) node[below]{$t_{\alpha_1}(c_r)
$};
\draw[red] (20,-5) to [out=right,in=down] (20.75,-4);
\draw[red] (26,-5) to [out=left,in=down] (25.25,-4);

\draw[red] (20.75,-4)--(20.75,14.5);
\draw[red] (25.25,-4)--(25.25,14.5);
\draw[red] (25.25,14.5) to[out=up,in=right] (23,16.5) to[out=left, in=up] (20.75,14.5);

\draw[blue] (24.5,-.5)--(21.5,-.5);
\draw[blue] (20,-.5) arc (90:270:2.5);
\draw[blue] (26,-.5) arc (90:-90:2.5);
\draw[blue] (20.,-5.5)--(26.,-5.5);

\draw (20,-2) arc (90:270:1);
\draw (21.5,-2) arc (90:-90:1);

\draw[red] (24.5,-1)--(21.5,-1);
%\draw[red] (20,-1.5) arc (1.5);
\draw (21.5,-2) arc (90:-90:1);
%vertical part
\draw (20,-4)--(20,14.5);
\draw (21.5,-4)--(21.5,2.5);
%\draw[red] (20.75,-4)--(20.75,8.5);
%\draw[red] (20.75,8.5) arc (180:0:2.25 and 1.75);
%\draw[red] (20.75,-4)--(23,-5.4)--(25.25,-4);
%\draw[red] (20.75,-4) arc (180:360:2.25 and 1.25);

\draw (21.5,2.5) arc (180:0:1.5);
\draw (20,14.5) arc (180:0:3);

\draw (23, 6.5) circle (1);

\draw (23,11.) node{$\vdots$};

\draw (23, 14.5) circle (1);

\draw (23,18) node[]{};
\draw (24+2,-2) arc (90:-90:1);
\draw (24.5,-2) arc (90:270:1);

\draw (26,-4)--(26,14.5);
\draw (24.5,-4)--(24.5,2.5);
\draw [decorate,decoration={brace,mirror,amplitude=5pt},xshift=0pt,yshift=0]
	(26.5,6.5) -- (26.5,14.5) node [black,midway,xshift=20pt,align=center] 
	{\footnotesize $r+2$};

%vertical part 끝

\draw (30,-3) circle (1);
\draw (34,-3) node{$\cdots$};
\draw (38,-3) circle (1);

%\draw[blue] (20,-1.5)--(8,-1.5);
%\draw[blue] (21.5,-4.5)--(8,-4.5);
%\draw[blue] (8,-1.5) arc (90:270:1.5);
%\draw[blue] (21.5,-1.5) arc (90:-90:1.5);
\draw [decorate,decoration={brace,mirror,amplitude=5pt},xshift=0pt,yshift=0]
	(30,-6.5) -- (38,-6.5) node [black,midway,yshift=-15pt] 
	{\footnotesize $q+3$};
\end{scope}

\begin{scope}[shift={(34,0)}]

\draw (19,0)--(39,0);
\draw (19,-6)--(39,-6);
\draw (19,0) arc (90:270:3);
\draw (39,0) arc (90:-90:3);

\draw (20,-3) circle (1);
\draw (24,-3) circle (1);
\draw (29,-3) node[]{$\cdots$};
\draw (34,-3) circle (1);
\draw (38,-3) circle (1);
\draw[red] (38,-3) circle (1.7);

\draw[blue] (19,-.5)--(39,-.5);
\draw[blue] (19,-5.5)--(39,-5.5);
\draw[blue] (19,-.5) arc (90:270:2.5);
\draw[blue] (39,-.5) arc (90:-90:2.5);

\draw [decorate,decoration={brace,mirror,amplitude=5pt},xshift=0pt,yshift=0]
	(19,-6.5) -- (38,-6.5) node [black,midway,yshift=-15pt] 
	{\footnotesize $r+4$};
\end{scope}
\end{tikzpicture}

\caption{Corresponding curves for a daisy relation}
\label{df3}
\end{figure}

\subsection{Relations for (j) family}

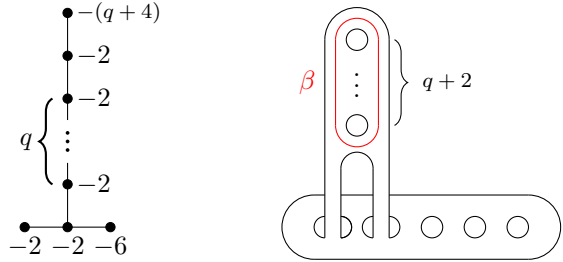
\begin{figure}[h]
\begin{tikzpicture}[scale=0.57]
\begin{scope}

\end{scope}

\begin{scope}[shift={(5,-3.25)}]
\node[bullet] at (0,0){};
\node[bullet] at (1,0){};
\node[bullet] at (1,1){};
\node[bullet] at (2,0){};

\node[bullet] at (1,3){};
\node[bullet] at (1,4){};
\node[bullet] at (1,5){};

\node[below] at (-0,0){$-2$};
\node[below] at (1,0){$-2$};
\node[below] at (2,0){$-6$};
\node[right] at (1,1){$-2$};

\node[right] at (1,5){\footnotesize{$-(q+4)$}};

\node[right] at (1,3){$-2$};
\node[right] at (1,4){$-2$};

\draw (1,1.8 ) node[circle, fill, inner sep=.5pt, black]{};
\draw (1,2 ) node[circle, fill, inner sep=.5pt, black]{};

\draw (1,2.2 ) node[circle, fill, inner sep=.5pt, black]{};

\draw (0,0)--(2,0);

\draw (1,0)--(1,1.5) (1,2.5)--(1,5);

	\draw [thick,decorate,decoration={brace,amplitude=5pt},xshift=-10pt,yshift=0pt]
	(1,1) -- (1,3) node [black,midway,xshift=-10pt] 
	{$q$};

\end{scope}

\begin{scope}[shift={(7.,-2.5)}, scale=0.25]
\draw (19,0)--(20,0) (21.5,0)--(24.5,0) (26,0)--(39,0);
\draw (19,-6)--(39,-6);
\draw (19,0) arc (90:270:3);
\draw (39,0) arc (90:-90:3);

\draw (20,-2) arc (90:270:1);
\draw (21.5,-2) arc (90:-90:1);

%\draw[red] (20,-1.5) arc (1.5);
\draw (21.5,-2) arc (90:-90:1);
%vertical part
\draw (20,-4)--(20,14.5);
\draw (21.5,-4)--(21.5,2.5);
%\draw[red] (20.75,-4)--(20.75,8.5);
%\draw[red] (20.75,8.5) arc (180:0:2.25 and 1.75);
%\draw[red] (20.75,-4)--(23,-5.4)--(25.25,-4);
%\draw[red] (20.75,-4) arc (180:360:2.25 and 1.25);

\draw (21.5,2.5) arc (180:0:1.5);
\draw (20,14.5) arc (180:0:3);

\draw (23, 6.5) circle (1);

\draw (23,11.) node{$\vdots$};

\draw (23, 14.5) circle (1);

\draw (23,18) node[]{};
\draw (24+2,-2) arc (90:-90:1);
\draw (24.5,-2) arc (90:270:1);

\draw (26,-4)--(26,14.5);
\draw (24.5,-4)--(24.5,2.5);
\draw [decorate,decoration={brace,mirror,amplitude=5pt},xshift=0pt,yshift=0]
	(26.5,6.5) -- (26.5,14.5) node [black,midway,xshift=20pt,align=center] 
	{\footnotesize $q+2$};
\draw[red] (21,6.5)--(21,14.5);
\draw[red] (25.,6.5)--(25.,14.5);
\draw[red] (25.,14.5) arc (0:180:2);
\draw[red] (25.,6.5) arc (0:-180:2);
\draw[red] (20,10.5) node[left]{$\beta$};

%vertical part 끝

\draw (30,-3) circle (1);
\draw (34,-3) circle (1);
\draw (38,-3) circle (1);

%\draw[blue] (20,-1.5)--(8,-1.5);
%\draw[blue] (21.5,-4.5)--(8,-4.5);
%\draw[blue] (8,-1.5) arc (90:270:1.5);
%\draw[blue] (21.5,-1.5) arc (90:-90:1.5);
\end{scope}
\end{tikzpicture}

\caption{Resolution graph $\Gamma_{q}$ and generic fiber $\Sigma_{\Gamma_{q}}$ for (j) family}
\label{j}
\end{figure}

Let $\Gamma_{q}$ be a resolution graph of (j) family as in Figure~\ref{j}. Then the generic fiber for $X_{\Gamma_{q}}$ is $\Sigma_{r}$ as in Figure~\ref{j} and the global monodromy  of $X_{\Gamma_{q}}$ is given by
\begin{eqnarray*}
\beta_1
\alpha_1^2
\beta_1^{q+1}
\beta_2
\cdots
\beta_{q+3}
\delta_{q+3}
\alpha_1
\alpha_2
\alpha_3
\alpha_4
\alpha_5
\gamma_5.
\end{eqnarray*}
We introduce a cancelling pair $\beta^{-1}\cdot\beta$ and rearrange the word using Hurwitz moves where $\beta$ is a simple closed curve in $\Sigma_r$ as in Figure~\ref{j}.
\begin{eqnarray*}
&&\hspace{-2em}
\beta_1
\cdot
\textcolor{red}{
\beta^{-1}
\cdot
\beta
}
\cdot
\alpha_1^2
\beta_1^{q+1}
\beta_2
\cdots
\beta_{q+3}
\delta_{q+3}
\alpha_1
\alpha_2
\alpha_3
\alpha_4
\alpha_5
\gamma_5
\\
&=&
\beta_2
\cdots
\beta_{q+3}
\cdot
\beta^{-1}
\cdot
\beta_1
\beta
\alpha_1^2
\beta_1^{q+1}
\delta_{q+3}
\alpha_1
\alpha_2
\alpha_3
\alpha_4
\alpha_5
\gamma_5
\\
&=&
\beta_2
\cdots
\beta_{q+3}
\cdot
\beta^{-1}
\cdot
\underline{
\beta_1
\beta
\alpha_1^2
\delta_{q+3}
\alpha_1
\alpha_2
\alpha_3
\alpha_4
\alpha_5
\gamma_5
}
\cdot
((t_{\alpha_1}^{-1}\cdot t_{\alpha_2}^{-1})(\beta_1))^{q+1}
\end{eqnarray*}
There is a obvious subsurface of $\Sigma_q$ which is diffeomorphic to $\Sigma_{0,-1}$ of Figure~\ref{df} so that image of each curve of $\Sigma_{0,-1}$ in $\Sigma_{q}$ is as follows
\begin{eqnarray*}
\beta_1 &\rightarrow& \beta_1
\\
\beta_2 &\rightarrow& \beta
\\
\delta_2 &\rightarrow& \delta_{q+3}
\\
\alpha_i &\rightarrow& \alpha_i
\end{eqnarray*}
By performing a monodromy substitution corresponds to $\Gamma_{0,-1}$ of (f) family, we have
\begin{eqnarray*}
&&
\beta_2
\cdots
\beta_{q+3}
\cdot
\beta^{-1}
\cdot
\textcolor{red}{
\beta_1
\beta
\alpha_1^2
\delta_{q+3}
\alpha_1
\alpha_2
\alpha_3
\alpha_4
\alpha_5
\gamma_5
}
\cdot
((t_{\alpha_1}^{-1}\cdot t_{\alpha_2}^{-1})(\beta_1))^{q+1}
\\
&=&
\beta_2
\cdots
\beta_{q+3}
\cdot
\beta^{-1}
\cdot
Y_{q,2}
\cdots
Y_{q,5}
\cdot
(t_{\beta_1}\cdot t_{\delta_{q+3}}^{-1}\cdot t_{\alpha_1}^{-1})(c_q)
\cdot
z_{q,1}
z_{q,2}
\cdot
B^{q+1}
\\
&=&
Y_{q,2}
\cdots
Y_{q,5}
\\
&&
\cdot
\beta_2
\cdots
\beta_{q+3}
\cdot
(t_{\beta_1}\cdot t_{\delta_{q+3}}^{-1}\cdot t_{\alpha_1}^{-1})(c_q)
\cdot
B^{q+1}
\cdot
\beta^{-1}
\cdot
(t_B^{-(q+1)})(z_{q,1})
\cdot
(t_B^{-(q+1)})(z_{q,2})
\end{eqnarray*}
Here $Y_{q,i}$ and $z_{q,j}$ is image of $Y_{i,-1}$ and $z_j$ of $\Sigma_{0,-1}$ in $\Sigma_{q}$ respectively, $c_{q}=t_{\alpha_2}^{-1}(\delta_{q+3})$ and $B=(t_{\alpha_1}^{-1}\cdot t_{\alpha_2}^{-1})(\beta_1)$.
We have a daisy relation of the form
$$\beta_2
\cdots
\beta_{q+3}
\cdot
(t_{\beta_1}\cdot t_{\delta_{q+3}}^{-1}\cdot t_{\alpha_1}^{-1})(c_q)
\cdot
B^{q+1}
=x_1\cdots x_{q+3}
$$
with $x_{q+3}=\beta$. See Figure~\ref{j1} for corresponding curves in planar surface. By performing a daisy substitution and cancelling $x_{q+3}$ with $\beta^{-1}$, we get a  monodromy factorization $W_{\Gamma_q}'$ whose length is $b_1(\Sigma_{\Gamma_{q}})$.
$$
Y_{q,2}\cdots
Y_{q,5} \cdot
x_1\cdots x_{q+2}
\cdot
(t_B^{-(q+1)})(z_{q,1})
\cdot
(t_B^{-(q+1)})(z_{q,2})
$$
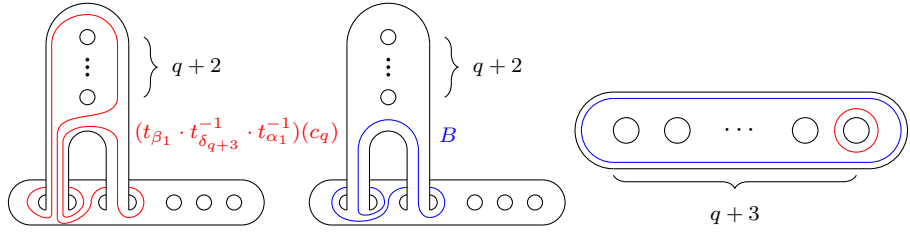
\begin{figure}[h]
\begin{tikzpicture}[scale=0.1]
\begin{scope}
\draw (0,0) arc (90:270:3);
\draw (2,-2) arc (90:270:1);
\draw (5,-2) arc (90:-90:1);
\draw (10,-2) arc (90:270:1);
\draw (13,-2) arc (90:-90:1);
\draw (19,-3) circle (1);
\draw (23,-3) circle (1);
\draw (27,-3) circle (1);
\draw (28,0) arc (90:-90:3);

\draw (5,-4)--(5,4); 
\draw (10,-4)--(10,4); 
\draw (5,4) arc (180:0:2.5);

\draw (7.5,11) circle (1);
\draw (7.5,14 ) node[circle, fill, inner sep=.5pt, black]{};
\draw (7.5,15 ) node[circle, fill, inner sep=.5pt, black]{};
\draw (7.5,16 ) node[circle, fill, inner sep=.5pt, black]{};

\draw [decorate,decoration={brace,mirror,amplitude=5pt},xshift=0pt,yshift=0]
	(15,11) -- (15,19) node [black,midway,xshift=20pt,align=center] 
	{\footnotesize $q+2$};

\draw (7.5,19) circle (1);

\draw (2,-4)--(2,18); 
\draw (13,-4)--(13,18); 
\draw (2,18) arc (180:0:5.5);

\draw (0,0)--(2,0) (5,0)--(10,0) (13,0)--(28,0);
\draw (0,-6)--(28,-6);

\draw[red] (2.75,-4.5)--(2.75,18);
\draw[red] (2.75,-4.5) to [out=down,in=down] (-0.5,-3) to [out=up,in=left] (2,-1);
\draw[red] (13,-1) arc (90:-90:2);
\draw[red] (13,-5) to [out=left,in=down] (11.5,-3)--(11.5,5) to [out=up,in=up] (4.25,5);

\draw[red] (3.5,-5)--(3.5,6);
\draw[red] (3.5,-5) to [out=down,in=down] (8,-3) to [out=up,in=left] (10,-1);
\draw[red] (3.5,6) to [out=up,in=down](11.5,11);
\draw[red] (4.25,-4.5)--(4.25,5);
\draw[red] (4.25,-4.5) to [out=down, in=down] (7,-3) to [out=up,in=right] (5,-1);

\draw[red](2.75,18) to [out=up,in=left] (7.5,22) to [out=right, in=up] (11.5,18)--(11.5,11);
\draw[red] (12.5,6) node[right]{\footnotesize $(t_{\beta_1}\cdot t_{\delta_{q+3}}^{-1}\cdot t_{\alpha_1}^{-1})(c_q)
$};

\end{scope}

\begin{scope}[shift={(40,0)}]
\draw (0,0) arc (90:270:3);
\draw (2,-2) arc (90:270:1);
\draw (5,-2) arc (90:-90:1);
\draw (10,-2) arc (90:270:1);
\draw (13,-2) arc (90:-90:1);
\draw (19,-3) circle (1);
\draw (23,-3) circle (1);
\draw (27,-3) circle (1);
\draw (28,0) arc (90:-90:3);

\draw (5,-4)--(5,4); 
\draw (10,-4)--(10,4); 
\draw (5,4) arc (180:0:2.5);

\draw (7.5,11) circle (1);
\draw (7.5,14 ) node[circle, fill, inner sep=.5pt, black]{};
\draw (7.5,15 ) node[circle, fill, inner sep=.5pt, black]{};
\draw (7.5,16 ) node[circle, fill, inner sep=.5pt, black]{};

\draw [decorate,decoration={brace,mirror,amplitude=5pt},xshift=0pt,yshift=0]
	(15,11) -- (15,19) node [black,midway,xshift=20pt,align=center] 
	{\footnotesize $q+2$};

\draw (7.5,19) circle (1);

\draw (2,-4)--(2,18); 
\draw (13,-4)--(13,18); 
\draw (2,18) arc (180:0:5.5);

\draw (0,0)--(2,0) (5,0)--(10,0) (13,0)--(28,0);
\draw (0,-6)--(28,-6);

\draw[blue] (13,-1) arc (90:-90:2);
\draw[blue] (13,-5) to [out=left,in=down] (11.5,-3)--(11.5,4) to[out=up,in=right](7.5,8) to [out=left,in=up] (3.5,4)--(3.5,-4) to [out=down,in=down] (7,-3) to [out=up,in=right](5,-1);
\draw[blue] (2,-1) to [out=left,in=up] (-0,-3) to [out=down,in=left] (3.5,-5.5) to [out=right,in=down] (7.5,-3) to[out=up,in=left] (10,-1);

\draw[blue] (13,6) node[right]{\footnotesize$B$};
\end{scope}
\end{tikzpicture}
\begin{tikzpicture}[scale=0.17]
\draw (19,4) node[]{};
\draw (19,0)--(39,0);
\draw (19,-6)--(39,-6);
\draw (19,0) arc (90:270:3);
\draw (39,0) arc (90:-90:3);

\draw (20,-3) circle (1);
\draw (24,-3) circle (1);
\draw (29,-3) node[]{$\cdots$};
\draw (34,-3) circle (1);
\draw (38,-3) circle (1);
\draw[red] (38,-3) circle (1.7);

\draw[blue] (19,-.5)--(39,-.5);
\draw[blue] (19,-5.5)--(39,-5.5);
\draw[blue] (19,-.5) arc (90:270:2.5);
\draw[blue] (39,-.5) arc (90:-90:2.5);

\draw [decorate,decoration={brace,mirror,amplitude=5pt},xshift=0pt,yshift=0]
	(19,-6.5) -- (38,-6.5) node [black,midway,yshift=-15pt] 
	{\footnotesize $q+3$};

\end{tikzpicture}
\caption{Corresponding curves for a daisy relation}
\label{j1}
\end{figure}

\subsection{Relations for (e) and (g) family}
\begin{figure}[h]
\begin{tikzpicture}[scale=0.6]
\begin{scope}
\node[bullet] at (0,0){};
\node[bullet] at (1,0){};
\node[bullet] at (1,1){};

\node[bullet] at (2,0){};
\node[bullet] at (4,0){};
\node[bullet] at (5,0){};
\node[bullet] at (6,0){};
\node[bullet] at (8,0){};
\node[bullet] at (9,0){};

\node[below] at (-0,0){$-3$};
\node[below] at (1,0){$-2$};
\node[above] at (2,0){$-2$};
\node[above] at (4,0){$-2$};
\node[above] at (5,0){$-4$};
\node[above] at (6,0){$-2$};
\node[above] at (8,0){$-2$};

\node[below] at (9,0){\footnotesize{$-(q+4)$}};

\node[right] at (1,1){$-(p+3)$};

\node at (3,0){$\cdots$};
\node at (7,0){$\cdots$};

\draw (0,0)--(2.5,0);
\draw (3.5,0)--(6.5,0);
\draw (7.5,0)--(9,0);
\draw (1,0)--(1,1);
	\draw [thick,decorate,decoration={brace,mirror,amplitude=5pt},xshift=0pt,yshift=-7pt]
	(2,0) -- (4,0) node [black,midway,yshift=-11pt] 
	{$q$};
	\draw [thick,decorate,decoration={brace,mirror,amplitude=5pt},xshift=0pt,yshift=-7pt]
	(6,0) -- (8,0) node [black,midway,yshift=-11pt] 
	{$p$};

\end{scope}
\end{tikzpicture}

\begin{tikzpicture}[scale=0.6]
\begin{scope}
\node[bullet] at (0,0){};
\node[bullet] at (1,0){};

\node[bullet] at (1,1){};
\node[] at (1,1.5){};

\node[bullet] at (2,0){};
\node[bullet] at (4,0){};
\node[bullet] at (5,0){};
\node[bullet] at (6,0){};
\node[bullet] at (8,0){};
\node[bullet] at (9,0){};
\node[bullet] at (10,0){};
\node[bullet] at (12,0){};
\node[bullet] at (13,0){};

\node[below] at (-0.5,0){\footnotesize{$-(r+4)$}};
\node[below] at (1,0){$-2$};
\node[above] at (2,0){$-2$};
\node[above] at (4,0){$-2$};
\node[above] at (5,0){$-3$};
\node[above] at (6,0){$-2$};
\node[above] at (8,0){$-2$};
\node[above] at (9,0){$-3$};
\node[above] at (10,0){$-2$};
\node[above] at (12,0){$-2$};
\node[below] at (13,0){\footnotesize{$-(q+4)$}};

\node[right] at (1,1){$-(p+3)$};

\node at (3,0){$\cdots$};
\node at (7,0){$\cdots$};
\node at (11,0){$\cdots$};

\draw (0,0)--(2.5,0);
\draw (3.5,0)--(6.5,0);
\draw (7.5,0)--(10.5,0);
\draw (11.5,0)--(13.,0);
\draw (1,0)--(1,1);
	\draw [thick,decorate,decoration={brace,mirror,amplitude=5pt},xshift=0pt,yshift=-7pt]
	(2,0) -- (4,0) node [black,midway,yshift=-11pt] 
	{$q$};
	\draw [thick,decorate,decoration={brace,mirror,amplitude=5pt},xshift=0pt,yshift=-7pt]
	(6,0) -- (8,0) node [black,midway,yshift=-11pt] 
	{$r$};
	\draw [thick,decorate,decoration={brace,mirror,amplitude=5pt},xshift=0pt,yshift=-7pt]
	(10,0) -- (12,0) node [black,midway,yshift=-11pt] 
	{$p$};

\end{scope}
\end{tikzpicture}

\begin{tikzpicture}[scale=0.13]

\draw (7,0)--(20,0) (21.5,0)--(24.5,0) (26,0)--(39,0);
\draw (7,-6)--(39,-6);
\draw (7,0) arc (90:270:3);
\draw (39,0) arc (90:-90:3);
\draw (8,-3) circle (1);
%\draw[red] (8,-3) circle (1.5);
\draw (16,-3) circle (1);
%\draw[red] (12,-3) circle (1.5);

\draw (12,-3) node{$\cdots$};
\draw (20,-2) arc (90:270:1);
\draw (21.5,-2) arc (90:-90:1);

%\draw[red] (20,-1.5) arc (1.5);
\draw (21.5,-2) arc (90:-90:1);
%vertical part
\draw (20,-4)--(20,14.5);
\draw (21.5,-4)--(21.5,2.5);
%\draw[red] (20.75,-4)--(20.75,8.5);
%\draw[red] (20.75,8.5) arc (180:0:2.25 and 1.75);
%\draw[red] (20.75,-4)--(23,-5.4)--(25.25,-4);
%\draw[red] (20.75,-4) arc (180:360:2.25 and 1.25);

\draw (21.5,2.5) arc (180:0:1.5);
\draw (20,14.5) arc (180:0:3);

\draw[red] (21,6.5)--(21,14.5);
\draw[red] (25.,6.5)--(25.,14.5);
\draw[red] (25.,14.5) arc (0:180:2);
\draw[red] (25.,6.5) arc (0:-180:2);
\draw[red] (20,10.5) node[left]{$\beta$};

\draw (23, 6.5) circle (1);

\draw (23,11.) node{$\vdots$};

\draw (23, 14.5) circle (1);

\draw (23,18) node[]{};
\draw (24+2,-2) arc (90:-90:1);
\draw (24.5,-2) arc (90:270:1);

\draw (26,-4)--(26,14.5);
\draw (24.5,-4)--(24.5,2.5);
\draw [decorate,decoration={brace,mirror,amplitude=5pt},xshift=0pt,yshift=0]
	(26.5,6.5) -- (26.5,14.5) node [black,midway,xshift=20pt,align=center] 
	{\footnotesize $p+1$};

%vertical part 끝

\draw (30,-3) circle (1);
\draw (34,-3) node{$\cdots$};
\draw (38,-3) circle (1);

%\draw[blue] (20,-1.5)--(8,-1.5);
%\draw[blue] (21.5,-4.5)--(8,-4.5);
%\draw[blue] (8,-1.5) arc (90:270:1.5);
%\draw[blue] (21.5,-1.5) arc (90:-90:1.5);
\draw [decorate,decoration={brace,mirror,amplitude=5pt},xshift=0pt,yshift=0]
	(8,-6.5) -- (16,-6.5) node [black,midway,yshift=-15pt] 
	{\footnotesize $r+2$};
\draw [decorate,decoration={brace,mirror,amplitude=5pt},xshift=0pt,yshift=0]
	(30,-6.5) -- (38,-6.5) node [black,midway,yshift=-15pt] 
	{\footnotesize $q+3$};

\end{tikzpicture}

\caption{Resolution graph $\Gamma_{p,q,r}$ and generic fiber $\Sigma_{\Gamma_{p,q,r}}$ for (e) and (g) family}
\label{eg}
\end{figure}
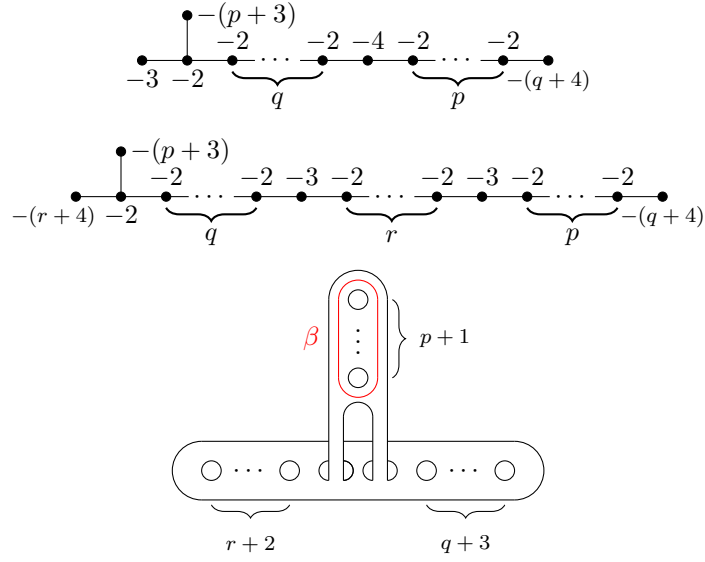
Let $\Gamma_{p,q,-1}$ be a resolution graph of (e) family and $\Gamma_{p,q,r}$(with $r\geq0$) be a resolution graph of (g) family as in Figure~\ref{eg}. Then the generic fiber for $X_{\Gamma_{p,q,r}}$ is $\Sigma_{p,q,r}$ as in Figure~\ref{eg} and the global monodromy  of $X_{\Gamma_{p,q,r}}$ is given by
\begin{eqnarray*}
&\phantom{0}& \hspace{-2 em} \beta_1\cdots\beta_{p+2}\alpha_1\cdots\alpha_{r+3}\gamma_{r+3}\delta_{p+2}\gamma_{r+3}^{q+1}\alpha_{r+4}\gamma_{r+4}^{r+1}\alpha_{r+5}\gamma_{r+5}^{p+1}\alpha_{r+6}\cdots\alpha_{r+q+7}\gamma_{r+q+7}
\end{eqnarray*}
We introduce cancelling pairs $\beta^{-1}\cdot\beta$
, $\gamma_{r+2}^{-1}\cdot\gamma_{r+2}$ and $\gamma_{r+3}^{-1}\cdot\gamma_{r+3}$ and rearrange the word using Hurwitz moves where $\beta$ is a simple closed curve in $\Sigma_{p,q,r}$ as in Figure~\ref{eg}.
\begin{eqnarray*}
&\phantom{0}& \hspace{-2 em} \beta_1\cdots\beta_{p+2}\alpha_1\cdots\alpha_{r+3}\gamma_{r+3}\delta_{p+2}\gamma_{r+3}^{q+1}\alpha_{r+4}\gamma_{r+4}^{r+1}\alpha_{r+5}\gamma_{r+5}^{p+1}\alpha_{r+6}\cdots\alpha_{r+q+7}\gamma_{r+q+7}
\\
&=& \beta_1\cdots\beta_{p+2}\cdot\textcolor{red}{\beta^{-1}\cdot\beta}\cdot\alpha_1\cdots\alpha_{r+2}\cdot\textcolor{red}{\gamma_{r+2}^{-1}\cdot\gamma_{r+2}}\cdot\alpha_{r+3}\cdot\textcolor{red}{\gamma_{r+3}^{-1}\cdot\gamma_{r+3}}
\\
&&\cdot\gamma_{r+3}\delta_{p+2}\gamma_{r+3}^{q+1}\alpha_{r+4}\gamma_{r+4}^{r+1}\alpha_{r+5}\gamma_{r+5}^{p+1}\alpha_{r+6}\cdots\alpha_{r+q+7}\gamma_{r+q+7}
\\
&=& \beta_1\cdots\beta_{p+2}\cdot\beta^{-1}\cdot\alpha_1\cdots\alpha_{r+2}\cdot\gamma_{r+2}^{-1}\cdot\gamma_{r+3}^{-1}
\\
&&\cdot\beta\gamma_{r+2}\alpha_{r+3}\gamma_{r+3}^2\delta_{p+2}\gamma_{r+3}^{q+1}\alpha_{r+4}\gamma_{r+4}^{r+1}\alpha_{r+5}\gamma_{r+5}^{p+1}\alpha_{r+6}\cdots\alpha_{r+q+7}\gamma_{r+q+7}
\\
&=& \beta_1\cdots\beta_{p+2}\cdot\beta^{-1}\cdot\alpha_1\cdots\alpha_{r+2}\cdot\gamma_{r+2}^{-1}\cdot\gamma_{r+3}^{-1}\cdot\gamma_{r+4}^{r+1}\cdot\gamma_{r+5}^p
\\
&&\cdot\beta\gamma_{r+2}\alpha_{r+3}\gamma_{r+3}^2\delta_{p+2}\gamma_{r+3}^{q+1}\alpha_{r+4}\alpha_{r+5}\gamma_{r+5}\alpha_{r+6}\cdots\alpha_{r+q+7}\gamma_{r+q+7}
\\
&=& \beta_2\cdots\beta_{p+2}\cdot\beta^{-1}\cdot\alpha_1\cdots\alpha_{r+2}\cdot\gamma_{r+2}^{-1}\cdot\gamma_{r+4}^{r+1}\cdot\gamma_{r+5}^p\cdot (t_{\beta_1}(\gamma_{r+3}))^{-1}
\\
&&\cdot\beta\gamma_{r+2}
t_{\beta_1}(\alpha_{r+3})
(t_{\beta_1}(\gamma_{r+3}))^2
\delta_{p+2}
(t_{\beta_1}(\gamma_{r+3}))^{q+1}
\beta_1
\alpha_{r+4}
\\
&&
\cdot\alpha_{r+5}\gamma_{r+5}\alpha_{r+6}\cdots\alpha_{r+q+7}\gamma_{r+q+7}
\\&=& \beta_2\cdots\beta_{p+2}\cdot\beta^{-1}\cdot\alpha_1\cdots\alpha_{r+2}\cdot\gamma_{r+2}^{-1}\cdot\gamma_{r+4}^{r+1}\cdot\gamma_{r+5}^p\cdot (t_{\beta_1}(\gamma_{r+3}))^{-1}
\\
&&\cdot\beta\gamma_{r+2}
t_{\beta_1}(\alpha_{r+3})
(t_{\beta_1}(\gamma_{r+3}))^2
\delta_{p+2}
(t_{\beta_1}(\gamma_{r+3}))^{q+1}
\alpha_{r+4}
(t_{\alpha_{r+4}}^{-1})(\beta_1)
\\
&&
\cdot\alpha_{r+5}\gamma_{r+5}\alpha_{r+6}\cdots\alpha_{r+q+7}\gamma_{r+q+7}
\end{eqnarray*}
Let $\gamma=t_{\beta_1}(\gamma_{r+3})$. Then $\gamma$ and $\delta_{p+2}$ intersect geometrically. Because of the braid relation $\gamma\cdot\delta_{p+2}\cdot\gamma=\delta_{p+2}\cdot\gamma\cdot\delta_{p+2}$, we have
\begin{eqnarray*}
&\phantom{0}& \hspace{-2em} 
\gamma^2
\delta_{p+2}
\gamma^{q+1}
\\
&=&
\gamma
\cdot
\delta_{p+2}
\cdot
\gamma
\cdot
\delta_{p+2}
\cdot
\gamma^q
\\
&=&
\gamma
\cdot
\delta_{p+2}^2
\cdot
\gamma
\cdot
\delta_{p+2}
\cdot
\gamma^{q-1}
\\
&&
\phantom{0000}
\vdots
\\
&=&
\gamma
\cdot
\delta_{p+2}^{q+1}
\cdot
\gamma
\cdot
\delta_{p+2}
\\
&=&
\gamma
\cdot
\delta_{p+2}^{q+2}
\cdot
(t_{\delta_{p+2}}^{-1})(\gamma)
\\
&=&
(t_{\gamma}(\delta_{p+2}))^{q+2}
\cdot\gamma\cdot
(t_{\delta_{p+2}}^{-1})(\gamma)
\\
&=&
(t_{\gamma}(\delta_{p+2}))^{q+3}
\cdot (t_{\delta_{p+2}}^{-1}\cdot t_{\gamma}^{-1}\cdot t_{\delta_{p+2}})(\gamma)
\phantom{000}
(\because (t_{\delta_{p+2}}^{-1})(\gamma)
=t_{\gamma}(\delta_{p+2}))
\\
\end{eqnarray*}
Back to the global monodromy, we have
\begin{eqnarray*}
&&\beta_2\cdots\beta_{p+2}\cdot\beta^{-1}\cdot\alpha_1\cdots\alpha_{r+2}\cdot\gamma_{r+2}^{-1}\cdot\gamma_{r+4}^{r+1}\cdot\gamma_{r+5}^p\cdot (t_{\beta_1}(\gamma_{r+3}))^{-1}
\\
&&\cdot\beta\gamma_{r+2}
t_{\beta_1}(\alpha_{r+3})
(t_{\gamma}(\delta_{p+2}))^{q+3}
(t_{\delta_{p+2}}^{-1}\cdot t_{\gamma}^{-1}\cdot t_{\delta_{p+2}})(\gamma)
\alpha_{r+4}
(t_{\alpha_{r+4}}^{-1})(\beta_1)
\\
&&
\cdot\alpha_{r+5}\gamma_{r+5}\alpha_{r+6}\cdots\alpha_{r+q+7}\gamma_{r+q+7}
\\
&=&
\beta_2\cdots\beta_{p+2}\cdot\beta^{-1}\cdot\alpha_1\cdots\alpha_{r+2}\cdot\gamma_{r+2}^{-1}\cdot\gamma_{r+4}^{r+1}\cdot\gamma_{r+5}^p\cdot (t_{\beta_1}(\gamma_{r+3}))^{-1}
\\
&&
\cdot\beta\gamma_{r+2}
t_{\beta_1}(\alpha_{r+3})
(t_{\gamma}(\delta_{p+2}))^{q+3}
\alpha_{r+4}
(t_{\alpha_{r+4}}^{-1}\cdot t_{\delta_{p+2}}^{-1}\cdot t_{\gamma}^{-1}\cdot t_{\delta_{p+2}})(\gamma)
\cdot
(t_{\alpha_{r+4}}^{-1})(\beta_1)
\\
&&
\cdot\alpha_{r+5}\gamma_{r+5}\alpha_{r+6}\cdots\alpha_{r+q+7}\gamma_{r+q+7}
\\
&=&
\beta_2\cdots\beta_{p+2}\cdot\beta^{-1}\cdot\alpha_1\cdots\alpha_{r+2}\cdot\gamma_{r+2}^{-1}\cdot\gamma_{r+4}^{r+1}\cdot\gamma_{r+5}^p\cdot (t_{\beta_1}(\gamma_{r+3}))^{-1}
\\
&&
\cdot\gamma_{r+2}
\cdot
t_{\beta_1}(\alpha_{r+3})
\cdot
\underline{
\beta
\cdot
(t_{\gamma}(\delta_{p+2}))^{q+3}
\alpha_{r+4}
\alpha_{r+5}\gamma_{r+5}\alpha_{r+6}\cdots\alpha_{r+q+7}\gamma_{r+q+7}
}
\\
&&
\cdot
(t_{\alpha_{r+4}}^{-1}\cdot t_{\delta_{p+2}}^{-1}\cdot t_{\gamma}^{-1}\cdot t_{\delta_{p+2}})(\gamma)
\cdot
(t_{\alpha_{r+4}}^{-1})(\beta_1)
\end{eqnarray*}
The underlined part can be seen an embedding of the linear plumbing $L_q$ in Figure~\ref{egfamily2}: Let $L_q$ be a linear plumbing and $\Sigma_{L_q}$ be a generic fiber for $X_{L_q}$ as in Figure~\ref{egfamily2}. Then the monodromy for $X_{L_q}$ can be written as
$$
a_1^{q+3}a_2 a_3a_4b_4 a_5\cdots a_{q+6} b_{q+6}
$$
where $a_i$ is simple closed curve in $\Sigma_{L_q}$ enclosing $i$th hole and $b_j$ is simple closed curve in $\Sigma_{L_q}$ enclosing from the first to $i$th holes. Then there is a planar subsurface of $\Sigma_{\Gamma_{p,q,r}}$ which is diffeomorphic to $\Sigma_{L_q}$ so that the image of each curves are
\begin{eqnarray*}
a_1 &\rightarrow&  t_{\gamma}(\delta_{p+2}) %\phantom{0000}(i=1,\dots, q+2)
\\
a_2 &\rightarrow& \beta
\\
a_{i} &\rightarrow& \alpha_{r+i+1}
\phantom{0000}(i=3,\dots, q+6)
\\
b_{4} &\rightarrow& \gamma_{r+5}
\\
b_{q+6} &\rightarrow& \gamma_{r+q+7}
\end{eqnarray*}
\begin{figure}[h]
\begin{tikzpicture}[scale=0.15]
\begin{scope}

\draw (7,0)--(20,0) (21.5,0)--(24.5,0) (26,0)--(39,0);
\draw (7,-6)--(39,-6);
\draw (7,0) arc (90:270:3);
\draw (39,0) arc (90:-90:3);
\draw (8,-3) circle (1);
%\draw[red] (8,-3) circle (1.5);
\draw (16,-3) circle (1);
%\draw[red] (12,-3) circle (1.5);

\draw (12,-3) node{$\cdots$};
\draw (20,-2) arc (90:270:1);
\draw (21.5,-2) arc (90:-90:1);

%\draw[red] (20,-1.5) arc (1.5);
\draw (21.5,-2) arc (90:-90:1);
%vertical part
\draw (20,-4)--(20,14.5);
\draw (21.5,-4)--(21.5,2.5);
%\draw[red] (20.75,-4)--(20.75,8.5);
%\draw[red] (20.75,8.5) arc (180:0:2.25 and 1.75);
%\draw[red] (20.75,-4)--(23,-5.4)--(25.25,-4);
%\draw[red] (20.75,-4) arc (180:360:2.25 and 1.25);

\draw (21.5,2.5) arc (180:0:1.5);
\draw (20,14.5) arc (180:0:3);

\draw[red] (20.5,-4)--(20.5,14.5);
\draw[red] (21,-4.5)--(21,4);
\draw[red] (21,4) to [out=up, in=down] (25.5,5.5)-- (25.5,14.5);
\draw[red] (20.5,14.5) to [out=up, in=left] (23,16.5)
to [out=right,in=up] (25.5,14.5);
\draw[red](21,-4.5) to [out=down,in=down] (23,-3)
to [out=up,in=right] (21.5,-1)(8,-1)--(20,-1);
\draw[red] (20.5,-4) to [out=down,in=right] (19.5,-4.8)--(8,-4.8);
\draw[red] (8,-1) to[out=left,in=up] (6,-3)
to [out=down,in=left](8,-4.8);

%\draw[red] (25.,6.5)--(25.,14.5);
%\draw[red] (25.,14.5) arc (0:180:2);
%\draw[red] (25.,6.5) arc (0:-180:2);
\draw[red] (20,10.5) node[left]{$t_{\gamma}(\delta_{p+2})$};

\draw (23, 6.5) circle (1);

\draw (23,11.) node{$\vdots$};

\draw (23, 14.5) circle (1);

\draw (23,18) node[]{};
\draw (24+2,-2) arc (90:-90:1);
\draw (24.5,-2) arc (90:270:1);

\draw (26,-4)--(26,14.5);
\draw (24.5,-4)--(24.5,2.5);
\draw [decorate,decoration={brace,mirror,amplitude=5pt},xshift=0pt,yshift=0]
	(26.5,6.5) -- (26.5,14.5) node [black,midway,xshift=20pt,align=center] 
	{\footnotesize $p+1$};

%vertical part 끝

\draw (30,-3) circle (1);
\draw (34,-3) node{$\cdots$};
\draw (38,-3) circle (1);

%\draw[blue] (20,-1.5)--(8,-1.5);
%\draw[blue] (21.5,-4.5)--(8,-4.5);
%\draw[blue] (8,-1.5) arc (90:270:1.5);
%\draw[blue] (21.5,-1.5) arc (90:-90:1.5);
\draw [decorate,decoration={brace,mirror,amplitude=5pt},xshift=0pt,yshift=0]
	(8,-6.5) -- (16,-6.5) node [black,midway,yshift=-15pt] 
	{\footnotesize $r+2$};
\draw [decorate,decoration={brace,mirror,amplitude=5pt},xshift=0pt,yshift=0]
	(30,-6.5) -- (38,-6.5) node [black,midway,yshift=-15pt] 
	{\footnotesize $q+3$};
\end{scope}

\end{tikzpicture}
\caption{$t_{\gamma}(\delta_{p+2})$ in $\Sigma_{\gamma_{p,q,r}}$}
\label{egfamily1}
\end{figure}
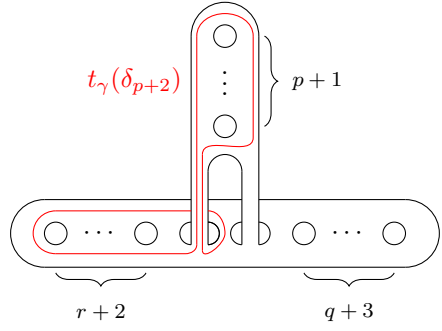
\begin{figure}[h]
\begin{tikzpicture}
\draw (0,0) node[above]{$-2$};
\draw (1,0) node[]{$\cdots$};
\draw (2,0) node[above]{$-2$};
\draw (3,0) node[above]{$-5$};
\draw (4,0) node[below]{$-(q+4)$};

\draw (0,0)--(.5,0) (1.5,0)--(4,0);
\draw (0,0 ) node[circle, fill, inner sep=1.5pt, black]{};
\draw (2,0 ) node[circle, fill, inner sep=1.5pt, black]{};

\draw (3,0 ) node[circle, fill, inner sep=1.5pt, black]{};
\draw (4,0 ) node[circle, fill, inner sep=1.5pt, black]{};
\draw (0,-.5) node{};

\draw [decorate,decoration={brace,mirror,amplitude=5pt},xshift=0pt,yshift=-5]
	(0,0) -- (2,0) node [black,midway,yshift=-15pt] 
	{\footnotesize $q+2$};

\draw (0,-1.5) node[]{};

\end{tikzpicture}
\begin{tikzpicture}
\begin{scope}[shift={(6,.75)},scale=0.17]
\draw (7,0)--(34,0);
\draw (7,-6)--(34,-6);
\draw (7,0) arc (90:270:3);
\draw (34,0) arc (90:-90:3);
\draw (8,-3) circle (1);
%\draw[red] (8,-3) circle (1.5);
\draw (12,-3) circle (1);
%\draw[red] (12,-3) circle (1.5);

\draw (16,-3) circle (1);
\draw (20,-3) circle (1);
%\draw[red] (20,-3) circle (1.5);

\draw (25,-3) circle (1);
\draw (29,-3) node[]{$\cdots$};

\draw (33,-3) circle (1);

\draw[red] (20,-1)--(8,-1);
\draw[red] (20,-5)--(8,-5);
\draw[red] (8,-1) arc (90:270:2);
\draw[red] (20,-1) arc (90:-90:2);
\draw[red] (23,-5.) node {\footnotesize{$b_{4}$}};

\draw [decorate,decoration={brace,mirror,amplitude=5pt},xshift=0pt,yshift=0]
	(25,-6.5) -- (33,-6.5) node [black,midway,yshift=-15pt] 
	{\footnotesize $q+2$};

\end{scope}

\end{tikzpicture}
\caption{Linear plumbing $L_q$ and the generic fiber $\Sigma_{L_q}$ for $X_{L_q}$}
\label{egfamily2}
\end{figure}
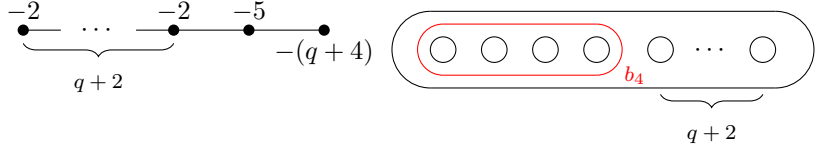
The linear plumbing $L_q$ can be rationally blowdown and the relation in the mapping class group of $\Sigma_{L_q}$ for the rational blowdown was given by Endo-Mark-Van Horn-Morris in 
~\cite{MR2783383}.
\begin{eqnarray*}
&\phantom{0}& \hspace{-2 em} a_1^{q+3}a_2 a_3a_4b_4 a_5\cdots a_{q+6} b_{q+6}= y_1y_2\cdots y_{q+6}
\end{eqnarray*}
\begin{figure}[h]
\begin{tikzpicture}[scale=0.13]
\begin{scope}
\draw (7,0)--(34,0);
\draw (7,-6)--(34,-6);
\draw (7,0) arc (90:270:3);
\draw (34,0) arc (90:-90:3);
\draw (8,-3) circle (1);
%\draw[red] (8,-3) circle (1.5);
\draw (12,-3) circle (1);
%\draw[red] (12,-3) circle (1.5);

\draw (16,-3) circle (1);
\draw (20,-3) circle (1);
%\draw[red] (20,-3) circle (1.5);

\draw (25,-3) circle (1);
\draw (29,-3) node[]{$\cdots$};

\draw (33,-3) circle (1);

\draw[red] (16,-1)--(8,-1);
\draw[red] (16,-5)--(8,-5);
\draw[red] (8,-1) arc (90:270:2);
\draw[red] (16,-1) arc (90:-90:2);
\draw[red] (14,2.) node {{$y_1$}};

\draw [decorate,decoration={brace,mirror,amplitude=5pt},xshift=0pt,yshift=0]
	(25,-6.5) -- (33,-6.5) node [black,midway,yshift=-15pt] 
	{\footnotesize $q+2$};

\end{scope}

\begin{scope}[shift={(35,0)}]

\draw (7,0)--(20,0) (21.5,0)--(24.5,0) (26,0)--(39,0);
\draw (7,-6)--(39,-6);
\draw (7,0) arc (90:270:3);
\draw (39,0) arc (90:-90:3);
\draw (8,-3) circle (1);
%\draw[red] (8,-3) circle (1.5);
\draw (16,-3) circle (1);
%\draw[red] (12,-3) circle (1.5);

\draw (12,-3) node{$\cdots$};
\draw (20,-2) arc (90:270:1);
\draw (21.5,-2) arc (90:-90:1);

%\draw[red] (20,-1.5) arc (1.5);
\draw (21.5,-2) arc (90:-90:1);
%vertical part
\draw (20,-4)--(20,14.5);
\draw (21.5,-4)--(21.5,2.5);
%\draw[red] (20.75,-4)--(20.75,8.5);
%\draw[red] (20.75,8.5) arc (180:0:2.25 and 1.75);
%\draw[red] (20.75,-4)--(23,-5.4)--(25.25,-4);
%\draw[red] (20.75,-4) arc (180:360:2.25 and 1.25);

\draw (21.5,2.5) arc (180:0:1.5);
\draw (20,14.5) arc (180:0:3);

\draw[red](25.5,-4.8) to [out=right,in=down] (27.5,-3)
to [out=up,in=right] (26,-1)(8,-1)--(20,-1) (21.5,-1)--(24.5,-1);
\draw[red] (25.5,-4.8)--(8,-4.8);
\draw[red] (8,-1) to[out=left,in=up] (6,-3)
to [out=down,in=left](8,-4.8);

%\draw[red] (25.,6.5)--(25.,14.5);
%\draw[red] (25.,14.5) arc (0:180:2);
%\draw[red] (25.,6.5) arc (0:-180:2);
\draw[red] (17,2) node[left]{$Y_{p,1,r}$};

\draw (23, 6.5) circle (1);

\draw (23,11.) node{$\vdots$};

\draw (23, 14.5) circle (1);

\draw (23,18) node[]{};
\draw (24+2,-2) arc (90:-90:1);
\draw (24.5,-2) arc (90:270:1);

\draw (26,-4)--(26,14.5);
\draw (24.5,-4)--(24.5,2.5);
\draw [decorate,decoration={brace,mirror,amplitude=5pt},xshift=0pt,yshift=0]
	(26.5,6.5) -- (26.5,14.5) node [black,midway,xshift=20pt,align=center] 
	{\footnotesize $p+1$};

%vertical part 끝

\draw (30,-3) circle (1);
\draw (34,-3) node{$\cdots$};
\draw (38,-3) circle (1);

%\draw[blue] (20,-1.5)--(8,-1.5);
%\draw[blue] (21.5,-4.5)--(8,-4.5);
%\draw[blue] (8,-1.5) arc (90:270:1.5);
%\draw[blue] (21.5,-1.5) arc (90:-90:1.5);
\draw [decorate,decoration={brace,mirror,amplitude=5pt},xshift=0pt,yshift=0]
	(8,-6.5) -- (16,-6.5) node [black,midway,yshift=-15pt] 
	{\footnotesize $r+2$};
\draw [decorate,decoration={brace,mirror,amplitude=5pt},xshift=0pt,yshift=0]
	(30,-6.5) -- (38,-6.5) node [black,midway,yshift=-15pt] 
	{\footnotesize $q+3$};
\end{scope}

\end{tikzpicture}
\caption{$y_1$ and $Y_{p,1,r}=(t_{\alpha_{r+4}}^{-1}\cdot t_{\beta_1}^{-1}\cdot t_{\alpha_{r+4}})(Y_{p,1,r})$}
\label{egfamily3}

\end{figure}
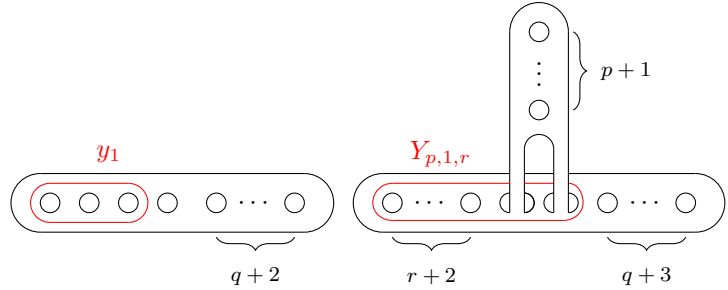
\begin{figure}[h]
\begin{tikzpicture}[scale=0.13]
\begin{scope}
\draw (7,0)--(34,0);
\draw (7,-6)--(34,-6);
\draw (7,0) arc (90:270:3);
\draw (34,0) arc (90:-90:3);
\draw (8,-3) circle (1);
%\draw[red] (8,-3) circle (1.5);
\draw (12,-3) circle (1);
%\draw[red] (12,-3) circle (1.5);

\draw (16,-3) circle (1);
\draw (20,-3) circle (1);
%\draw[red] (20,-3) circle (1.5);

\draw (25,-3) circle (1);
\draw (29,-3) node[]{$\cdots$};

\draw (33,-3) circle (1);

\draw[red] (33,-1)--(12,-1);
\draw[red] (33,-5)--(12,-5);
\draw[red] (12,-1) arc (90:270:2);
\draw[red] (33,-1) arc (90:-90:2);
\draw[red] (14,2.) node {{$y_2$}};

\draw [decorate,decoration={brace,mirror,amplitude=5pt},xshift=0pt,yshift=0]
	(25,-6.5) -- (33,-6.5) node [black,midway,yshift=-15pt] 
	{\footnotesize $q+2$};

\end{scope}

\begin{scope}[shift={(35,0)}]

\draw (7,0)--(20,0) (21.5,0)--(24.5,0) (26,0)--(39,0);
\draw (7,-6)--(39,-6);
\draw (7,0) arc (90:270:3);
\draw (39,0) arc (90:-90:3);
\draw (8,-3) circle (1);
%\draw[red] (8,-3) circle (1.5);
\draw (16,-3) circle (1);
%\draw[red] (12,-3) circle (1.5);

\draw (12,-3) node{$\cdots$};
\draw (20,-2) arc (90:270:1);
\draw (21.5,-2) arc (90:-90:1);

%\draw[red] (20,-1.5) arc (1.5);
\draw (21.5,-2) arc (90:-90:1);
%vertical part
\draw (20,-4)--(20,14.5);
\draw (21.5,-4)--(21.5,2.5);
%\draw[red] (20.75,-4)--(20.75,8.5);
%\draw[red] (20.75,8.5) arc (180:0:2.25 and 1.75);
%\draw[red] (20.75,-4)--(23,-5.4)--(25.25,-4);
%\draw[red] (20.75,-4) arc (180:360:2.25 and 1.25);

\draw (21.5,2.5) arc (180:0:1.5);
\draw (20,14.5) arc (180:0:3);

\draw[red] (20.5,-4)--(20.5,14.5);
\draw[red] (21,-4.5)--(21,4);
\draw[red] (21,4) to [out=up, in=down] (25.5,5.5)-- (25.5,14.5);
\draw[red] (20.5,14.5) to [out=up, in=left] (23,16.5)
to [out=right,in=up] (25.5,14.5);
\draw[red](21,-4.5) to [out=down,in=down] (23,-3)
to [out=up,in=left] (24.5,-1) (26,-1)--(38,-1);
\draw[red] (20.5,-4) to [out=down,in=left] (21.,-5.5)--(38,-5.5);
\draw[red] (38,-1) arc(90:-90:2.25);

%\draw[red] (25.,6.5)--(25.,14.5);
%\draw[red] (25.,14.5) arc (0:180:2);
%\draw[red] (25.,6.5) arc (0:-180:2);

%\draw[red] (25.,6.5)--(25.,14.5);
%\draw[red] (25.,14.5) arc (0:180:2);
%\draw[red] (25.,6.5) arc (0:-180:2);
\draw[red] (17,2) node[left]{$Y_{p,2,r}$};

\draw (23, 6.5) circle (1);

\draw (23,11.) node{$\vdots$};

\draw (23, 14.5) circle (1);

\draw (23,18) node[]{};
\draw (24+2,-2) arc (90:-90:1);
\draw (24.5,-2) arc (90:270:1);

\draw (26,-4)--(26,14.5);
\draw (24.5,-4)--(24.5,2.5);
\draw [decorate,decoration={brace,mirror,amplitude=5pt},xshift=0pt,yshift=0]
	(26.5,6.5) -- (26.5,14.5) node [black,midway,xshift=20pt,align=center] 
	{\footnotesize $p+1$};

%vertical part 끝

\draw (30,-3) circle (1);
\draw (34,-3) node{$\cdots$};
\draw (38,-3) circle (1);

%\draw[blue] (20,-1.5)--(8,-1.5);
%\draw[blue] (21.5,-4.5)--(8,-4.5);
%\draw[blue] (8,-1.5) arc (90:270:1.5);
%\draw[blue] (21.5,-1.5) arc (90:-90:1.5);
\draw [decorate,decoration={brace,mirror,amplitude=5pt},xshift=0pt,yshift=0]
	(8,-6.5) -- (16,-6.5) node [black,midway,yshift=-15pt] 
	{\footnotesize $r+2$};
\draw [decorate,decoration={brace,mirror,amplitude=5pt},xshift=0pt,yshift=0]
	(30,-6.5) -- (38,-6.5) node [black,midway,yshift=-15pt] 
	{\footnotesize $q+3$};
\end{scope}

\end{tikzpicture}
\caption{$y_2$ and $Y_{p,2,r}$}
\label{egfamily4}
\end{figure}
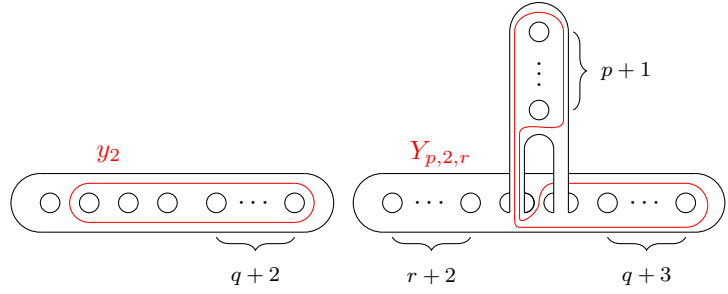
\begin{figure}[h]
\begin{tikzpicture}[scale=0.13]
\begin{scope}
\draw (7,0)--(34,0);
\draw (7,-6)--(34,-6);
\draw (7,0) arc (90:270:3);
\draw (34,0) arc (90:-90:3);
\draw (8,-3) circle (1);
%\draw[red] (8,-3) circle (1.5);
\draw (12,-3) circle (1);
%\draw[red] (12,-3) circle (1.5);

\draw (16,-3) circle (1);
\draw (20,-3) circle (1);
%\draw[red] (20,-3) circle (1.5);

\draw (25,-3) circle (1);
\draw (29,-3) node[]{$\cdots$};

\draw (33,-3) circle (1);

\draw[red] (33,-1)--(8,-1);
\draw[red] (8,-1) arc (90:270:2);
\draw[red] (33,-1) arc (90:-90:2);
\draw[red] (8,-5) arc (270:360:2);
\draw[red] (33,-5) arc (270:180:2);
\draw[red] (10,-3) to [out=up,in=left] (12,-1.5)--(29,-1.5);
\draw[red] (31,-3) to [out=up,in=right] (29,-1.5);

\draw[red] (14,2.) node {{$y_3$}};

\draw [decorate,decoration={brace,mirror,amplitude=5pt},xshift=0pt,yshift=0]
	(25,-6.5) -- (33,-6.5) node [black,midway,yshift=-15pt] 
	{\footnotesize $q+2$};

\end{scope}

\begin{scope}[shift={(35,0)}]

\draw (7,0)--(20,0) (21.5,0)--(24.5,0) (26,0)--(39,0);
\draw (7,-6)--(39,-6);
\draw (7,0) arc (90:270:3);
\draw (39,0) arc (90:-90:3);
\draw (8,-3) circle (1);
%\draw[red] (8,-3) circle (1.5);
\draw (16,-3) circle (1);
%\draw[red] (12,-3) circle (1.5);

\draw (12,-3) node{$\cdots$};
\draw (20,-2) arc (90:270:1);
\draw (21.5,-2) arc (90:-90:1);

%\draw[red] (20,-1.5) arc (1.5);
\draw (21.5,-2) arc (90:-90:1);
%vertical part
\draw (20,-4)--(20,14.5);
\draw (21.5,-4)--(21.5,2.5);
%\draw[red] (20.75,-4)--(20.75,8.5);
%\draw[red] (20.75,8.5) arc (180:0:2.25 and 1.75);
%\draw[red] (20.75,-4)--(23,-5.4)--(25.25,-4);
%\draw[red] (20.75,-4) arc (180:360:2.25 and 1.25);

\draw (21.5,2.5) arc (180:0:1.5);
\draw (20,14.5) arc (180:0:3);

\draw[red] (20.5,-4)--(20.5,14.5);
\draw[red] (21,-4.5)--(21,4);
\draw[red] (21,4) to [out=up, in=down] (25.5,5.5)-- (25.5,14.5);
\draw[red] (20.5,14.5) to [out=up, in=left] (23,16.5)
to [out=right,in=up] (25.5,14.5);
\draw[red](21,-4.5) to [out=down,in=down] (23,-3)
to [out=up,in=left] (24.5,-1)(8,-.5)--(20,-.5)
(21.5,-.5)--(24.5,-.5)
(26,-.5)--(38,-.5)
;
\draw[red] (38,-.5) arc (90:-180:2.25);
\draw[red] (35.75,-2.75) to [out=up,in=right] (34.5,-1)--(26,-1);
\draw[red] (20.5,-4) to [out=down,in=right] (19.5,-4.8)--(8,-4.8);
\draw[red] (8,-.5) to[out=left,in=up] (6,-3)
to [out=down,in=left](8,-4.8);

%\draw[red] (25.,6.5)--(25.,14.5);
%\draw[red] (25.,14.5) arc (0:180:2);
%\draw[red] (25.,6.5) arc (0:-180:2);
\draw[red] (17,2) node[left]{$Y_{p,3,r}$};

\draw (23, 6.5) circle (1);

\draw (23,11.) node{$\vdots$};

\draw (23, 14.5) circle (1);

\draw (23,18) node[]{};
\draw (24+2,-2) arc (90:-90:1);
\draw (24.5,-2) arc (90:270:1);

\draw (26,-4)--(26,14.5);
\draw (24.5,-4)--(24.5,2.5);
\draw [decorate,decoration={brace,mirror,amplitude=5pt},xshift=0pt,yshift=0]
	(26.5,6.5) -- (26.5,14.5) node [black,midway,xshift=20pt,align=center] 
	{\footnotesize $p+1$};

%vertical part 끝

\draw (30,-3) circle (1);
\draw (34,-3) node{$\cdots$};
\draw (38,-3) circle (1);

%\draw[blue] (20,-1.5)--(8,-1.5);
%\draw[blue] (21.5,-4.5)--(8,-4.5);
%\draw[blue] (8,-1.5) arc (90:270:1.5);
%\draw[blue] (21.5,-1.5) arc (90:-90:1.5);
\draw [decorate,decoration={brace,mirror,amplitude=5pt},xshift=0pt,yshift=0]
	(8,-6.5) -- (16,-6.5) node [black,midway,yshift=-15pt] 
	{\footnotesize $r+2$};
\draw [decorate,decoration={brace,mirror,amplitude=5pt},xshift=0pt,yshift=0]
	(30,-6.5) -- (38,-6.5) node [black,midway,yshift=-15pt] 
	{\footnotesize $q+3$};
\end{scope}

\end{tikzpicture}
\begin{tikzpicture}[scale=0.13]
\begin{scope}
\draw (7,0)--(34,0);
\draw (7,-6)--(34,-6);
\draw (7,0) arc (90:270:3);
\draw (34,0) arc (90:-90:3);
\draw (8,-3) circle (1);
%\draw[red] (8,-3) circle (1.5);
\draw (12,-3) circle (1);
%\draw[red] (12,-3) circle (1.5);

\draw (16,-3) circle (1);
\draw (20,-3) circle (1);
%\draw[red] (20,-3) circle (1.5);

\draw (25,-3) circle (1);
\draw (29,-3) node[]{$\cdots$};

\draw (33,-3) circle (1);

\draw[red] (25,-1)--(8,-1);
\draw[red] (8,-1) arc (90:270:2);
\draw[red] (25,-1) arc (90:-90:2);
\draw[red] (8,-5) arc (270:360:2);
\draw[red] (25,-5) arc (270:180:2);
\draw[red] (10,-3) to [out=up,in=left] (12,-1.5)--(21,-1.5);
\draw[red] (23,-3) to [out=up,in=right] (21,-1.5);

\draw[red] (14,2.) node {{$y_{q+4}$}};

\draw [decorate,decoration={brace,mirror,amplitude=5pt},xshift=0pt,yshift=0]
	(25,-6.5) -- (33,-6.5) node [black,midway,yshift=-15pt] 
	{\footnotesize $q+2$};

\end{scope}

\begin{scope}[shift={(35,0)}]

\draw (3,12) node[]{$\vdots$};

\draw (7,0)--(20,0) (21.5,0)--(24.5,0) (26,0)--(39,0);
\draw (7,-6)--(39,-6);
\draw (7,0) arc (90:270:3);
\draw (39,0) arc (90:-90:3);
\draw (8,-3) circle (1);
%\draw[red] (8,-3) circle (1.5);
\draw (16,-3) circle (1);
%\draw[red] (12,-3) circle (1.5);

\draw (12,-3) node{$\cdots$};
\draw (20,-2) arc (90:270:1);
\draw (21.5,-2) arc (90:-90:1);

%\draw[red] (20,-1.5) arc (1.5);
\draw (21.5,-2) arc (90:-90:1);
%vertical part
\draw (20,-4)--(20,14.5);
\draw (21.5,-4)--(21.5,2.5);
%\draw[red] (20.75,-4)--(20.75,8.5);
%\draw[red] (20.75,8.5) arc (180:0:2.25 and 1.75);
%\draw[red] (20.75,-4)--(23,-5.4)--(25.25,-4);
%\draw[red] (20.75,-4) arc (180:360:2.25 and 1.25);

\draw (21.5,2.5) arc (180:0:1.5);
\draw (20,14.5) arc (180:0:3);

\draw[red] (20.5,-4)--(20.5,14.5);
\draw[red] (21,-4.5)--(21,4);
\draw[red] (21,4) to [out=up, in=down] (25.5,5.5)-- (25.5,14.5);
\draw[red] (20.5,14.5) to [out=up, in=left] (23,16.5)
to [out=right,in=up] (25.5,14.5);
\draw[red](21,-4.5) to [out=down,in=down] (23,-3)
to [out=up,in=left] (24.5,-1)(8,-.5)--(20,-.5)
(21.5,-.5)--(24.5,-.5)
(26,-.5)--(34,-.5)
;
\draw[red] (34,-.5) arc (90:-180:2.25);
\draw[red] (31.75,-2.75) to [out=up,in=right] (30.5,-1)--(26,-1);
\draw[red] (20.5,-4) to [out=down,in=right] (19.5,-4.8)--(8,-4.8);
\draw[red] (8,-.5) to[out=left,in=up] (6,-3)
to [out=down,in=left](8,-4.8);

%\draw[red] (25.,6.5)--(25.,14.5);
%\draw[red] (25.,14.5) arc (0:180:2);
%\draw[red] (25.,6.5) arc (0:-180:2);
\draw[red] (17,2) node[left]{$Y_{p,q+4,r}$};

\draw (23, 6.5) circle (1);

\draw (23,11.) node{$\vdots$};

\draw (23, 14.5) circle (1);

\draw (23,18) node[]{};
\draw (24+2,-2) arc (90:-90:1);
\draw (24.5,-2) arc (90:270:1);

\draw (26,-4)--(26,14.5);
\draw (24.5,-4)--(24.5,2.5);
\draw [decorate,decoration={brace,mirror,amplitude=5pt},xshift=0pt,yshift=0]
	(26.5,6.5) -- (26.5,14.5) node [black,midway,xshift=20pt,align=center] 
	{\footnotesize $p+1$};

%vertical part 끝

\draw (30,-3) circle (1);
\draw (34,-3) circle (1);
\draw (37.5, -3) node[circle,fill, inner sep=.5pt,black]{};
\draw (38., -3) node[circle,fill, inner sep=.5pt,black]{};
\draw (37., -3) node[circle,fill, inner sep=.5pt,black]{};

\draw (40,-3) circle (1);

%\draw[blue] (20,-1.5)--(8,-1.5);
%\draw[blue] (21.5,-4.5)--(8,-4.5);
%\draw[blue] (8,-1.5) arc (90:270:1.5);
%\draw[blue] (21.5,-1.5) arc (90:-90:1.5);
\draw [decorate,decoration={brace,mirror,amplitude=5pt},xshift=0pt,yshift=0]
	(8,-6.5) -- (16,-6.5) node [black,midway,yshift=-15pt] 
	{\footnotesize $r+2$};
\draw [decorate,decoration={brace,mirror,amplitude=5pt},xshift=0pt,yshift=0]
	(34,-6.5) -- (40,-6.5) node [black,midway,yshift=-15pt] 
	{\footnotesize $q+2$};
\end{scope}

\end{tikzpicture}
\caption{$y_i$ and $Y_{p,i,r}$ for $(i=3,\dots, q+4)$} 
\label{egfamily5}
\end{figure}
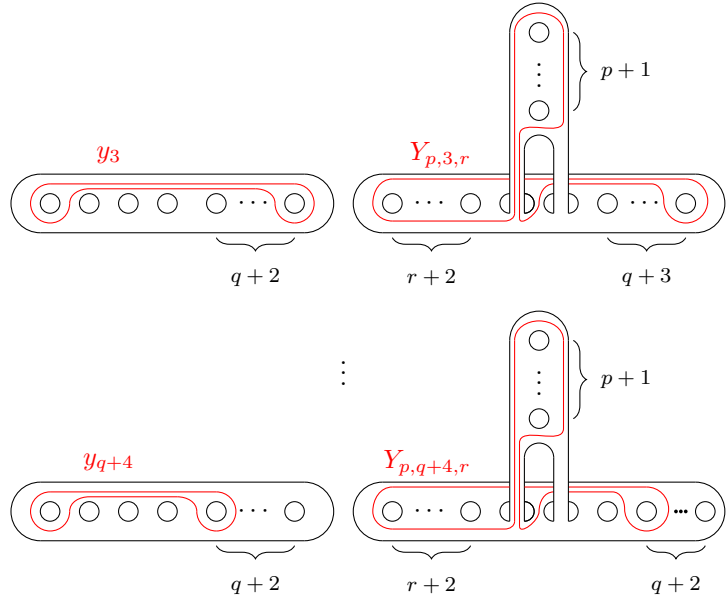
\begin{figure}[h]
\begin{tikzpicture}[scale=0.13]
\begin{scope}
\draw (7,0)--(34,0);
\draw (7,-6)--(34,-6);
\draw (7,0) arc (90:270:3);
\draw (34,0) arc (90:-90:3);
\draw (8,-3) circle (1);
%\draw[red] (8,-3) circle (1.5);
\draw (12,-3) circle (1);
%\draw[red] (12,-3) circle (1.5);

\draw (16,-3) circle (1);
\draw (20,-3) circle (1);
%\draw[red] (20,-3) circle (1.5);

\draw (25,-3) circle (1);
\draw (29,-3) node[]{$\cdots$};

\draw (33,-3) circle (1);

\draw[red] (20,-1)--(8,-1);
\draw[red] (10,-3)to[out=up,in=left](12,-1.5)to[out=right,in=up](14,-3);
\draw[red] (8,-1) arc (90:360:2);
\draw[red] (20,-1) arc (90:-90:2) (20,-5)--(16,-5);
\draw[red] (16,-5) arc (270:180:2);
\draw[red] (14,2.) node {{$y_{q+5}$}};

\draw [decorate,decoration={brace,mirror,amplitude=5pt},xshift=0pt,yshift=0]
	(25,-6.5) -- (33,-6.5) node [black,midway,yshift=-15pt] 
	{\footnotesize $q+2$};

\end{scope}

\begin{scope}[shift={(35,0)}]

\draw (7,0)--(20,0) (21.5,0)--(24.5,0) (26,0)--(39,0);
\draw (7,-6)--(39,-6);
\draw (7,0) arc (90:270:3);
\draw (39,0) arc (90:-90:3);
\draw (8,-3) circle (1);
%\draw[red] (8,-3) circle (1.5);
\draw (16,-3) circle (1);
%\draw[red] (12,-3) circle (1.5);

\draw (12,-3) node{$\cdots$};
\draw (20,-2) arc (90:270:1);
\draw (21.5,-2) arc (90:-90:1);

%\draw[red] (20,-1.5) arc (1.5);
\draw (21.5,-2) arc (90:-90:1);
%vertical part
\draw (20,-4)--(20,14.5);
\draw (21.5,-4)--(21.5,2.5);
%\draw[red] (20.75,-4)--(20.75,8.5);
%\draw[red] (20.75,8.5) arc (180:0:2.25 and 1.75);
%\draw[red] (20.75,-4)--(23,-5.4)--(25.25,-4);
%\draw[red] (20.75,-4) arc (180:360:2.25 and 1.25);

\draw (21.5,2.5) arc (180:0:1.5);
\draw (20,14.5) arc (180:0:3);

\draw[red] (20.5,-4)--(20.5,14.5);
\draw[red] (21,-4)--(21,4);
\draw[red] (21,4) to [out=up, in=down] (25.5,5.5)-- (25.5,14.5);
\draw[red] (20.5,14.5) to [out=up, in=left] (23,16.5)
to [out=right,in=up] (25.5,14.5);
\draw[red](8,-.5)--(20,-.5)
(21.5,-.5)--(24.5,-.5)
(26,-.5)--(30,-.5)
;
\draw[red] (30,-.5) arc(90:-90:2.15);
\draw[red] (20.5,-4) to [out=down,in=right] (19.5,-4.8)--(8,-4.8);
\draw[red] (21,-4) to [out=down,in=left](22,-4.8)--(30,-4.8);
\draw[red] (8,-.5) to[out=left,in=up] (6,-3)
to [out=down,in=left](8,-4.8);

%\draw[red] (25.,6.5)--(25.,14.5);
%\draw[red] (25.,14.5) arc (0:180:2);
%\draw[red] (25.,6.5) arc (0:-180:2);
\draw[red] (17,2) node[left]{$Y_{p,q+5,r}$};

%\draw[red] (25.,6.5)--(25.,14.5);
%\draw[red] (25.,14.5) arc (0:180:2);
%\draw[red] (25.,6.5) arc (0:-180:2);

\draw (23, 6.5) circle (1);

\draw (23,11.) node{$\vdots$};

\draw (23, 14.5) circle (1);

\draw (23,18) node[]{};
\draw (24+2,-2) arc (90:-90:1);
\draw (24.5,-2) arc (90:270:1);

\draw (26,-4)--(26,14.5);
\draw (24.5,-4)--(24.5,2.5);
\draw [decorate,decoration={brace,mirror,amplitude=5pt},xshift=0pt,yshift=0]
	(26.5,6.5) -- (26.5,14.5) node [black,midway,xshift=20pt,align=center] 
	{\footnotesize $p+1$};

%vertical part 끝

\draw (30,-3) circle (1);
\draw (34,-3) node{$\cdots$};
\draw (38,-3) circle (1);

%\draw[blue] (20,-1.5)--(8,-1.5);
%\draw[blue] (21.5,-4.5)--(8,-4.5);
%\draw[blue] (8,-1.5) arc (90:270:1.5);
%\draw[blue] (21.5,-1.5) arc (90:-90:1.5);
\draw [decorate,decoration={brace,mirror,amplitude=5pt},xshift=0pt,yshift=0]
	(8,-6.5) -- (16,-6.5) node [black,midway,yshift=-15pt] 
	{\footnotesize $r+2$};
\draw [decorate,decoration={brace,mirror,amplitude=5pt},xshift=0pt,yshift=0]
	(30,-6.5) -- (38,-6.5) node [black,midway,yshift=-15pt] 
	{\footnotesize $q+3$};
\end{scope}

\end{tikzpicture}
\caption{$y_{q+5}$ and $Y_{p,q+5,r}=(t_{\alpha_{r+4}}^{-1}\cdot t_{\beta_1}^{-1}\cdot t_{\alpha_{r+4}})(Y_{p,q+5,r})
$}
\label{egfamily6}

\end{figure}
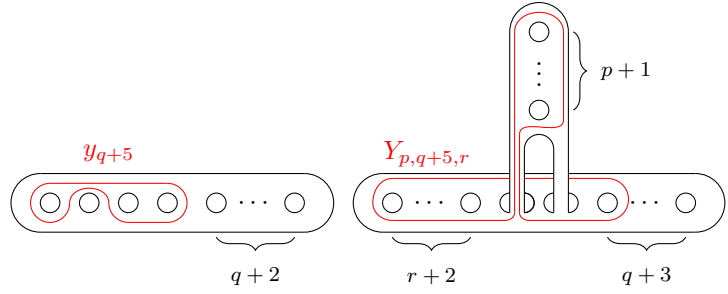
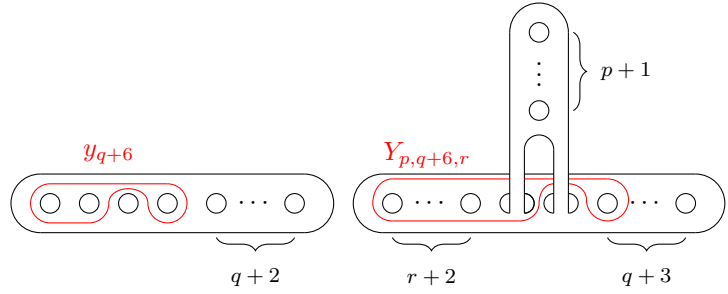
\begin{figure}[h]
\begin{tikzpicture}[scale=0.13]
\begin{scope}
\draw (7,0)--(34,0);
\draw (7,-6)--(34,-6);
\draw (7,0) arc (90:270:3);
\draw (34,0) arc (90:-90:3);
\draw (8,-3) circle (1);
%\draw[red] (8,-3) circle (1.5);
\draw (12,-3) circle (1);
%\draw[red] (12,-3) circle (1.5);

\draw (16,-3) circle (1);
\draw (20,-3) circle (1);
%\draw[red] (20,-3) circle (1.5);

\draw (25,-3) circle (1);
\draw (29,-3) node[]{$\cdots$};

\draw (33,-3) circle (1);

\draw[red] (20,-1)--(8,-1);
\draw[red] (14,-3)to[out=up,in=left](16,-1.5)to[out=right,in=up](18,-3);
\draw[red] (8,-1) arc (90:270:2);
\draw[red] (20,-1) arc (90:-180:2) (8,-5)--(12,-5);
\draw[red] (12,-5) arc (270:360:2);
\draw[red] (14,2.) node {{$y_{q+6}$}};

\draw [decorate,decoration={brace,mirror,amplitude=5pt},xshift=0pt,yshift=0]
	(25,-6.5) -- (33,-6.5) node [black,midway,yshift=-15pt] 
	{\footnotesize $q+2$};

\end{scope}

\begin{scope}[shift={(35,0)}]

\draw (7,0)--(20,0) (21.5,0)--(24.5,0) (26,0)--(39,0);
\draw (7,-6)--(39,-6);
\draw (7,0) arc (90:270:3);
\draw (39,0) arc (90:-90:3);
\draw (8,-3) circle (1);
%\draw[red] (8,-3) circle (1.5);
\draw (16,-3) circle (1);
%\draw[red] (12,-3) circle (1.5);

\draw (12,-3) node{$\cdots$};
\draw (20,-2) arc (90:270:1);
\draw (21.5,-2) arc (90:-90:1);

%\draw[red] (20,-1.5) arc (1.5);
\draw (21.5,-2) arc (90:-90:1);
%vertical part
\draw (20,-4)--(20,14.5);
\draw (21.5,-4)--(21.5,2.5);
%\draw[red] (20.75,-4)--(20.75,8.5);
%\draw[red] (20.75,8.5) arc (180:0:2.25 and 1.75);
%\draw[red] (20.75,-4)--(23,-5.4)--(25.25,-4);
%\draw[red] (20.75,-4) arc (180:360:2.25 and 1.25);

\draw (21.5,2.5) arc (180:0:1.5);
\draw (20,14.5) arc (180:0:3);

\draw[red](21,-4.8) to [out=right,in=down] (23,-3)
to [out=up,in=left] (24.5,-1)(8,-.5)--(20,-.5);
\draw[red] (21.,-4.8)--(8,-4.8);
\draw[red] (21.5,-.5)--(24.5,-.5);
\draw[red] (8,-.5) to[out=left,in=up] (6,-3)
to [out=down,in=left](8,-4.8);
\draw[red] (26,-.5)--(30,-.5);
\draw[red] (30,-.5) arc (90:-180:2.15);

\draw[red] (27.85,-2.65) to[out=up,in=right] (26,-1);
%\draw[red] (25.,6.5)--(25.,14.5);
%\draw[red] (25.,14.5) arc (0:180:2);
%\draw[red] (25.,6.5) arc (0:-180:2);

%\draw[red] (25.,6.5)--(25.,14.5);
%\draw[red] (25.,14.5) arc (0:180:2);
%\draw[red] (25.,6.5) arc (0:-180:2);
\draw[red] (17,2) node[left]{$Y_{p,q+6,r}$};

%\draw[red] (25.,6.5)--(25.,14.5);
%\draw[red] (25.,14.5) arc (0:180:2);
%\draw[red] (25.,6.5) arc (0:-180:2);

\draw (23, 6.5) circle (1);

\draw (23,11.) node{$\vdots$};

\draw (23, 14.5) circle (1);

\draw (23,18) node[]{};
\draw (24+2,-2) arc (90:-90:1);
\draw (24.5,-2) arc (90:270:1);

\draw (26,-4)--(26,14.5);
\draw (24.5,-4)--(24.5,2.5);
\draw [decorate,decoration={brace,mirror,amplitude=5pt},xshift=0pt,yshift=0]
	(26.5,6.5) -- (26.5,14.5) node [black,midway,xshift=20pt,align=center] 
	{\footnotesize $p+1$};

%vertical part 끝

\draw (30,-3) circle (1);
\draw (34,-3) node{$\cdots$};
\draw (38,-3) circle (1);

%\draw[blue] (20,-1.5)--(8,-1.5);
%\draw[blue] (21.5,-4.5)--(8,-4.5);
%\draw[blue] (8,-1.5) arc (90:270:1.5);
%\draw[blue] (21.5,-1.5) arc (90:-90:1.5);
\draw [decorate,decoration={brace,mirror,amplitude=5pt},xshift=0pt,yshift=0]
	(8,-6.5) -- (16,-6.5) node [black,midway,yshift=-15pt] 
	{\footnotesize $r+2$};
\draw [decorate,decoration={brace,mirror,amplitude=5pt},xshift=0pt,yshift=0]
	(30,-6.5) -- (38,-6.5) node [black,midway,yshift=-15pt] 
	{\footnotesize $q+3$};
\end{scope}

\end{tikzpicture}
\caption{$y_{q+6}$ and $Y_{p,q+6,r}$}
\label{egfamily7}

\end{figure}
Let $Y_{p,i,r}$ a simple closed curve in $\Sigma_{\Gamma_{p,q,r}}$ be the image of $y_i$ which can be drawn as in Figure~\ref{egfamily3} to Figure~\ref{egfamily7}. Then we have
\begin{eqnarray*}
&&
\beta_2\cdots\beta_{p+2}\cdot\beta^{-1}\cdot\alpha_1\cdots\alpha_{r+2}\cdot\gamma_{r+2}^{-1}\cdot\gamma_{r+4}^{r+1}\cdot\gamma_{r+5}^p\cdot (t_{\beta_1}(\gamma_{r+3}))^{-1}
\\
&&
\cdot\gamma_{r+2}
\cdot
t_{\beta_1}(\alpha_{r+3})
\cdot
\underline{
\beta
(t_{\gamma}(\delta_{p+2}))^{q+3}
\alpha_{r+4}
\alpha_{r+5}\gamma_{r+5}\alpha_{r+6}\cdots\alpha_{r+q+7}\gamma_{r+q+7}
}
\\
&&
\cdot
(t_{\alpha_{r+4}}^{-1}\cdot t_{\delta_{p+2}}^{-1}\cdot t_{\gamma}^{-1}\cdot t_{\delta_{p+2}})(\gamma)
\cdot
(t_{\alpha_{r+4}}^{-1})(\beta_1)
\\
&=&
\beta_2\cdots\beta_{p+2}\cdot\beta^{-1}\cdot\alpha_1\cdots\alpha_{r+2}\cdot\gamma_{r+2}^{-1}\cdot\gamma_{r+4}^{r+1}\cdot\gamma_{r+5}^p\cdot (t_{\beta_1}(\gamma_{r+3}))^{-1}
\\
&&
\cdot\gamma_{r+2}
\cdot
t_{\beta_1}(\alpha_{r+3})
\cdot
Y_{p,1,r}
\cdots
Y_{p,q+6,r}
\cdot
(t_{\alpha_{r+4}}^{-1}\cdot t_{\delta_{p+2}}^{-1}\cdot t_{\gamma}^{-1}\cdot t_{\delta_{p+2}})(\gamma)
\cdot
(t_{\alpha_{r+4}}^{-1})(\beta_1)
\\
&=&
\beta_2\cdots\beta_{p+2}\cdot\beta^{-1}\cdot\alpha_1\cdots\alpha_{r+2}\cdot\gamma_{r+2}^{-1}\cdot\gamma_{r+4}^{r+1}\cdot\gamma_{r+5}^p
\\
&&
\cdot
(t_{\beta_1}(\gamma_{r+3}))^{-1}
\cdot\gamma_{r+2}
\cdot
t_{\beta_1}(\alpha_{r+3})
\cdot
(t_{\alpha_{r+4}}^{-1})(\beta_1)
\cdot
(t_{\alpha_{r+4}}^{-1}\cdot t_{\beta_1}^{-1}\cdot t_{\alpha_{r+4}})(Y_{p,1,r})
\\
&&
(t_{\alpha_{r+4}}^{-1}\cdot t_{\beta_1}^{-1}\cdot t_{\alpha_{r+4}})(Y_{p,1,r})
\cdots
(t_{\alpha_{r+4}}^{-1}\cdot t_{\beta_1}^{-1}\cdot t_{\alpha_{r+4}})(Y_{p,q+6,r})
\\
&&
\cdot
(t_{\alpha_{r+4}}^{-1}\cdot t_{\beta_1}^{-1} \cdot t_{\delta_{p+2}}^{-1}\cdot t_{\gamma}^{-1}\cdot t_{\delta_{p+2}})(\gamma)
\end{eqnarray*}
Note that 
\begin{eqnarray*}
&&\hspace{-2em}
\gamma_{r+2}
\cdot
t_{\beta_1}(\alpha_{r+3})
\cdot
(t_{\alpha_{r+4}}^{-1})(\beta_1)
\cdot
(t_{\alpha_{r+4}}^{-1}\cdot t_{\beta_1}^{-1}\cdot t_{\alpha_{r+4}})(Y_{p,1,r})
\\
&=&
t_{\beta_1}(\gamma_{r+2})
\cdot
t_{\beta_1}(\alpha_{r+3})
\cdot
t_{\beta_1}(\alpha_{r+4})
\cdot
t_{\beta_1}(\gamma_{r+4})
\\
&=&
t_{\beta_1}(\gamma_{r+3})
\cdot
t_{\beta_1}(a)
\cdot
t_{\beta_1}(b)
\end{eqnarray*}
where $a$ and $b$ are simple closed curves as in Figure~\ref{egfamily8} due to the lantern relation.
\begin{figure}[h]
\begin{tikzpicture}[scale=0.15]
\begin{scope}

\draw (7,0)--(20,0) (21.5,0)--(24.5,0) (26,0)--(39,0);
\draw (7,-6)--(39,-6);
\draw (7,0) arc (90:270:3);
\draw (39,0) arc (90:-90:3);
\draw (8,-3) circle (1);
%\draw[red] (8,-3) circle (1.5);
\draw (16,-3) circle (1);
%\draw[red] (12,-3) circle (1.5);

\draw (12,-3) node{$\cdots$};
\draw (20,-2) arc (90:270:1);
\draw (21.5,-2) arc (90:-90:1);

%\draw[red] (20,-1.5) arc (1.5);
\draw (21.5,-2) arc (90:-90:1);
%vertical part
\draw (20,-4)--(20,14.5);
\draw (21.5,-4)--(21.5,2.5);
%\draw[red] (20.75,-4)--(20.75,8.5);
%\draw[red] (20.75,8.5) arc (180:0:2.25 and 1.75);
%\draw[red] (20.75,-4)--(23,-5.4)--(25.25,-4);
%\draw[red] (20.75,-4) arc (180:360:2.25 and 1.25);

\draw (21.5,2.5) arc (180:0:1.5);
\draw (20,14.5) arc (180:0:3);

\draw[red] (20,-1) arc (90:270:2);
\draw[red] (26,-1) arc (90:-90:2);

\draw[red] (21.5,-1)--(24.5,-1);
\draw[red] (20,-5)--(26,-5);
%\draw[red] (25.,6.5)--(25.,14.5);
%\draw[red] (25.,14.5) arc (0:180:2);
%\draw[red] (25.,6.5) arc (0:-180:2);
\draw[red] (19,1.5) node[left]{$a$};

\draw (23, 6.5) circle (1);

\draw (23,11.) node{$\vdots$};

\draw (23, 14.5) circle (1);

\draw (23,18) node[]{};
\draw (24+2,-2) arc (90:-90:1);
\draw (24.5,-2) arc (90:270:1);

\draw (26,-4)--(26,14.5);
\draw (24.5,-4)--(24.5,2.5);
\draw [decorate,decoration={brace,mirror,amplitude=5pt},xshift=0pt,yshift=0]
	(26.5,6.5) -- (26.5,14.5) node [black,midway,xshift=20pt,align=center] 
	{\footnotesize $p+1$};

%vertical part 끝

\draw (30,-3) circle (1);
\draw (34,-3) node{$\cdots$};
\draw (38,-3) circle (1);

%\draw[blue] (20,-1.5)--(8,-1.5);
%\draw[blue] (21.5,-4.5)--(8,-4.5);
%\draw[blue] (8,-1.5) arc (90:270:1.5);
%\draw[blue] (21.5,-1.5) arc (90:-90:1.5);
\draw [decorate,decoration={brace,mirror,amplitude=5pt},xshift=0pt,yshift=0]
	(8,-6.5) -- (16,-6.5) node [black,midway,yshift=-15pt] 
	{\footnotesize $r+2$};
\draw [decorate,decoration={brace,mirror,amplitude=5pt},xshift=0pt,yshift=0]
	(30,-6.5) -- (38,-6.5) node [black,midway,yshift=-15pt] 
	{\footnotesize $q+3$};
\end{scope}

\begin{scope}[shift={(39,0)}]

\draw (7,0)--(20,0) (21.5,0)--(24.5,0) (26,0)--(39,0);
\draw (7,-6)--(39,-6);
\draw (7,0) arc (90:270:3);
\draw (39,0) arc (90:-90:3);
\draw (8,-3) circle (1);
%\draw[red] (8,-3) circle (1.5);
\draw (16,-3) circle (1);
%\draw[red] (12,-3) circle (1.5);

\draw (13,-3) node{$\cdots$};
\draw (20,-2) arc (90:270:1);
\draw (21.5,-2) arc (90:-90:1);

%\draw[red] (20,-1.5) arc (1.5);
\draw (21.5,-2) arc (90:-90:1);
%vertical part
\draw (20,-4)--(20,14.5);
\draw (21.5,-4)--(21.5,2.5);
%\draw[red] (20.75,-4)--(20.75,8.5);
%\draw[red] (20.75,8.5) arc (180:0:2.25 and 1.75);
%\draw[red] (20.75,-4)--(23,-5.4)--(25.25,-4);
%\draw[red] (20.75,-4) arc (180:360:2.25 and 1.25);

\draw (21.5,2.5) arc (180:0:1.5);
\draw (20,14.5) arc (180:0:3);

\draw[red] (8,-.5) arc (90:360:2.5);
\draw[red] (26,-.5) arc (90:-90:2.5);
\draw[red] (10.5,-3) to[out=up,in=left] (12.,-1)--(20,-1);
\draw[red] (8,-.5)--(20,-.5);
\draw[red] (21.5,-.5)--(24.5,-.5);
\draw[red] (25.25,-5.5)--(26,-5.5);
\draw[red] (25.25,-5.5) to[out=left,in=down] (23,-3) to[out=up,in=right](21.5,-1);
%\draw[red] (25.,6.5)--(25.,14.5);
%\draw[red] (25.,14.5) arc (0:180:2);
%\draw[red] (25.,6.5) arc (0:-180:2);
\draw[red] (19,1.5) node[left]{$b$};

%\draw[red] (25.,6.5)--(25.,14.5);
%\draw[red] (25.,14.5) arc (0:180:2);
%\draw[red] (25.,6.5) arc (0:-180:2);

\draw (23, 6.5) circle (1);

\draw (23,11.) node{$\vdots$};

\draw (23, 14.5) circle (1);

\draw (23,18) node[]{};
\draw (24+2,-2) arc (90:-90:1);
\draw (24.5,-2) arc (90:270:1);

\draw (26,-4)--(26,14.5);
\draw (24.5,-4)--(24.5,2.5);
\draw [decorate,decoration={brace,mirror,amplitude=5pt},xshift=0pt,yshift=0]
	(26.5,6.5) -- (26.5,14.5) node [black,midway,xshift=20pt,align=center] 
	{\footnotesize $p+1$};

%vertical part 끝

\draw (30,-3) circle (1);
\draw (34,-3) node{$\cdots$};
\draw (38,-3) circle (1);

%\draw[blue] (20,-1.5)--(8,-1.5);
%\draw[blue] (21.5,-4.5)--(8,-4.5);
%\draw[blue] (8,-1.5) arc (90:270:1.5);
%\draw[blue] (21.5,-1.5) arc (90:-90:1.5);
\draw [decorate,decoration={brace,mirror,amplitude=5pt},xshift=0pt,yshift=0]
	(8,-6.5) -- (16,-6.5) node [black,midway,yshift=-15pt] 
	{\footnotesize $r+2$};
\draw [decorate,decoration={brace,mirror,amplitude=5pt},xshift=0pt,yshift=0]
	(30,-6.5) -- (38,-6.5) node [black,midway,yshift=-15pt] 
	{\footnotesize $q+3$};
\end{scope}

\end{tikzpicture}
\caption{$a=t_{\beta_1}(a)$ and $b$ in $\Sigma_{\gamma_{p,q,r}}$}
\label{egfamily8}
\end{figure}
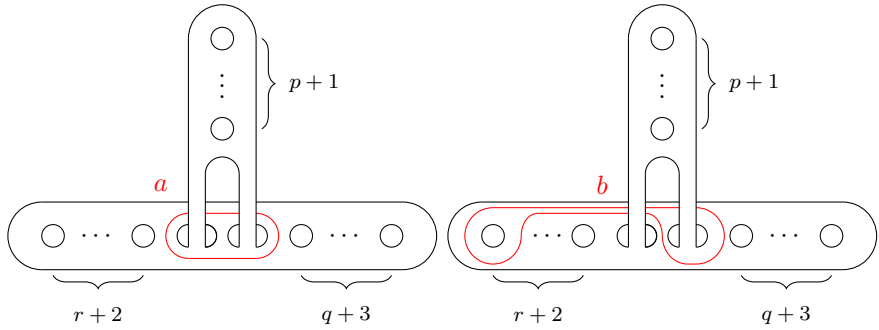
After a lantern substitution and cancelling $(t_{\beta_1}(t_{\gamma_{r+3}}))^{-1}$ with $t_{\beta_1}(t_{\gamma_{r+3}})$, we have
\begin{eqnarray*}
&&
\beta_2\cdots\beta_{p+2}\cdot\beta^{-1}\cdot\alpha_1\cdots\alpha_{r+2}\cdot\gamma_{r+2}^{-1}\cdot\gamma_{r+4}^{r+1}\cdot\gamma_{r+5}^p
\\
&&
\cdot
(t_{\beta_1}(\gamma_{r+3}))^{-1}
\cdot\gamma_{r+2}
\cdot
t_{\beta_1}(\alpha_{r+3})
\cdot
(t_{\alpha_{r+4}}^{-1})(\beta_1)
\cdot
(t_{\alpha_{r+4}}^{-1}\cdot t_{\beta_1}^{-1}\cdot t_{\alpha_{r+4}})(Y_{p,1,r})
\\
&&
\cdot
(t_{\alpha_{r+4}}^{-1}\cdot t_{\beta_1}^{-1}\cdot t_{\alpha_{r+4}})(Y_{p,2,r})
\cdots
(t_{\alpha_{r+4}}^{-1}\cdot t_{\beta_1}^{-1}\cdot t_{\alpha_{r+4}})(Y_{p,q+6,r})
\\
&&
\cdot
(t_{\alpha_{r+4}}^{-1}\cdot t_{\beta_1}^{-1} \cdot t_{\delta_{p+2}}^{-1}\cdot t_{\gamma}^{-1}\cdot t_{\delta_{p+2}})(\gamma)
\\
&=&
\beta_2\cdots\beta_{p+2}\cdot\beta^{-1}\cdot\alpha_1\cdots\alpha_{r+2}\cdot\gamma_{r+2}^{-1}\cdot\gamma_{r+4}^{r+1}\cdot\gamma_{r+5}^p
\cdot
t_{\beta_1}(a)
\cdot
t_{\beta_1}(b)
\\
&&
\cdot
(t_{\alpha_{r+4}}^{-1}\cdot t_{\beta_1}^{-1}\cdot t_{\alpha_{r+4}})(Y_{p,2,r})
\cdots
(t_{\alpha_{r+4}}^{-1}\cdot t_{\beta_1}^{-1}\cdot t_{\alpha_{r+4}})(Y_{p,q+6,r})
\\
&&
\cdot
(t_{\alpha_{r+4}}^{-1}\cdot t_{\beta_1}^{-1} \cdot t_{\delta_{p+2}}^{-1}\cdot t_{\gamma}^{-1}\cdot t_{\delta_{p+2}})(\gamma)
\\
&=&
\alpha_1\cdots\alpha_{r+2}
\cdot
t_{\beta_1}(a)
\cdot\gamma_{r+4}^{r+1}
\cdot
\gamma_{r+2}^{-1}
\cdot
t_{\beta_1}(b)
\\
&&
\cdot
(t_{\alpha_{r+4}}^{-1}\cdot t_{\beta_1}^{-1}\cdot t_{\alpha_{r+4}})(Y_{p,2,r})
\cdots
(t_{\alpha_{r+4}}^{-1}\cdot t_{\beta_1}^{-1}\cdot t_{\alpha_{r+4}})(Y_{p,q+4,r})
\\
&&
\cdot\beta_2\cdots\beta_{p+2}\cdot
Y_{p,q+5,r}
\cdot
\gamma_{r+5}^p
\cdot
\beta^{-1}
\\
&&
\cdot
(t_{\alpha_{r+4}}^{-1}\cdot t_{\beta_1}^{-1}\cdot t_{\alpha_{r+4}})(Y_{p,q+6,r})
\cdot
(t_{\alpha_{r+4}}^{-1}\cdot t_{\beta_1}^{-1} \cdot t_{\delta_{p+2}}^{-1}\cdot t_{\gamma}^{-1}\cdot t_{\delta_{p+2}})(\gamma)
\end{eqnarray*}
One can easily check that 
\begin{eqnarray*}\alpha_1\cdots\alpha_{r+2}
\cdot
t_{\beta_1}(a)
\cdot\gamma_{r+4}^{r+1}
&=&
Z_{p,q,1}\cdots Z_{p,q,r+3}
\\
\beta_2\cdots\beta_{p+2}\cdot
Y_{p,q+5,r}
\cdot
\gamma_{r+5}^p
&=&X_{1,q,r}\dots X_{p+2,q,r}
\end{eqnarray*}
for some simple closed curves $X_{i,q,r}$ and $Z_{p,q,j}$ with $X_{p+2,q,r}=\beta$ and $Z_{p,q,r+3}=\gamma_{r+2}$ due to the daisy relations. By performing  daisy substitutions and cancelling $X_{p+2,q,r}$ with $\beta^{-1}$ and $Z_{p,q,r+3}$ with $\gamma_{r+2}^{-1}$, we get a  monodromy factorization $W_{\Gamma_{p,q,r}}'$ whose length is $b_1(\Sigma_{\Gamma_{p,q,r}})$.
\subsection{Relations for (h) family}
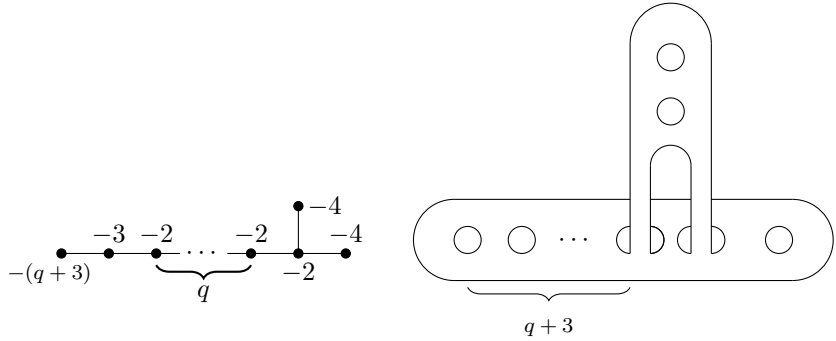
\begin{figure}[h]
\begin{tikzpicture}[scale=0.18]
\draw (7,0)--(20,0) (21.5,0)--(24.5,0) (26,0)--(32,0);
\draw (7,-6)--(32,-6);
\draw (7,0) arc (90:270:3);
\draw (30+2,0) arc (90:-90:3);
\draw (8,-3) circle (1);
\draw (12,-3) circle (1);
\draw (16,-3) node{$\cdots$};
\draw (20,-2) arc (90:270:1);
\draw (21.5,-2) arc (90:-90:1);

\draw (21.5,-2) arc (90:-90:1);
\draw (29+2,-3) circle (1);
%\draw (33+2,-3) circle (1);
\draw [decorate,decoration={brace,mirror,amplitude=5pt},xshift=0pt,yshift=0]
	(8,-6.5) -- (20,-6.5) node [black,midway,yshift=-15pt] 
	{\footnotesize $q+3$};

%vertical part
\draw (20,-4)--(20,11.5);
\draw (21.5,-4)--(21.5,2.5);

\draw (21.5,2.5) arc (180:0:1.5);
\draw (20,11.5) arc (180:0:3);

\draw (23, 6.5) circle (1);
\draw (23, 10.5) circle (1);

\draw (24+2,-2) arc (90:-90:1);
\draw (24.5,-2) arc (90:270:1);

\draw (26,-4)--(26,11.5);
\draw (24.5,-4)--(24.5,2.5);

%vertical part 끝

\begin{scope}[shift={(-22,-4)},scale=3.5]
\node[bullet] at (0,0){};
\node[bullet] at (1,0){};
\node[bullet] at (2,0){};
\node[bullet] at (4,0){};
\node[bullet] at (5,0){};
\node[bullet] at (6,0){};

\node[bullet] at (5,1){};

\node[below] at (-0.25,0){\footnotesize{$-(q+3)$}};
\node[above] at (1,0){$-3$};
\node[above] at (2,0){$-2$};
\node[above] at (4,0){$-2$};
\node[below] at (5,0){$-2$};
\node[above] at (6,0){$-4$};

\node[right] at (5,1){$-4$};

\node at (3,0){$\cdots$};

\draw (0,0)--(2.5,0);
\draw (3.5,0)--(6,0);
\draw (5,0)--(5,1);

	\draw [thick,decorate,decoration={brace,mirror,amplitude=5pt},xshift=0pt,yshift=-7pt]
	(2,0) -- (4,0) node [black,midway,yshift=-11pt] 
	{$q$};

\end{scope}
\end{tikzpicture}

\caption{Resolution graph $\Gamma_q$ and generic fiber $\Sigma_{\Gamma_q}$ for (h) family}
\label{h}
\end{figure}
Let $\Gamma_q$ be a resolution graph of (h) family as in Figure~\ref{h}. Then the generic fiber for $X_{\Gamma_q}$ is $\Sigma_{q}$ as in Figure~\ref{h} and the global monodromy  of $X_{\Gamma_q}$ is given by
%\begin{eqnarray*}
%&\phantom{0}& \hspace{-2 em} xy\gamma_2^2\alpha_1\alpha_2\alpha_3\alpha_4\alpha_5\gamma_5\\ 
%&\sim& x\alpha_2\gamma_2 x\alpha_3\alpha_1\alpha_4\alpha_5\gamma_2\gamma_5\\
%&\sim&t_x(\alpha_2)t_x(\gamma_2)x^2\alpha_3\alpha_1\alpha_4\alpha_5\gamma_2\gamma_5\\&\sim&t_x(\alpha_2)t_x(\gamma_2)t_x^2(\alpha_3)x^2\alpha_1\alpha_4\alpha_5\gamma_2\gamma_5\\
%&\sim&t_x^2(\alpha_3)(t_x^2\cdot t_{\alpha_3}^{-1}\cdot t_x^{-1})(\alpha_2)(t_x^2\cdot t_{\alpha_3}^{-1}\cdot t_x^{-1})(\gamma_2)x^2\alpha_1\alpha_4\alpha_5\gamma_2\gamma_5.
%\end{eqnarray*}
$$\beta_1\beta_2\beta_3\alpha_1\alpha_2\cdots\alpha_{q+2}\gamma_{q+2}\alpha_{q+3}\gamma_{q+3}^{q+1}\delta_3\gamma_{q+3}\alpha_{q+4}\alpha_{q+5}\gamma_{q+5}
$$
We rearrange the word using Hurwitz moves as follows.
\begin{eqnarray*}
&\phantom{0}& \hspace{-2 em} \beta_1\beta_2\beta_3\alpha_1\alpha_2\cdots\alpha_{q+2}\gamma_{q+2}\alpha_{q+3}\gamma_{q+3}^{q+1}\delta_3\gamma_{q+3}\alpha_{q+4}\alpha_{q+5}\gamma_{q+5}\\ 
&=&\hspace{-.55 em} \beta_1\beta_2\beta_3\alpha_1\alpha_2\cdots \alpha_{q+2}\gamma_{q+2}\gamma_{q+3}^{q+1}\alpha_{q+4}(t_{\alpha_{q+3}}\cdot t_{\alpha_{q+4}}^{-1})(\delta_3)\gamma_{q+3}\alpha_{q+3}\alpha_{q+5}\gamma_{q+5}
\\
&=&\hspace{-.55 em}\beta_2\beta_3\alpha_1\alpha_2\cdots\alpha_{q+2}\gamma_{q+2} (t_{\beta_1}(\gamma_{q+3}))^{q+1} t_{\beta_1}(\alpha_{q+4})\beta_1
(t_{\alpha_{q+3}}\cdot t_{\alpha_{q+4}}^{-1})(\delta_3)\gamma_{q+3}
\\
&&\hspace{-.55 em}\cdot\alpha_{q+3}\alpha_{q+5}\gamma_{q+5}
\\
&=&\hspace{-.55 em}\beta_2\beta_3\alpha_1\alpha_2\cdots\alpha_{q+2}\gamma_{q+2} (t_{\beta_1}(\gamma_{q+3}))^{q+1} t_{\beta_1}(\alpha_{q+4})\beta_1\gamma_{q+3}
(t_{\gamma_{q+3}}^{-1}\cdot t_{\alpha_{q+3}}\cdot t_{\alpha_{q+4}}^{-1})(\delta_3)
\\
&&\hspace{-.55 em}\cdot\alpha_{q+3}\alpha_{q+5}\gamma_{q+5}
\end{eqnarray*}
Let $\delta=(t_{\gamma_{q+3}}^{-1}\cdot t_{\alpha_{q+3}}\cdot t_{\alpha_{q+4}}^{-1})(\delta_3)$.

\begin{eqnarray*}
&\phantom{0}& \hspace{-2 em} \beta_2\beta_3\alpha_1\alpha_2\cdots\alpha_{q+2}\gamma_{q+2} (t_{\beta_1}(\gamma_{q+3}))^{q+1} t_{\beta_1}(\alpha_{q+4})\beta_1\gamma_{q+3}
\delta\alpha_{q+3}\alpha_{q+5}\gamma_{q+5}
\\
&=&\hspace{-.55 em}
\beta_2\beta_3\alpha_1\alpha_2\cdots\alpha_{q+2}\gamma_{q+2} (t_{\beta_1}(\gamma_{q+3}))^{q+1} t_{\beta_1}(\alpha_{q+4})t_{\beta_1}(\gamma_{q+3})
\beta_1 \delta \alpha_{q+3}\alpha_{q+5}\gamma_{q+5}
\\
&=&\hspace{-.55 em} 
\beta_2\beta_3\alpha_1\alpha_2\cdots\alpha_{q+2}\gamma_{q+2} (t_{\beta_1}(\gamma_{q+3}))^{q+1} t_{\beta_1}(\alpha_{q+4})t_{\beta_1}(\gamma_{q+3})
\delta t_\delta^{-1}(\beta_1)\alpha_{q+3}\alpha_{q+5}\gamma_{q+5}
\\
&=&\hspace{-.55 em} 
\beta_2\beta_3\alpha_1\alpha_2\cdots\alpha_{q+2}\gamma_{q+2} (t_{\beta_1}(\gamma_{q+3}))^{q+2} t_{\beta_1}(\alpha_{q+4})\delta t_\delta^{-1}(\beta_1)\alpha_{q+3}\alpha_{q+5}\gamma_{q+5}
\end{eqnarray*}
Now we introduce cancelling pair $\gamma_{q+4}\cdot \gamma_{q+4}^{-1}$ as follows.
$$\underline{\beta_2\beta_3\alpha_1\alpha_2\cdots\alpha_{q+2}\gamma_{q+2} (t_{\beta_1}(\gamma_{q+3}))^{q+2} t_{\beta_1}(\alpha_{q+4})\delta \cdot \textcolor{red}{\gamma_{q+4}}}
\cdot \textcolor{red}{\gamma_{q+4}^{-1}} t_\delta^{-1}(\beta_1)\alpha_{q+3}\alpha_{q+5}\gamma_{q+5}
$$
By taking a global conjugation of each monodromy with $\gamma_{q+3}$, the underlined part can be seen an embeeding of the linear plumbing $L_q$ in Figure~\ref{hfamily2}: The $\alpha_{i}$, $\beta{j}$ and $\gamma_{k}$ is unchanged under conjugation.  Note that $t_{\beta_1}(\gamma_{q+3})=t_{\gamma_{q+3}}^{-1}(\beta_1)$ since $\beta_1$ and $\gamma_{q+3}$ intersect geometrically once. And $(t_{\gamma_{q+3}} \cdot t_{\beta_1})(\alpha_{q+4})$ and $t_{\gamma_{q+3}}(\delta)$ can be drawn as in Figure~\ref{hfamily1}. 
\begin{figure}[h]
\begin{tikzpicture}[scale=0.19]
\begin{scope}
\draw (7,0)--(20,0) (21.5,0)--(24.5,0) (26,0)--(32,0);
\draw (7,-6)--(32,-6);
\draw (7,0) arc (90:270:3);
\draw (30+2,0) arc (90:-90:3);
\draw (8,-3) circle (1);
\draw (12,-3) circle (1);
\draw (16,-3) node{$\cdots$};
\draw (20,-2) arc (90:270:1);
\draw (21.5,-2) arc (90:-90:1);

\draw (21.5,-2) arc (90:-90:1);
\draw (29+2,-3) circle (1);
%\draw (33+2,-3) circle (1);
\draw [decorate,decoration={brace,mirror,amplitude=5pt},xshift=0pt,yshift=0]
	(8,-6.5) -- (20,-6.5) node [black,midway,yshift=-15pt] 
	{\footnotesize $q+3$};

%vertical part

\draw[red] (20.75,-4)--(20.75, 10.5);
\draw[red] (25.25,-4)--(25.25, 10.5);
\draw[red] (20.75,10.5) to [out=up, in=left] (23,12.5) to [out=right, in=up] (25.25, 10.5);

\draw[red] (20.75,-4) to [out=down, in= down] (18.25,-3);
\draw[red] (25.25,-4) to [out=down, in= down] (27.75,-3);

\draw[red] (18.25,-3) to [out=up, in=left] (20,-1.5);
\draw[red] (27.75,-3) to [out=up, in=right] (26,-1.5);
\draw[red] (21.5,-1.5)--(24.5,-1.5);

\draw[red] (20,6) node[left]{\footnotesize{$t_{\gamma_{q+3}}(\delta)$}};

\draw (20,-4)--(20,11.5);
\draw (21.5,-4)--(21.5,2.5);

\draw (21.5,2.5) arc (180:0:1.5);
\draw (20,11.5) arc (180:0:3);

\draw (23, 6.5) circle (1);
\draw (23, 10.5) circle (1);

\draw (24+2,-2) arc (90:-90:1);
\draw (24.5,-2) arc (90:270:1);

\draw (26,-4)--(26,11.5);
\draw (24.5,-4)--(24.5,2.5);
\end{scope}

\begin{scope}[shift={(35,0)}]

\draw[red] (20.75,-4)--(20.75, 3);
\draw[red] (25.25,-4)--(25.25, 3);
\draw[red] (20.75,3) to [out=up, in=left] (23,5) to [out=right, in=up] (25.25, 3);
\draw[red] (25.25,-4) to [out=down, in= down] (27.75,-3);

\draw[red] (20.75,-4) to[out=down,in=right] (18,-5)--(7.5,-5) to [out=left, in= down] (6,-3);
\draw[red] (6,-3) to [out=up, in= left] (7.5,-1.5)--(20,-1.5);

\draw[red] (27.75,-3) to [out=up, in=right] (26,-1.5);
\draw[red] (21.5,-1.5)--(24.5,-1.5);

\draw[red] (20,6) node[left]{\footnotesize{$(t_{\gamma_{q+3}}\cdot t_{\beta_1})(\alpha_{q+4})$}};

\draw (7,0)--(20,0) (21.5,0)--(24.5,0) (26,0)--(32,0);
\draw (7,-6)--(32,-6);
\draw (7,0) arc (90:270:3);
\draw (30+2,0) arc (90:-90:3);
\draw (8,-3) circle (1);
\draw (12,-3) circle (1);
\draw (16,-3) node{$\cdots$};
\draw (20,-2) arc (90:270:1);
\draw (21.5,-2) arc (90:-90:1);

\draw (21.5,-2) arc (90:-90:1);
\draw (29+2,-3) circle (1);
%\draw (33+2,-3) circle (1);
\draw [decorate,decoration={brace,mirror,amplitude=5pt},xshift=0pt,yshift=0]
	(8,-6.5) -- (20,-6.5) node [black,midway,yshift=-15pt] 
	{\footnotesize $q+3$};

%vertical part
\draw (20,-4)--(20,11.5);
\draw (21.5,-4)--(21.5,2.5);

\draw (21.5,2.5) arc (180:0:1.5);
\draw (20,11.5) arc (180:0:3);

\draw (23, 6.5) circle (1);
\draw (23, 10.5) circle (1);

\draw (24+2,-2) arc (90:-90:1);
\draw (24.5,-2) arc (90:270:1);

\draw (26,-4)--(26,11.5);
\draw (24.5,-4)--(24.5,2.5);
\end{scope}

\end{tikzpicture}
\caption{$t_{\gamma_{q+3}}(\delta)$ and $(t_{\gamma_{q+3}}\cdot t_{\beta_1})(\alpha_{q+4})$ in $\Sigma_{\gamma_q}$}
\label{hfamily1}
\end{figure}
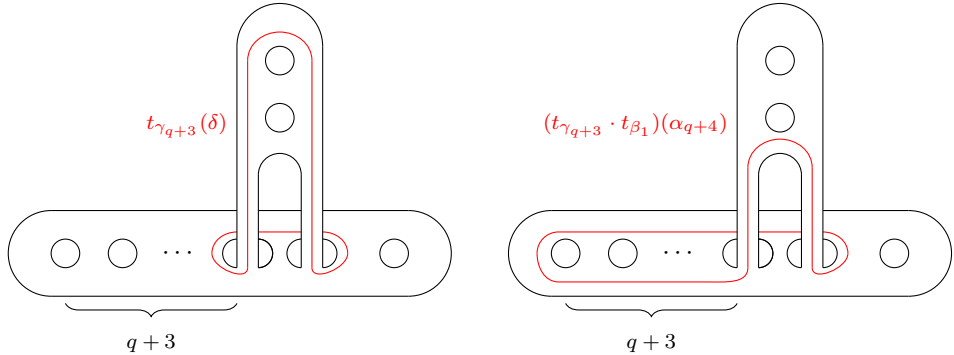
Let $L_q$ be a linear plumbing and $\Sigma_{L_q}$ be a generic fiber for $X_{L_q}$ as in Figure~\ref{hfamily2}. Then the monodromy for $X_{L_q}$ can be written as
$$
a_1a_2\cdots a_{q+2}b_{q+2}a_{q+3}a_{q+4}a_{q+5}b_{q+5}a_{q+6}^{q+2}b_{q+6}
$$
where $a_i$ is simple closed curve in $\Sigma_{L_q}$ enclosing $i$th hole and $b_j$ is simple closed curve in $\Sigma_{L_q}$ enclosing from the first to $i$th holes. Then there is a planar subsurface of $\Sigma_{\Gamma_q}$ which is diffeomorphic to $\Sigma_{L_q}$ so that the image of each curves are
\begin{eqnarray*}
a_i &\rightarrow&  \alpha_i \phantom{0000}(i=1,\dots, q+2)
\\
b_{q+2} &\rightarrow& \gamma_{q+2}
\\
a_{q+3} &\rightarrow& t_{\gamma_{q+3}}(\delta)
\\
a_{q+4} &\rightarrow& \beta_3
\\
a_{q+5} &\rightarrow& \beta_2
\\
b_{q+5} &\rightarrow& (t_{\gamma_{q+3}}\cdot t_{\beta_1})(\alpha_{q+4})
\\
a_{q+6} &\rightarrow& \beta_1
\\
b_{q+6} &\rightarrow& \gamma_{q+4}
\end{eqnarray*}
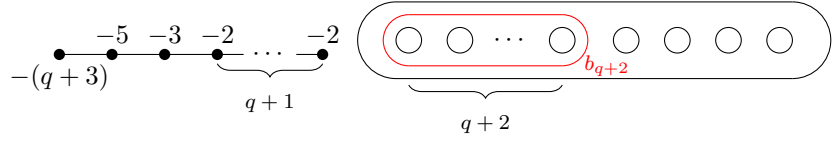
\begin{figure}[h]
\begin{tikzpicture}[scale=0.7]
\draw (0,0) node[below]{$-(q+3)$};

\draw (1,0) node[above]{$-5$};
\draw (2,0) node[above]{$-3$};
\draw (3,0) node[above]{$-2$};
\draw (5,0) node[above]{$-2$};

\draw (0,0)--(3.5,0) (4.5,0)--(5,0);
\draw (0,0 ) node[circle, fill, inner sep=1.5pt, black]{};
\draw (1,0 ) node[circle, fill, inner sep=1.5pt, black]{};
\draw (2,0 ) node[circle, fill, inner sep=1.5pt, black]{};

\draw (4,0) node{$\cdots$};
\draw (3,0 ) node[circle, fill, inner sep=1.5pt, black]{};
\draw (5,0 ) node[circle, fill, inner sep=1.5pt, black]{};
\draw (0,-.5) node{};

\draw [decorate,decoration={brace,mirror,amplitude=5pt},xshift=0pt,yshift=-5]
	(3,0) -- (5,0) node [black,midway,yshift=-15pt] 
	{\footnotesize $q+1$};

\draw (0,-1.5) node[]{};

\end{tikzpicture}
\begin{tikzpicture}
\begin{scope}[shift={(5,.5)},scale=0.17]
\draw (7,0)--(38,0);
\draw (7,-6)--(38,-6);
\draw (7,0) arc (90:270:3);
\draw (38,0) arc (90:-90:3);
\draw (8,-3) circle (1);
%\draw[red] (8,-3) circle (1.5);
\draw (12,-3) circle (1);
%\draw[red] (12,-3) circle (1.5);

\draw (16,-3) node{$\cdots$};
\draw (20,-3) circle (1);
%\draw[red] (20,-3) circle (1.5);

\draw (25,-3) circle (1);
\draw (29,-3) circle (1);
\draw (33,-3) circle (1);
\draw (37,-3) circle (1);

\draw[red] (20,-1)--(8,-1);
\draw[red] (20,-5)--(8,-5);
\draw[red] (8,-1) arc (90:270:2);
\draw[red] (20,-1) arc (90:-90:2);
\draw[red] (23.5,-5.) node {\footnotesize{$b_{q+2}$}};

\draw [decorate,decoration={brace,mirror,amplitude=5pt},xshift=0pt,yshift=0]
	(8,-6.5) -- (20,-6.5) node [black,midway,yshift=-15pt] 
	{\footnotesize $q+2$};

\end{scope}

\end{tikzpicture}
\caption{Linear plumbing $L_q$ and the generic fiber $\Sigma_{L_q}$ for $X_{L_q}$}
\label{hfamily2}
\end{figure}
The linear plumbing $L_q$ can be rationally blowdown and the relation in the mapping class group of $\Sigma_{L_q}$ for the rational blowdown was given by Endo-Mark-Van Horn-Morris in 
~\cite{MR2783383}.
\begin{eqnarray*}
a_1a_2\cdots a_{q+2}b_{q+2}a_{q+3}a_{q+4}a_{q+5}b_{q+5}a_{q+6}^{q+2}
=y_1y_2\cdots y_{q+6}
\end{eqnarray*}
Let $Y_i$ a simple closed curve in $\Sigma_{\Gamma_q}$ be the image of $y_i$ in $\Sigma_{L_q}$ which can be drawn as in Figure~\ref{hfamily3} to Figure~\ref{hfamily7} and $X_i$ be $t_{\gamma_{q+3}}^{-1}(Y_i)$. 
\begin{figure}[h]
\begin{tikzpicture}[scale=0.15]
\begin{scope}
\draw (7,0)--(33,0);
\draw (7,-6)--(33,-6);
\draw (7,0) arc (90:270:3);
\draw (33,0) arc (90:-90:3);
\draw (8,-3) circle (1);
%\draw[red] (8,-3) circle (1.5);
\draw (16,-3) circle (1);
%\draw[red] (12,-3) circle (1.5);

\draw[red] (32,-1)--(20,-1);
\draw[red] (32,-5)--(20,-5);
\draw[red] (20,-1) arc (90:270:2);
\draw[red] (32,-1) arc (90:-90:2);
\draw[red] (26,-7.25) node {\footnotesize{$y_1$}};

\draw (12,-3) node{$\cdots$};
\draw (20,-3) circle (1);
%\draw[red] (20,-3) circle (1.5);

\draw (24,-3) circle (1);
\draw (28,-3) circle (1);
\draw (32,-3) circle (1);

\draw [decorate,decoration={brace,mirror,amplitude=5pt},xshift=0pt,yshift=0]
	(8,-6.5) -- (16,-6.5) node [black,midway,yshift=-15pt] 
	{\footnotesize $q+2$};
\end{scope}

\begin{scope}[shift={(35,0)}]

\draw[red] (26,-5)--(20,-5);
\draw[red] (20,-1) arc (90:270:2);
\draw[red] (26,-1) arc (90:-90:2);
\draw[red] (24.5,-1)--(21.5,-1);
\draw[red] (25,-7.25) node {\footnotesize{$Y_1$}};

\draw (7,0)--(20,0) (21.5,0)--(24.5,0) (26,0)--(32,0);
\draw (7,-6)--(32,-6);
\draw (7,0) arc (90:270:3);
\draw (30+2,0) arc (90:-90:3);
\draw (8,-3) circle (1);
\draw (16,-3) circle (1);
\draw (12,-3) node{$\cdots$};
\draw (20,-2) arc (90:270:1);
\draw (21.5,-2) arc (90:-90:1);

\draw (21.5,-2) arc (90:-90:1);
\draw (29+2,-3) circle (1);
%\draw (33+2,-3) circle (1);
\draw [decorate,decoration={brace,mirror,amplitude=5pt},xshift=0pt,yshift=0]
	(8,-6.5) -- (16,-6.5) node [black,midway,yshift=-15pt] 
	{\footnotesize $q+2$};

%vertical part
\draw (20,-4)--(20,11.5);
\draw (21.5,-4)--(21.5,2.5);

\draw (21.5,2.5) arc (180:0:1.5);
\draw (20,11.5) arc (180:0:3);

\draw (23, 6.5) circle (1);
\draw (23, 10.5) circle (1);

\draw (24+2,-2) arc (90:-90:1);
\draw (24.5,-2) arc (90:270:1);

\draw (26,-4)--(26,11.5);
\draw (24.5,-4)--(24.5,2.5);
\end{scope}

\end{tikzpicture}
\caption{$y_1$ and $Y_1$}
\label{hfamily3}
\end{figure}
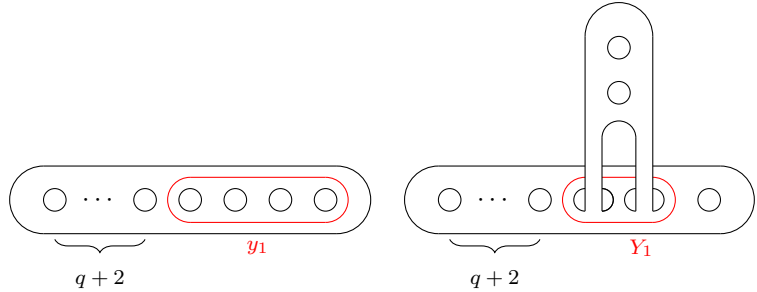
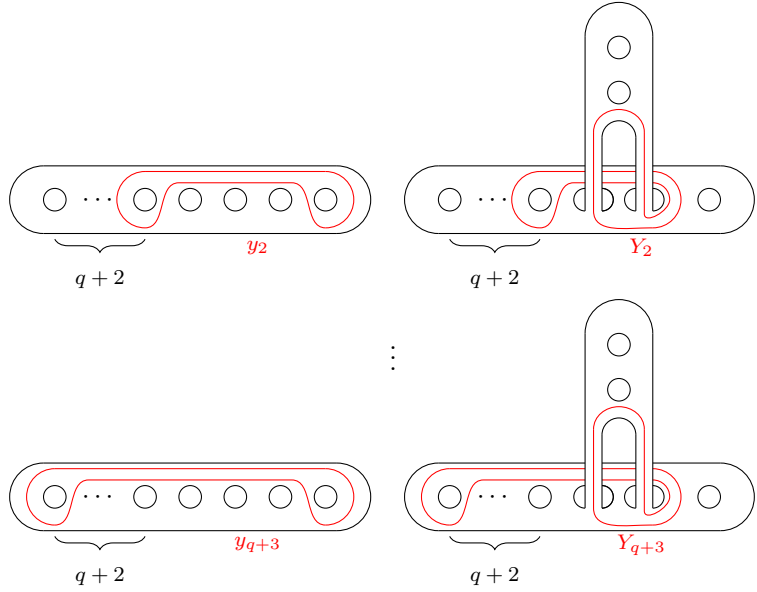
\begin{figure}[h]
\begin{tikzpicture}[scale=0.15]

\begin{scope}

\draw[red] (32,-.5)--(16,-.5);
\draw[red] (16,-.5) arc (90:270:2.5);
\draw[red] (32,-.5) arc (90:-90:2.5);
\draw[red] (16,-5.5) to[out=right,in=left] (19,-1.5)--(29,-1.5) to[out=right, in=left] (32,-5.5);

\draw[red] (26,-7.25) node {\footnotesize{$y_2$}};

\draw (7,0)--(33,0);
\draw (7,-6)--(33,-6);
\draw (7,0) arc (90:270:3);
\draw (33,0) arc (90:-90:3);
\draw (8,-3) circle (1);
%\draw[red] (8,-3) circle (1.5);
\draw (16,-3) circle (1);
%\draw[red] (12,-3) circle (1.5);

\draw (12,-3) node{$\cdots$};
\draw (20,-3) circle (1);
%\draw[red] (20,-3) circle (1.5);

\draw (24,-3) circle (1);
\draw (28,-3) circle (1);
\draw (32,-3) circle (1);

\draw [decorate,decoration={brace,mirror,amplitude=5pt},xshift=0pt,yshift=0]
	(8,-6.5) -- (16,-6.5) node [black,midway,yshift=-15pt] 
	{\footnotesize $q+2$};
\end{scope}

\begin{scope}[shift={(35,0)}]

\draw[red] (16,-.5)--(20,-.5);
\draw[red] (21.5,-.5)--(24.5,-.5);
\draw[red] (21.5,-1.5)--(24.5,-1.5);

\draw[red] (26,-.5) arc (90:-90:2.5);

\draw[red] (16,-.5) arc (90:270:2.5);
\draw[red] (16,-5.5) to[out=right,in=left] (19,-1.5)--(20,-1.5);

\draw[red] (26,-5.5)--(25.75,-5.5) to [out=left,in=down](20.75,-4)--(20.75,3);
\draw[red] (25.25,3)--(25.25,-4.25) to[out=down,in=down] (27.5,-3) to[out=up,in=right] (26,-1.5);
\draw[red] (25,-7.25) node {\footnotesize{$Y_2$}};

%\draw[red] (20.75,-4)--(20.75, 3);
%\draw[red] (25.25,-4)--(25.25, 3);
\draw[red] (20.75,3) to [out=up, in=left] (23,5) to [out=right, in=up] (25.25, 3);
%\draw[red] (25.25,-4) to [out=down, in= down] (27.75,-3);

%\draw[red] (20.75,-4) to[out=down,in=right] (18,-5)--(7.5,-5) to [out=left, in= down] (6,-3);
%\draw[red] (6,-3) to [out=up, in= left] (7.5,-1.5)--(20,-1.5);

%\draw[red] (27.75,-3) to [out=up, in=right] (26,-1.5);
%\draw[red] (21.5,-1.5)--(24.5,-1.5);

%\draw[red] (20,6) node[left]{\footnotesize{$(t_{\gamma_{q+3}}\cdot t_{\beta_1})(\alpha_{q+4})$}};

\draw (7,0)--(20,0) (21.5,0)--(24.5,0) (26,0)--(32,0);
\draw (7,-6)--(32,-6);
\draw (7,0) arc (90:270:3);
\draw (30+2,0) arc (90:-90:3);
\draw (8,-3) circle (1);
\draw (16,-3) circle (1);
\draw (12,-3) node{$\cdots$};
\draw (20,-2) arc (90:270:1);
\draw (21.5,-2) arc (90:-90:1);

\draw (21.5,-2) arc (90:-90:1);
\draw (29+2,-3) circle (1);
%\draw (33+2,-3) circle (1);
\draw [decorate,decoration={brace,mirror,amplitude=5pt},xshift=0pt,yshift=0]
	(8,-6.5) -- (16,-6.5) node [black,midway,yshift=-15pt] 
	{\footnotesize $q+2$};

%vertical part
\draw (20,-4)--(20,11.5);
\draw (21.5,-4)--(21.5,2.5);

\draw (21.5,2.5) arc (180:0:1.5);
\draw (20,11.5) arc (180:0:3);

\draw (23, 6.5) circle (1);
\draw (23, 10.5) circle (1);

\draw (24+2,-2) arc (90:-90:1);
\draw (24.5,-2) arc (90:270:1);

\draw (26,-4)--(26,11.5);
\draw (24.5,-4)--(24.5,2.5);
\end{scope}

\end{tikzpicture}
\begin{tikzpicture}[scale=0.15]
\begin{scope}

\draw[red] (32,-.5)--(8,-.5);
\draw[red] (8,-.5) arc (90:270:2.5);
\draw[red] (32,-.5) arc (90:-90:2.5);
\draw[red] (8,-5.5) to[out=right,in=left] (11,-1.5)--(29,-1.5) to[out=right, in=left] (32,-5.5);

\draw[red] (26,-7.25) node {\footnotesize{$y_{q+3}$}};

\draw (7,0)--(33,0);
\draw (7,-6)--(33,-6);
\draw (7,0) arc (90:270:3);
\draw (33,0) arc (90:-90:3);
\draw (8,-3) circle (1);
%\draw[red] (8,-3) circle (1.5);
\draw (16,-3) circle (1);
%\draw[red] (12,-3) circle (1.5);

\draw (12,-3) node{$\cdots$};
\draw (20,-3) circle (1);
%\draw[red] (20,-3) circle (1.5);

\draw (24,-3) circle (1);
\draw (28,-3) circle (1);
\draw (32,-3) circle (1);

\draw [decorate,decoration={brace,mirror,amplitude=5pt},xshift=0pt,yshift=0]
	(8,-6.5) -- (16,-6.5) node [black,midway,yshift=-15pt] 
	{\footnotesize $q+2$};
\end{scope}

\begin{scope}[shift={(35,0)}]

\draw[red] (8,-.5)--(20,-.5);
\draw[red] (21.5,-.5)--(24.5,-.5);
\draw[red] (21.5,-1.5)--(24.5,-1.5);

\draw[red] (26,-.5) arc (90:-90:2.5);

\draw[red] (8,-.5) arc (90:270:2.5);
\draw[red] (8,-5.5) to[out=right,in=left] (11,-1.5)--(20,-1.5);

\draw[red] (26,-5.5)--(25.75,-5.5) to [out=left,in=down](20.75,-4)--(20.75,3);
\draw[red] (25.25,3)--(25.25,-4.25) to[out=down,in=down] (27.5,-3) to[out=up,in=right] (26,-1.5);
\draw[red] (25,-7.25) node {\footnotesize{$Y_{q+3}$}};

%\draw[red] (20.75,-4)--(20.75, 3);
%\draw[red] (25.25,-4)--(25.25, 3);
\draw[red] (20.75,3) to [out=up, in=left] (23,5) to [out=right, in=up] (25.25, 3);

\draw (7,0)--(20,0) (21.5,0)--(24.5,0) (26,0)--(32,0);
\draw (7,-6)--(32,-6);
\draw (7,0) arc (90:270:3);
\draw (30+2,0) arc (90:-90:3);
\draw (8,-3) circle (1);
\draw (16,-3) circle (1);
\draw (12,-3) node{$\cdots$};
\draw (20,-2) arc (90:270:1);
\draw (21.5,-2) arc (90:-90:1);

\draw (21.5,-2) arc (90:-90:1);
\draw (29+2,-3) circle (1);
%\draw (33+2,-3) circle (1);
\draw [decorate,decoration={brace,mirror,amplitude=5pt},xshift=0pt,yshift=0]
	(8,-6.5) -- (16,-6.5) node [black,midway,yshift=-15pt] 
	{\footnotesize $q+2$};

%vertical part
\draw (20,-4)--(20,11.5);
\draw (21.5,-4)--(21.5,2.5);

\draw (21.5,2.5) arc (180:0:1.5);
\draw (20,11.5) arc (180:0:3);

\draw (23, 6.5) circle (1);
\draw (23, 10.5) circle (1);

\draw (24+2,-2) arc (90:-90:1);
\draw (24.5,-2) arc (90:270:1);

\draw (26,-4)--(26,11.5);
\draw (24.5,-4)--(24.5,2.5);
\draw (3,10) node[]{$\vdots$};

\end{scope}

\end{tikzpicture}

\caption{$y_i$ and $Y_i$ for $(i=2,\dots, q+3)$}
\label{hfamily4}
\end{figure}
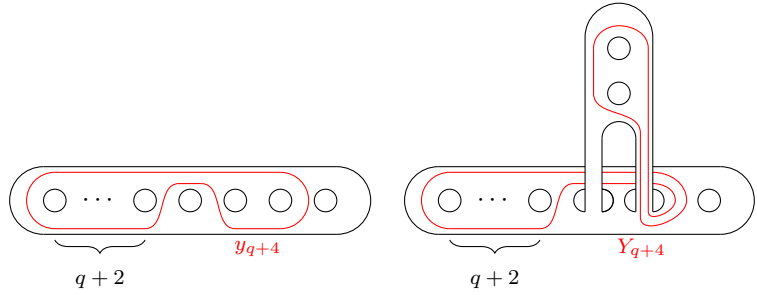
\begin{figure}[h]
\begin{tikzpicture}[scale=0.15]
\begin{scope}
\draw (7,0)--(33,0);
\draw (7,-6)--(33,-6);
\draw (7,0) arc (90:270:3);
\draw (33,0) arc (90:-90:3);
\draw (8,-3) circle (1);
%\draw[red] (8,-3) circle (1.5);
\draw (16,-3) circle (1);
%\draw[red] (12,-3) circle (1.5);

\draw[red] (28,-.5)--(8,-.5);
\draw[red] (8,-.5) arc (90:270:2.5);
\draw[red] (28,-.5) arc (90:-90:2.5);
\draw[red] (8,-5.5)--(16,-5.5) to [out=right,in=left] (19,-1.5)--(21,-1.5);
\draw[red] (21,-1.5) to[out=right, in=left] (24,-5.5)--(28,-5.5);

\draw[red] (26,-7.25) node {\footnotesize{$y_{q+4}$}};

\draw (12,-3) node{$\cdots$};
\draw (20,-3) circle (1);
%\draw[red] (20,-3) circle (1.5);

\draw (24,-3) circle (1);
\draw (28,-3) circle (1);
\draw (32,-3) circle (1);

\draw [decorate,decoration={brace,mirror,amplitude=5pt},xshift=0pt,yshift=0]
	(8,-6.5) -- (16,-6.5) node [black,midway,yshift=-15pt] 
	{\footnotesize $q+2$};
\end{scope}

\begin{scope}[shift={(35,0)}]

\draw[red] (8,-.5)--(20,-.5);
\draw[red] (21.5,-.5)--(24.5,-.5);
\draw[red] (21.5,-1.5)--(24.5,-1.5);

\draw[red] (8,-.5) arc (90:270:2.5);
\draw[red] (16,-5.5) to[out=right,in=left] (19,-1.5)--(20,-1.5);

\draw[red] (8,-5.5)--(16,-5.5);

%\draw[red] (26,-5.5)--(25.75,-5.5) to [out=left,in=down](20.75,-4)--(20.75,3);

%\draw[red] (25.25,3)--(25.25,-4.25) to[out=down,in=down] (27.5,-3) to[out=up,in=right] (26,-1.5);
\draw[red] (25,-7.25) node {\footnotesize{$Y_{q+4}$}};

%\draw[red] (20.75,-4)--(20.75, 3);
%\draw[red] (25.25,-4)--(25.25, 3);
\draw[red] (20.75,10.5) to [out=up, in=left] (23,12.5) to [out=right, in=up] (25.6, 10.5)--(25.6,-4) to [out=down,in=down] (28,-3) to [out=up,in=right] (26,-1.5);

\draw[red] (20.75,10.5)--(20.75,6.5) to [out=down,in=up] (24.9,3)--(24.9,-4.5) to [out=down,in=down] (29,-3)to [out=up,in=right] (26,-.5);
\draw (7,0)--(20,0) (21.5,0)--(24.5,0) (26,0)--(32,0);
\draw (7,-6)--(32,-6);
\draw (7,0) arc (90:270:3);
\draw (30+2,0) arc (90:-90:3);
\draw (8,-3) circle (1);
\draw (16,-3) circle (1);
\draw (12,-3) node{$\cdots$};
\draw (20,-2) arc (90:270:1);
\draw (21.5,-2) arc (90:-90:1);

\draw (21.5,-2) arc (90:-90:1);
\draw (29+2,-3) circle (1);
%\draw (33+2,-3) circle (1);
\draw [decorate,decoration={brace,mirror,amplitude=5pt},xshift=0pt,yshift=0]
	(8,-6.5) -- (16,-6.5) node [black,midway,yshift=-15pt] 
	{\footnotesize $q+2$};

%vertical part
\draw (20,-4)--(20,11.5);
\draw (21.5,-4)--(21.5,2.5);

\draw (21.5,2.5) arc (180:0:1.5);
\draw (20,11.5) arc (180:0:3);

\draw (23, 6.5) circle (1);
\draw (23, 10.5) circle (1);

\draw (24+2,-2) arc (90:-90:1);
\draw (24.5,-2) arc (90:270:1);

\draw (26,-4)--(26,11.5);
\draw (24.5,-4)--(24.5,2.5);
\end{scope}

\end{tikzpicture}
\caption{$y_{q+4}$ and $Y_{q+4}$}
\label{hfamily5}
\end{figure}
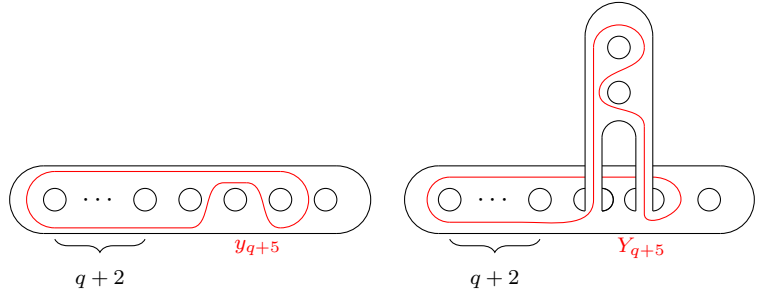
\begin{figure}
\begin{tikzpicture}[scale=0.15]

\begin{scope}

\draw[red] (28,-.5)--(8,-.5);
\draw[red] (8,-.5) arc (90:270:2.5);
\draw[red] (28,-.5) arc (90:-90:2.5);
\draw[red] (8,-5.5)--(20,-5.5) to [out=right,in=left] (23,-1.5)--(25,-1.5);
\draw[red] (25,-1.5) to[out=right, in=left] (28,-5.5);

\draw[red] (26,-7.25) node {\footnotesize{$y_{q+5}$}};

\draw (7,0)--(33,0);
\draw (7,-6)--(33,-6);
\draw (7,0) arc (90:270:3);
\draw (33,0) arc (90:-90:3);
\draw (8,-3) circle (1);
%\draw[red] (8,-3) circle (1.5);
\draw (16,-3) circle (1);
%\draw[red] (12,-3) circle (1.5);

\draw (12,-3) node{$\cdots$};
\draw (20,-3) circle (1);
%\draw[red] (20,-3) circle (1.5);

\draw (24,-3) circle (1);
\draw (28,-3) circle (1);
\draw (32,-3) circle (1);

\draw [decorate,decoration={brace,mirror,amplitude=5pt},xshift=0pt,yshift=0]
	(8,-6.5) -- (16,-6.5) node [black,midway,yshift=-15pt] 
	{\footnotesize $q+2$};
\end{scope}

\begin{scope}[shift={(35,0)}]

\draw[red] (8,-1)--(20,-1);
\draw[red] (21.5,-1.)--(24.5,-1.);
\draw[red] (25.25,3.5)--(25.25,3);

\draw[red] (8,-1) arc (90:270:2);
\draw[red] (8,-5)-- (14.375,-5) to [out=right, in=down](20.75,-3)--(20.75,10.5);

\draw[red] (25.25,3)--(25.25,-4.25) to[out=down,in=down] (28.5,-3) to[out=up,in=right] (26,-1.);
\draw[red] (25,-7.25) node {\footnotesize{$Y_{q+5}$}};
 
\draw[red] (25.25,10.5) to [out=down,in=up] (21.25,6.5) to[out=down,in=up] (25.25,3.5);

%\draw[red] (20.75,-4)--(20.75, 3);
%\draw[red] (25.25,-4)--(25.25, 3);
\draw[red] (20.75,10.5) to [out=up, in=left] (23,12.5) to [out=right, in=up] (25.25, 10.5);
%\draw[red] (25.25,-4) to [out=down, in= down] (27.75,-3);

%\draw[red] (20.75,-4) to[out=down,in=right] (18,-5)--(7.5,-5) to [out=left, in= down] (6,-3);
%\draw[red] (6,-3) to [out=up, in= left] (7.5,-1.5)--(20,-1.5);

%\draw[red] (27.75,-3) to [out=up, in=right] (26,-1.5);
%\draw[red] (21.5,-1.5)--(24.5,-1.5);

%\draw[red] (20,6) node[left]{\footnotesize{$(t_{\gamma_{q+3}}\cdot t_{\beta_1})(\alpha_{q+4})$}};

\draw (7,0)--(20,0) (21.5,0)--(24.5,0) (26,0)--(32,0);
\draw (7,-6)--(32,-6);
\draw (7,0) arc (90:270:3);
\draw (30+2,0) arc (90:-90:3);
\draw (8,-3) circle (1);
\draw (16,-3) circle (1);
\draw (12,-3) node{$\cdots$};
\draw (20,-2) arc (90:270:1);
\draw (21.5,-2) arc (90:-90:1);

\draw (21.5,-2) arc (90:-90:1);
\draw (29+2,-3) circle (1);
%\draw (33+2,-3) circle (1);
\draw [decorate,decoration={brace,mirror,amplitude=5pt},xshift=0pt,yshift=0]
	(8,-6.5) -- (16,-6.5) node [black,midway,yshift=-15pt] 
	{\footnotesize $q+2$};

%vertical part
\draw (20,-4)--(20,11.5);
\draw (21.5,-4)--(21.5,2.5);

\draw (21.5,2.5) arc (180:0:1.5);
\draw (20,11.5) arc (180:0:3);

\draw (23, 6.5) circle (1);
\draw (23, 10.5) circle (1);

\draw (24+2,-2) arc (90:-90:1);
\draw (24.5,-2) arc (90:270:1);

\draw (26,-4)--(26,11.5);
\draw (24.5,-4)--(24.5,2.5);
\end{scope}

\end{tikzpicture}
\caption{$y_{q+5}$ and $Y_{q+5}$}
\label{hfamily6}
\end{figure}
\begin{figure}[h]
\begin{tikzpicture}[scale=0.15]
\begin{scope}

\draw[red] (8,-1)--(24,-1);
\draw[red] (8,-1) arc (90:270:2);
\draw[red] (24,-1) arc (90:-90:2);
\draw[red] (8,-5)--(24,-5);

\draw[red] (26,-7.25) node {\footnotesize{$y_{q+6}$}};

\draw (7,0)--(33,0);
\draw (7,-6)--(33,-6);
\draw (7,0) arc (90:270:3);
\draw (33,0) arc (90:-90:3);
\draw (8,-3) circle (1);
%\draw[red] (8,-3) circle (1.5);
\draw (16,-3) circle (1);
%\draw[red] (12,-3) circle (1.5);

\draw (12,-3) node{$\cdots$};
\draw (20,-3) circle (1);
%\draw[red] (20,-3) circle (1.5);

\draw (24,-3) circle (1);
\draw (28,-3) circle (1);
\draw (32,-3) circle (1);

\draw [decorate,decoration={brace,mirror,amplitude=5pt},xshift=0pt,yshift=0]
	(8,-6.5) -- (16,-6.5) node [black,midway,yshift=-15pt] 
	{\footnotesize $q+2$};
\end{scope}

\begin{scope}[shift={(35,0)}]

\draw[red] (8,-1)--(20,-1);
\draw[red] (21.5,-1.)--(24.5,-1.);
\draw[red] (25.25,6.5)--(25.25,3);

\draw[red] (8,-1) arc (90:270:2);
\draw[red] (8,-5)-- (14.375,-5) to [out=right, in=down](20.75,-3)--(20.75,6.5);

\draw[red] (25.25,3)--(25.25,-4.25) to[out=down,in=down] (28.5,-3) to[out=up,in=right] (26,-1.);
\draw[red] (25,-7.25) node {\footnotesize{$Y_{q+6}$}};

%\draw[red] (20.75,-4)--(20.75, 3);
%\draw[red] (25.25,-4)--(25.25, 3);
\draw[red] (20.75,6.5) to [out=up, in=left] (23,8.5) to [out=right, in=up] (25.25, 6.5);

\draw (7,0)--(20,0) (21.5,0)--(24.5,0) (26,0)--(32,0);
\draw (7,-6)--(32,-6);
\draw (7,0) arc (90:270:3);
\draw (30+2,0) arc (90:-90:3);
\draw (8,-3) circle (1);
\draw (16,-3) circle (1);
\draw (12,-3) node{$\cdots$};
\draw (20,-2) arc (90:270:1);
\draw (21.5,-2) arc (90:-90:1);

\draw (21.5,-2) arc (90:-90:1);
\draw (29+2,-3) circle (1);
%\draw (33+2,-3) circle (1);
\draw [decorate,decoration={brace,mirror,amplitude=5pt},xshift=0pt,yshift=0]
	(8,-6.5) -- (16,-6.5) node [black,midway,yshift=-15pt] 
	{\footnotesize $q+2$};

%vertical part
\draw (20,-4)--(20,11.5);
\draw (21.5,-4)--(21.5,2.5);

\draw (21.5,2.5) arc (180:0:1.5);
\draw (20,11.5) arc (180:0:3);

\draw (23, 6.5) circle (1);
\draw (23, 10.5) circle (1);

\draw (24+2,-2) arc (90:-90:1);
\draw (24.5,-2) arc (90:270:1);

\draw (26,-4)--(26,11.5);
\draw (24.5,-4)--(24.5,2.5);
\end{scope}

\end{tikzpicture}
\caption{$y_{q+6}$ and $Y_{q+6}$}
\label{hfamily7}
\end{figure}
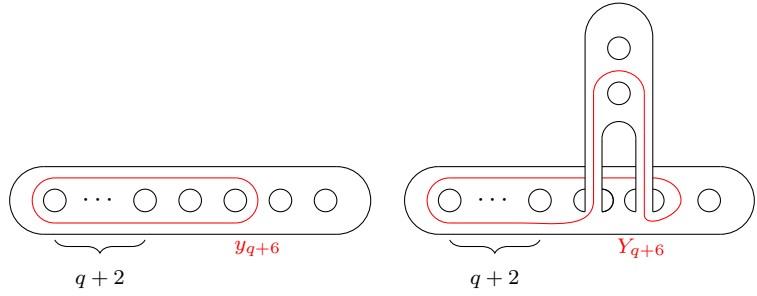
Then we have
\begin{eqnarray*}
&\phantom{0}& \hspace{-2 em}\underline{\beta_2\beta_3\alpha_1\alpha_2\cdots\alpha_{q+2}\gamma_{q+2} (t_{\beta_1}(\gamma_{q+3}))^{q+2} t_{\beta_1}(\alpha_{q+4})\delta \gamma_{q+4}}\cdot \gamma_{q+4}^{-1} t_\delta^{-1}(\beta_1)\alpha_{q+3}\alpha_{q+5}\gamma_{q+5}\\
&=& \hspace{-.55 em}t_{\gamma_{q+3}^{-1}}(Y_1)t_{\gamma_{q+3}^{-1}}(Y_2)\cdots t_{\gamma_{q+3}^{-1}}(Y_{q+6})\gamma_{q+4}^{-1} t_\delta^{-1}(\beta_1)\alpha_{q+3}\alpha_{q+5}\gamma_{q+5}\\
&=&\hspace{-.55 em}X_1\cdots X_{q+6}t_\delta^{-1}(\beta_1)\alpha_{q+3}\alpha_{q+5}\gamma_{q+5}\gamma_{q+4}^{-1}\\
&=&\hspace{-.55 em}X_1\cdots X_{q+4}(t_{X_{q+5}}\cdot t_{X_{q+6}}\cdot t_\delta^{-1})(\beta_1) X_{q+5}X_{q+6} \alpha_{q+3}\alpha_{q+5}\gamma_{q+5}\gamma_{q+4}^{-1}\\
&=&\hspace{-.55 em}X_1\cdots X_{q+4}(t_{X_{q+5}}\cdot t_{X_{q+6}}\cdot t_\delta^{-1})(\beta_1)  \alpha_{q+3}\alpha_{q+5}\gamma_{q+5}\gamma_{q+4}^{-1} t_{\alpha_{q+3}^{-1}}(X_{q+5})t_{\alpha_{q+3}^{-1}}(X_{q+6})
\end{eqnarray*}
One can easily see that $(t_{X_{q+5}}\cdot t_{X_{q+6}}\cdot t_\delta^{-1})(\beta_1)  \alpha_{q+3}\alpha_{q+5}\gamma_{q+5}=Z_{q}\cdot W_q\cdot\gamma_{q+4}$ for some simple closed curve $W_q$ and $Z_q$ due to the lantern relation (See Figure~\ref{hfamily5}). By performing a lantern substitution and cancelling $\gamma_{q+4}$ with $\gamma_{q+4}^{-1}$, we get a  monodromy factorization $W_{\Gamma_q}'$ whose length is $b_1(\Sigma_{\Gamma_{q}})$.
$$X_1\cdots X_{q+4}
\cdot
Z_q
\cdot
W_q
\cdot
t_{\alpha_{q+3}^{-1}}(X_{q+5})
\cdot
t_{\alpha_{q+3}^{-1}}(X_{q+6})$$
\begin{figure}[h]
\begin{tikzpicture}[scale=0.15]
\begin{scope}
\draw[red] (8,-5.5)--(26,-5.5);
\draw[red] (8,-1) arc (90:270:2.25);
\draw[red] (26,-1) arc (90:-90:2.25);
\draw[red] (8,-1)--(16,-1);
\draw[red] (16,-1) to [out=right,in=left] (19,-4.5);
\draw[red] (24.5,-1) to [out=left,in=right] (21.5,-4.5);
\draw[red] (21.5,-4.5)--(19,-4.5);
\draw[red] (27,-7.25) node {\footnotesize{$(t_{X_{q+5}}\cdot t_{X_{q+6}}\cdot t_\delta^{-1})(\beta_1)$}};

\draw (7,0)--(20,0) (21.5,0)--(24.5,0) (26,0)--(32,0);
\draw (7,-6)--(32,-6);
\draw (7,0) arc (90:270:3);
\draw (30+2,0) arc (90:-90:3);
\draw (8,-3) circle (1);
\draw (16,-3) circle (1);
\draw (12,-3) node{$\cdots$};
\draw (20,-2) arc (90:270:1);
\draw (21.5,-2) arc (90:-90:1);

\draw (21.5,-2) arc (90:-90:1);
\draw (29+2,-3) circle (1);
%\draw (33+2,-3) circle (1);
\draw [decorate,decoration={brace,mirror,amplitude=5pt},xshift=0pt,yshift=0]
	(8,-6.5) -- (16,-6.5) node [black,midway,yshift=-15pt] 
	{\footnotesize $q+2$};

%vertical part
\draw (20,-4)--(20,11.5);
\draw (21.5,-4)--(21.5,2.5);

\draw (21.5,2.5) arc (180:0:1.5);
\draw (20,11.5) arc (180:0:3);

\draw (23, 6.5) circle (1);
\draw (23, 10.5) circle (1);

\draw (24+2,-2) arc (90:-90:1);
\draw (24.5,-2) arc (90:270:1);

\draw (26,-4)--(26,11.5);
\draw (24.5,-4)--(24.5,2.5);
\end{scope}

\end{tikzpicture}
\caption{$(t_{X_{q+5}}\cdot t_{X_{q+6}}\cdot t_\delta^{-1})(\beta_1)$}
\label{hfamily5}
\end{figure}
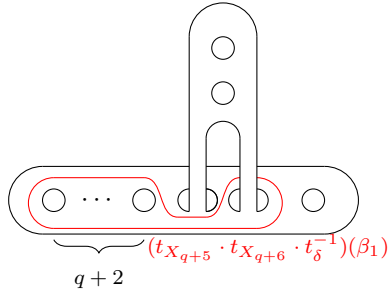

\subsection{Relations for (i) family}

\begin{figure}[h]
\begin{tikzpicture}[scale=0.15]

\draw (7,0)--(20,0) (21.5,0)--(24.5,0) (26,0)--(39,0);
\draw (7,-6)--(39,-6);
\draw (7,0) arc (90:270:3);
\draw (39,0) arc (90:-90:3);
\draw (7,-3) circle (1);
\draw (10,-3) circle (1);
\draw (13,-3) circle (1);
\draw (16,-3) circle (1);

\draw (20,-2) arc (90:270:1);
\draw (21.5,-2) arc (90:-90:1);

%\draw[red] (20,-1.5) arc (1.5);
\draw (21.5,-2) arc (90:-90:1);
%vertical part
\draw (20,-4)--(20,7.5);
\draw (21.5,-4)--(21.5,2.5);
%\draw[red] (20.75,-4)--(20.75,8.5);
%\draw[red] (20.75,8.5) arc (180:0:2.25 and 1.75);
%\draw[red] (20.75,-4)--(23,-5.4)--(25.25,-4);
%\draw[red] (20.75,-4) arc (180:360:2.25 and 1.25);

\draw (21.5,2.5) arc (180:0:1.5);
\draw (20,7.5) arc (180:0:3);

\draw (23, 6.5) circle (1);

\draw (23,18) node[]{};
\draw (24+2,-2) arc (90:-90:1);
\draw (24.5,-2) arc (90:270:1);

\draw (26,-4)--(26,7.5);
\draw (24.5,-4)--(24.5,2.5);

%vertical part 끝

\draw (30,-3) circle (1);
\draw (34,-3) node{$\cdots$};
\draw (38,-3) circle (1);

%\draw[blue] (20,-1.5)--(8,-1.5);
%\draw[blue] (21.5,-4.5)--(8,-4.5);
%\draw[blue] (8,-1.5) arc (90:270:1.5);
%\draw[blue] (21.5,-1.5) arc (90:-90:1.5);
\draw [decorate,decoration={brace,mirror,amplitude=5pt},xshift=0pt,yshift=0]
	(30,-6.5) -- (38,-6.5) node [black,midway,yshift=-15pt] 
	{\footnotesize $q$};

\begin{scope}[shift={(-26,-4)},scale=4.]
\node[bullet] at (0,0){};
\node[bullet] at (1,0){};
\node[bullet] at (2,0){};
\node[bullet] at (4,0){};
\node[bullet] at (5,0){};
\node[bullet] at (6,0){};

\node[bullet] at (1,1){};

\node[below] at (-0.,0){$-6$};
\node[below] at (1,0){$-2$};
\node[below] at (2,0){$-2$};
\node[below] at (4,0){$-2$};
\node[below] at (5,0){$-2$};
\node[above] at (6.25,0){\footnotesize{$-(q+3)$}};

\node[right] at (1,1){$-3$};

\node at (3,0){$\cdots$};

\draw (0,0)--(2.5,0);
\draw (3.5,0)--(6,0);
\draw (1,0)--(1,1);

	\draw [thick,decorate,decoration={brace,amplitude=5pt},xshift=0pt,yshift=7pt]
	(2,0) -- (4,0) node [black,midway,yshift=11pt] 
	{$q$};

\end{scope}

\end{tikzpicture}

\caption{Resolution graph $\Gamma_q$ and generic fiber $\Sigma_{\gamma_q}$ for (i) family}
\label{i}
\end{figure}
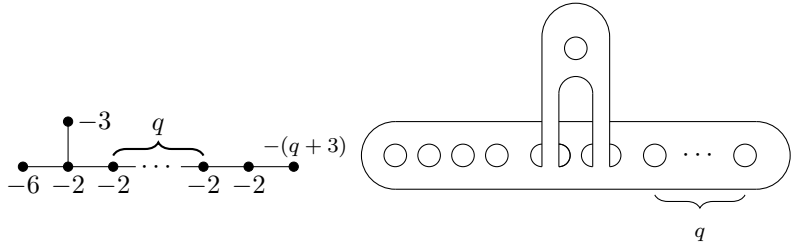
Let $\Gamma_q$ be a resolution graph of (i) family as in Figure~\ref{i}. Then the generic fiber for $X_{\Gamma_q}$ is $\Sigma_{q}$ as in Figure~\ref{i} and the global monodromy  of $X_{\Gamma_q}$ is given by
$$
\beta_1
\beta_2
\alpha_1
\alpha_2
\alpha_3
\alpha_4
\alpha_5
\gamma_5
\delta_2
\gamma_5^{q+2}
\alpha_6
\cdots
\alpha_{q+6}
\gamma_{q+6}
$$
We introduce a cancelling pair $\delta_{2}\cdot\delta_{2}^{-1}$ and rearrange the word using Hurwitz moves and braid relations $\delta_2\cdot\alpha_5\cdot\delta_2=\alpha_5\cdot\delta_2\cdot\alpha_5$ and $\delta_2\cdot\gamma_5\cdot\delta_2=\gamma_5\cdot\delta_2\cdot\gamma_5$.
\begin{eqnarray*}
&&
\beta_1
\beta_2
\cdot
\textcolor{red}{
\delta_2^{-1}
\cdot
\delta_2
}
\cdot
\alpha_1
\alpha_2
\alpha_3
\alpha_4
\alpha_5
\gamma_5
\delta_2
\gamma_5^{q+2}
\alpha_6
\cdots
\alpha_{q+6}
\gamma_{q+6}
\\
&=&
\beta_1
\beta_2
\delta_2^{-1}
\cdot
\alpha_1
\alpha_2
\alpha_3
\alpha_4
\cdot
\underline{
\delta_2
\alpha_5
\gamma_5
\delta_2
\gamma_5^{q+2}
}
\alpha_6
\cdots
\alpha_{q+6}
\gamma_{q+6}
\\
&=&
\beta_1
\beta_2
\delta_2^{-1}
\cdot
\alpha_1
\alpha_2
\alpha_3
\alpha_4
\cdot
\underline{
\delta_2
\alpha_5
\delta_2
\gamma_5
\delta_2
\gamma_5^{q+1}
}
\alpha_6
\cdots
\alpha_{q+6}
\gamma_{q+6}
\\
&=&
\beta_1
\beta_2
\delta_2^{-1}
\cdot
\alpha_1
\alpha_2
\alpha_3
\alpha_4
\cdot
\underline{
\alpha_5
\delta_2
\alpha_5
\gamma_5
\delta_2
\gamma_5^{q+1}
}
\alpha_6
\cdots
\alpha_{q+6}
\gamma_{q+6}
\\
&&
\phantom{000000000000000000000000000}
\vdots
\\
&=&
\beta_1
\beta_2
\delta_2^{-1}
\cdot
\alpha_1
\alpha_2
\alpha_3
\alpha_4
\cdot
\underline{
\alpha_5^{q+1}
\delta_2
\alpha_5
\gamma_5
\delta_2
\gamma_5
}
\alpha_6
\cdots
\alpha_{q+6}
\gamma_{q+6}
\\
&=&
\beta_1
\beta_2
\delta_2^{-1}
\cdot
\alpha_1
\alpha_2
\alpha_3
\alpha_4
\cdot
\alpha_5^{q+2}
\cdot
(t_{\alpha_5}^{-1})(\delta_2)
\cdot
\gamma_5^{2}
\cdot
(t_{\gamma_5}^{-1})(\delta_2)
\cdot
\alpha_6
\cdots
\alpha_{q+6}
\gamma_{q+6}
\\
&=&
\beta_1
\beta_2
\delta_2^{-1}
\underline{
\alpha_1
\alpha_2
\alpha_3
\alpha_4
\alpha_5^{q+2}
\gamma_5^{2}
\alpha_6
\cdots
\alpha_{q+6}
\gamma_{q+6}
}
\cdot
(t_{\alpha_6}^{-1}\cdot t_{\gamma_5}^{-2}\cdot t_{\alpha_5}^{-1})(\delta_2)
\cdot
(t_{\alpha_6}^{-1}\cdot t_{\gamma_5}^{-1})(\delta_2)
\end{eqnarray*}
The underlined part can be seen an embedding of the linear plumbing $L_q$ in Figure~\ref{i1}: Let $L_q$ be a linear plumbing and $\Sigma_{L_q}$ be a generic fiber for $X_{L_q}$ as in Figure~\ref{i1}. Then the monodromy for $X_{L_q}$ can be written as
$$
a_1a_2a_3a_4a_5^{q+2}b_5^2a_6\cdots a_{q+6}b_{q+6}
$$
where $a_i$ is simple closed curve in $\Sigma_{L_q}$ enclosing $i$th hole and $b_j$ is simple closed curve in $\Sigma_{L_q}$ enclosing from the first to $j$th holes. Then there is obvious planar subsurface of $\Sigma_{\Gamma_q}$ which is diffeomorphic to $\Sigma_{L_q}$ so that the image of each curves are
\begin{eqnarray*}
a_i &\rightarrow&  \alpha_i \phantom{0000}(i=1,\dots, q+6)
\\
b_{j} &\rightarrow& \gamma_{j} \phantom{0000}(i=5, q+6)
\end{eqnarray*}
\begin{figure}[h]
\begin{tikzpicture}[scale=0.8]
\draw (0,0) node[above]{$-2$};
\draw (1,0) node[]{$\cdots$};
\draw (2,0) node[above]{$-2$};
\draw (3,0) node[above]{$-6$};
\draw (4,0) node[above]{$-2$};

\draw (5,0) node[below]{$-(q+3)$};

\draw (0,0)--(.5,0) (1.5,0)--(5,0);
\draw (0,0 ) node[circle, fill, inner sep=1.5pt, black]{};
\draw (2,0 ) node[circle, fill, inner sep=1.5pt, black]{};

\draw (3,0 ) node[circle, fill, inner sep=1.5pt, black]{};
\draw (4,0 ) node[circle, fill, inner sep=1.5pt, black]{};
\draw (5,0 ) node[circle, fill, inner sep=1.5pt, black]{};

\draw (0,-.5) node{};

\draw [decorate,decoration={brace,mirror,amplitude=5pt},xshift=0pt,yshift=-5]
	(0,0) -- (2,0) node [black,midway,yshift=-15pt] 
	{\footnotesize $q+1$};

\draw (0,-1.5) node[]{};

\end{tikzpicture}
\begin{tikzpicture}
\begin{scope}[shift={(6,.75)},scale=0.16]
\draw (7,0)--(34,0);
\draw (7,-6)--(34,-6);
\draw (7,0) arc (90:270:3);
\draw (34,0) arc (90:-90:3);
\draw (7,-3) circle (1);
%\draw[red] (8,-3) circle (1.5);
\draw (10,-3) circle (1);
%\draw[red] (12,-3) circle (1.5);

\draw (13,-3) circle (1);
\draw (16,-3) circle (1);
\draw (19,-3) circle (1);
\draw (22,-3) circle (1);

%\draw[red] (20,-3) circle (1.5);

\draw (25,-3) circle (1);

\draw (28.,-3 ) node[circle, fill, inner sep=.5pt, black]{};
\draw (29.,-3 ) node[circle, fill, inner sep=.5pt, black]{};
\draw (30.,-3 ) node[circle, fill, inner sep=.5pt, black]{};

\draw (33,-3) circle (1);

\draw [decorate,decoration={brace,mirror,amplitude=5pt},xshift=0pt,yshift=0]
	(26,-6.5) -- (33,-6.5) node [black,midway,yshift=-15pt] 
	{\footnotesize $q$};

\end{scope}

\end{tikzpicture}
\caption{Linear plumbing $L_q$ and the generic fiber $\Sigma_{L_q}$ for $X_{L_q}$}
\label{i1}
\end{figure}
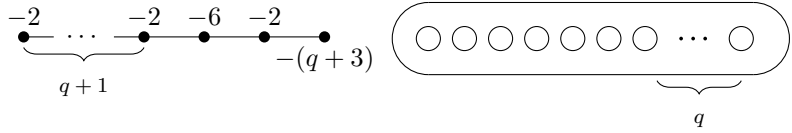
The linear plumbing $L_q$ can be rationally blowdown and the relation in the mapping class group of $\Sigma_{L_q}$ for the rational blowdown was given by Endo-Mark-Van Horn-Morris in 
~\cite{MR2783383}.
$$
a_1a_2a_3a_4a_5^{q+2}b_5^2a_6\cdots a_{q+6}b_{q+6}=y_1y_2\cdots y_{q+6}
$$
Let $Y_i$ a simple closed curve in $\Sigma_{\Gamma_q}$ be the image of $y_i$ in $\Sigma_{L_q}$ which can be drawn as in Figure~\ref{i2} to Figure~\ref{i5}. Then we have
\begin{eqnarray*}
&&
\beta_1
\beta_2
\delta_2^{-1}
\underline{
\alpha_1
\alpha_2
\alpha_3
\alpha_4
\alpha_5^{q+2}
\gamma_5^{2}
\alpha_6
\cdots
\alpha_{q+6}
\gamma_{q+6}
}
\cdot
(t_{\alpha_6}^{-1}\cdot t_{\gamma_5}^{-2}\cdot t_{\alpha_5}^{-1})(\delta_2)
\cdot
(t_{\alpha_6}^{-1}\cdot t_{\gamma_5}^{-1})(\delta_2)
\\
&=&
\beta_1
\beta_2
\delta_2^{-1}
Y_1
\cdots
Y_{q+5}
\cdot
Y_{q+6}
\cdot
(t_{\alpha_6}^{-1}\cdot t_{\gamma_5}^{-2}\cdot t_{\alpha_5}^{-1})(\delta_2)
\cdot
(t_{\alpha_6}^{-1}\cdot t_{\gamma_5}^{-1})(\delta_2)
\\
&=&
\beta_1
\beta_2
\delta_2^{-1}
Y_1
\cdots
Y_{q+5}
\cdot
Y_{q+6}
\cdot
(t_{\alpha_6}^{-1}\cdot t_{\gamma_5}^{-1})(\delta_2)
\cdot
(t_{\alpha_6}^{-1}\cdot t_{\alpha_5})(\delta_2)
\end{eqnarray*}
\begin{figure}[h]
\begin{tikzpicture}[scale=0.14]

\draw (7,0)--(20,0) (21.5,0)--(24.5,0) (26,0)--(39,0);
\draw (7,-6)--(39,-6);
\draw (7,0) arc (90:270:3);
\draw (39,0) arc (90:-90:3);
\draw (7,-3) circle (1);
\draw (10,-3) circle (1);
\draw (13,-3) circle (1);
\draw (16,-3) circle (1);

\draw[red] (10,-1)--(20,-1);
\draw[red] (10,-5)--(21.5,-5);
\draw[red] (21.5,-5) to [out=right, in=down] (23,-3)
to [out=up,in=right] (21.5,-1);
\draw[red] (10,-5) to [out=left, in=down] (8.5,-3)
to [out=up,in=left] (10,-1);
\draw[red] (13,0) node[above]{$Y_1$};

\draw (20,-2) arc (90:270:1);
\draw (21.5,-2) arc (90:-90:1);

%\draw[red] (20,-1.5) arc (1.5);
\draw (21.5,-2) arc (90:-90:1);
%vertical part
\draw (20,-4)--(20,7.5);
\draw (21.5,-4)--(21.5,2.5);
%\draw[red] (20.75,-4)--(20.75,8.5);
%\draw[red] (20.75,8.5) arc (180:0:2.25 and 1.75);
%\draw[red] (20.75,-4)--(23,-5.4)--(25.25,-4);
%\draw[red] (20.75,-4) arc (180:360:2.25 and 1.25);

\draw (21.5,2.5) arc (180:0:1.5);
\draw (20,7.5) arc (180:0:3);

\draw (23, 6.5) circle (1);

\draw (23,18) node[]{};
\draw (24+2,-2) arc (90:-90:1);
\draw (24.5,-2) arc (90:270:1);

\draw (26,-4)--(26,7.5);
\draw (24.5,-4)--(24.5,2.5);

%vertical part 끝

\draw (30,-3) circle (1);
\draw (34,-3) node{$\cdots$};
\draw (38,-3) circle (1);

%\draw[blue] (20,-1.5)--(8,-1.5);
%\draw[blue] (21.5,-4.5)--(8,-4.5);
%\draw[blue] (8,-1.5) arc (90:270:1.5);
%\draw[blue] (21.5,-1.5) arc (90:-90:1.5);
\draw [decorate,decoration={brace,mirror,amplitude=5pt},xshift=0pt,yshift=0]
	(30,-6.5) -- (38,-6.5) node [black,midway,yshift=-15pt] 
	{\footnotesize $q$};
\begin{scope}[shift={(41,0)}]

\draw (7,0)--(20,0) (21.5,0)--(24.5,0) (26,0)--(39,0);
\draw (7,-6)--(39,-6);
\draw (7,0) arc (90:270:3);
\draw (39,0) arc (90:-90:3);
\draw (7,-3) circle (1);
\draw (10,-3) circle (1);
\draw (13,-3) circle (1);
\draw (16,-3) circle (1);

\draw[red] (21.5,-5.5) to [out=right, in=down] (23,-3)
to [out=up,in=right] (21.5,-.5);
\draw[red] (7,-5.5) to [out=left, in=down] (5.5,-3)
to [out=up,in=left] (7,-.5);
\draw[red] (7,-5.5) to [out=right, in=down] (8.5,-3)
to [out=up, in=left] (10,-1)
to [out=right, in=up] (11.5,-3)
to [out=down, in=left] (13,-5.5);
\draw[red] (7,-.5)--(20,-.5) (13,-5.5)--(21.5,-5.5);
\draw[red] (13,0) node[above]{$Y_2$};

\draw (20,-2) arc (90:270:1);
\draw (21.5,-2) arc (90:-90:1);

%\draw[red] (20,-1.5) arc (1.5);
\draw (21.5,-2) arc (90:-90:1);
%vertical part
\draw (20,-4)--(20,7.5);
\draw (21.5,-4)--(21.5,2.5);
%\draw[red] (20.75,-4)--(20.75,8.5);
%\draw[red] (20.75,8.5) arc (180:0:2.25 and 1.75);
%\draw[red] (20.75,-4)--(23,-5.4)--(25.25,-4);
%\draw[red] (20.75,-4) arc (180:360:2.25 and 1.25);

\draw (21.5,2.5) arc (180:0:1.5);
\draw (20,7.5) arc (180:0:3);

\draw (23, 6.5) circle (1);

\draw (23,18) node[]{};
\draw (24+2,-2) arc (90:-90:1);
\draw (24.5,-2) arc (90:270:1);

\draw (26,-4)--(26,7.5);
\draw (24.5,-4)--(24.5,2.5);

%vertical part 끝

\draw (30,-3) circle (1);
\draw (34,-3) node{$\cdots$};
\draw (38,-3) circle (1);

%\draw[blue] (20,-1.5)--(8,-1.5);
%\draw[blue] (21.5,-4.5)--(8,-4.5);
%\draw[blue] (8,-1.5) arc (90:270:1.5);
%\draw[blue] (21.5,-1.5) arc (90:-90:1.5);
\draw [decorate,decoration={brace,mirror,amplitude=5pt},xshift=0pt,yshift=0]
	(30,-6.5) -- (38,-6.5) node [black,midway,yshift=-15pt] 
	{\footnotesize $q$};

\end{scope}

\end{tikzpicture}
\caption{$Y_1$ and $Y_2$ in $\Sigma_{\Gamma_q}$}
\label{i2}
\end{figure}
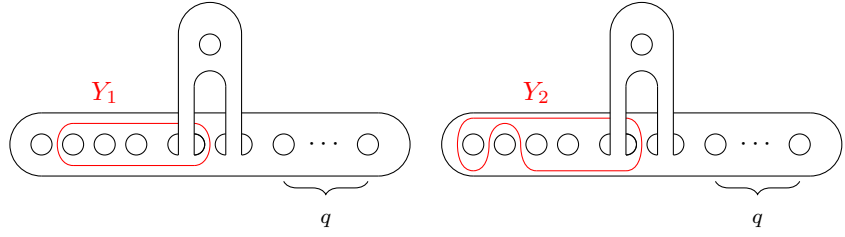
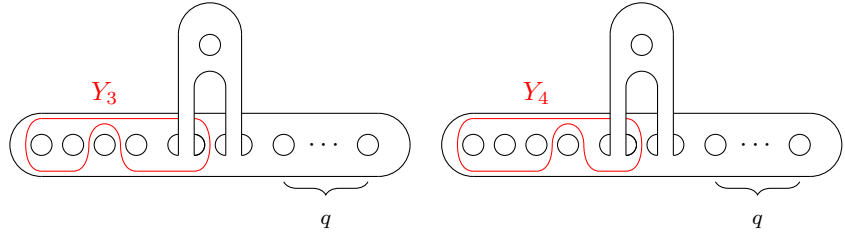
\begin{figure}[h]
\begin{tikzpicture}[scale=0.14]

\draw (7,0)--(20,0) (21.5,0)--(24.5,0) (26,0)--(39,0);
\draw (7,-6)--(39,-6);
\draw (7,0) arc (90:270:3);
\draw (39,0) arc (90:-90:3);
\draw (7,-3) circle (1);
\draw (10,-3) circle (1);
\draw (13,-3) circle (1);
\draw (16,-3) circle (1);

\draw[red] (21.5,-5.5) to [out=right, in=down] (23,-3)
to [out=up,in=right] (21.5,-.5);
\draw[red] (10,-5.5)--(7,-5.5) to [out=left, in=down] (5.5,-3)
to [out=up,in=left] (7,-.5);
\draw[red] (10,-5.5) to [out=right, in=down] (11.5,-3)
to [out=up, in=left] (13,-1)
to [out=right, in=up] (14.5,-3)
to [out=down, in=left] (16,-5.5);
\draw[red] (7,-.5)--(20,-.5) (16,-5.5)--(21.5,-5.5);
\draw[red] (13,0) node[above]{$Y_3$};

\draw (20,-2) arc (90:270:1);
\draw (21.5,-2) arc (90:-90:1);

%\draw[red] (20,-1.5) arc (1.5);
\draw (21.5,-2) arc (90:-90:1);
%vertical part
\draw (20,-4)--(20,7.5);
\draw (21.5,-4)--(21.5,2.5);
%\draw[red] (20.75,-4)--(20.75,8.5);
%\draw[red] (20.75,8.5) arc (180:0:2.25 and 1.75);
%\draw[red] (20.75,-4)--(23,-5.4)--(25.25,-4);
%\draw[red] (20.75,-4) arc (180:360:2.25 and 1.25);

\draw (21.5,2.5) arc (180:0:1.5);
\draw (20,7.5) arc (180:0:3);

\draw (23, 6.5) circle (1);

\draw (23,18) node[]{};
\draw (24+2,-2) arc (90:-90:1);
\draw (24.5,-2) arc (90:270:1);

\draw (26,-4)--(26,7.5);
\draw (24.5,-4)--(24.5,2.5);

%vertical part 끝

\draw (30,-3) circle (1);
\draw (34,-3) node{$\cdots$};
\draw (38,-3) circle (1);

%\draw[blue] (20,-1.5)--(8,-1.5);
%\draw[blue] (21.5,-4.5)--(8,-4.5);
%\draw[blue] (8,-1.5) arc (90:270:1.5);
%\draw[blue] (21.5,-1.5) arc (90:-90:1.5);
\draw [decorate,decoration={brace,mirror,amplitude=5pt},xshift=0pt,yshift=0]
	(30,-6.5) -- (38,-6.5) node [black,midway,yshift=-15pt] 
	{\footnotesize $q$};
\begin{scope}[shift={(41,0)}]

\draw (7,0)--(20,0) (21.5,0)--(24.5,0) (26,0)--(39,0);
\draw (7,-6)--(39,-6);
\draw (7,0) arc (90:270:3);
\draw (39,0) arc (90:-90:3);
\draw (7,-3) circle (1);
\draw (10,-3) circle (1);
\draw (13,-3) circle (1);
\draw (16,-3) circle (1);

\draw[red] (21.5,-5.5) to [out=right, in=down] (23,-3)
to [out=up,in=right] (21.5,-.5);
\draw[red] (13,-5.5)--(7,-5.5) to [out=left, in=down] (5.5,-3)
to [out=up,in=left] (7,-.5);
\draw[red] (13,-5.5) to [out=right, in=down] (14.5,-3)
to [out=up, in=left] (16,-1)
to [out=right, in=up] (17.5,-3)
to [out=down, in=left] (19,-5.5);
\draw[red] (7,-.5)--(20,-.5) (19,-5.5)--(21.5,-5.5);
\draw[red] (13,0) node[above]{$Y_4$};

\draw (20,-2) arc (90:270:1);
\draw (21.5,-2) arc (90:-90:1);

%\draw[red] (20,-1.5) arc (1.5);
\draw (21.5,-2) arc (90:-90:1);
%vertical part
\draw (20,-4)--(20,7.5);
\draw (21.5,-4)--(21.5,2.5);
%\draw[red] (20.75,-4)--(20.75,8.5);
%\draw[red] (20.75,8.5) arc (180:0:2.25 and 1.75);
%\draw[red] (20.75,-4)--(23,-5.4)--(25.25,-4);
%\draw[red] (20.75,-4) arc (180:360:2.25 and 1.25);

\draw (21.5,2.5) arc (180:0:1.5);
\draw (20,7.5) arc (180:0:3);

\draw (23, 6.5) circle (1);

\draw (23,18) node[]{};
\draw (24+2,-2) arc (90:-90:1);
\draw (24.5,-2) arc (90:270:1);

\draw (26,-4)--(26,7.5);
\draw (24.5,-4)--(24.5,2.5);

%vertical part 끝

\draw (30,-3) circle (1);
\draw (34,-3) node{$\cdots$};
\draw (38,-3) circle (1);

%\draw[blue] (20,-1.5)--(8,-1.5);
%\draw[blue] (21.5,-4.5)--(8,-4.5);
%\draw[blue] (8,-1.5) arc (90:270:1.5);
%\draw[blue] (21.5,-1.5) arc (90:-90:1.5);
\draw [decorate,decoration={brace,mirror,amplitude=5pt},xshift=0pt,yshift=0]
	(30,-6.5) -- (38,-6.5) node [black,midway,yshift=-15pt] 
	{\footnotesize $q$};

\end{scope}

\end{tikzpicture}
\caption{$Y_3$ and $Y_4$ in $\Sigma_{\Gamma_q}$}
\label{i3}
\end{figure}
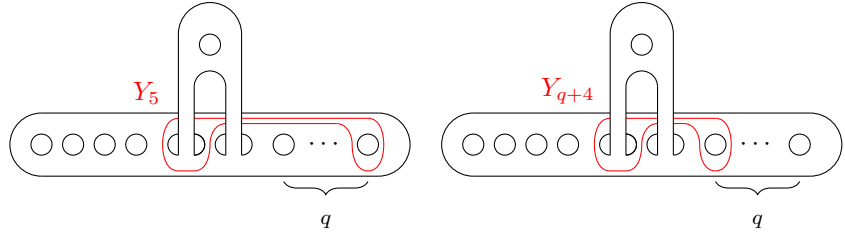
\begin{figure}[h]
\begin{tikzpicture}[scale=0.14]

\draw (7,0)--(20,0) (21.5,0)--(24.5,0) (26,0)--(39,0);
\draw (7,-6)--(39,-6);
\draw (7,0) arc (90:270:3);
\draw (39,0) arc (90:-90:3);
\draw (7,-3) circle (1);
\draw (10,-3) circle (1);
\draw (13,-3) circle (1);
\draw (16,-3) circle (1);

\draw[red] (20,-.5) to [out=left,in=up](18.5,-3)
to [out=down,in=left] (20,-5.5)--(21.5,-5.5)
to [out=right, in=down] (23,-3)
to [out=up,in=left] (24.5,-1) (24.5,-.5)--(21.5,-.5);

\draw[red] (38,-.5) to [out=right,in=up] (39.5,-3)
to [out=down,in=right] (38,-5.5)
to [out=left,in=down] (36.5,-3)
to [out=up,in=right] (35,-1);

\draw[red] (26,-.5)--(38,-.5);
\draw[red] (26,-1)--(35,-1);

\draw[red] (17,0) node[above]{$Y_5$};

\draw (20,-2) arc (90:270:1);
\draw (21.5,-2) arc (90:-90:1);

%\draw[red] (20,-1.5) arc (1.5);
\draw (21.5,-2) arc (90:-90:1);
%vertical part
\draw (20,-4)--(20,7.5);
\draw (21.5,-4)--(21.5,2.5);
%\draw[red] (20.75,-4)--(20.75,8.5);
%\draw[red] (20.75,8.5) arc (180:0:2.25 and 1.75);
%\draw[red] (20.75,-4)--(23,-5.4)--(25.25,-4);
%\draw[red] (20.75,-4) arc (180:360:2.25 and 1.25);

\draw (21.5,2.5) arc (180:0:1.5);
\draw (20,7.5) arc (180:0:3);

\draw (23, 6.5) circle (1);

\draw (23,18) node[]{};
\draw (24+2,-2) arc (90:-90:1);
\draw (24.5,-2) arc (90:270:1);

\draw (26,-4)--(26,7.5);
\draw (24.5,-4)--(24.5,2.5);

%vertical part 끝

\draw (30,-3) circle (1);
\draw (34,-3) node{$\cdots$};
\draw (38,-3) circle (1);

%\draw[blue] (20,-1.5)--(8,-1.5);
%\draw[blue] (21.5,-4.5)--(8,-4.5);
%\draw[blue] (8,-1.5) arc (90:270:1.5);
%\draw[blue] (21.5,-1.5) arc (90:-90:1.5);
\draw [decorate,decoration={brace,mirror,amplitude=5pt},xshift=0pt,yshift=0]
	(30,-6.5) -- (38,-6.5) node [black,midway,yshift=-15pt] 
	{\footnotesize $q$};
\begin{scope}[shift={(41,0)}]

\draw (7,0)--(20,0) (21.5,0)--(24.5,0) (26,0)--(39,0);
\draw (7,-6)--(39,-6);
\draw (7,0) arc (90:270:3);
\draw (39,0) arc (90:-90:3);
\draw (7,-3) circle (1);
\draw (10,-3) circle (1);
\draw (13,-3) circle (1);
\draw (16,-3) circle (1);

\draw[red] (20,-.5) to [out=left,in=up](18.5,-3)
to [out=down,in=left] (20,-5.5)--(21.5,-5.5)
to [out=right, in=down] (23,-3)
to [out=up,in=left] (24.5,-1) (24.5,-.5)--(21.5,-.5);

\draw[red] (30,-.5) to [out=right,in=up] (31.5,-3)
to [out=down,in=right] (30,-5.5)
to [out=left,in=down] (28.5,-3)
to [out=up,in=right] (27,-1);

\draw[red] (26,-.5)--(30,-.5);
\draw[red] (26,-1)--(27,-1);
\draw[red] (16,0) node[above]{$Y_{q+4}$};

\draw (20,-2) arc (90:270:1);
\draw (21.5,-2) arc (90:-90:1);

%\draw[red] (20,-1.5) arc (1.5);
\draw (21.5,-2) arc (90:-90:1);
%vertical part
\draw (20,-4)--(20,7.5);
\draw (21.5,-4)--(21.5,2.5);
%\draw[red] (20.75,-4)--(20.75,8.5);
%\draw[red] (20.75,8.5) arc (180:0:2.25 and 1.75);
%\draw[red] (20.75,-4)--(23,-5.4)--(25.25,-4);
%\draw[red] (20.75,-4) arc (180:360:2.25 and 1.25);

\draw (21.5,2.5) arc (180:0:1.5);
\draw (20,7.5) arc (180:0:3);

\draw (23, 6.5) circle (1);

\draw (23,18) node[]{};
\draw (24+2,-2) arc (90:-90:1);
\draw (24.5,-2) arc (90:270:1);

\draw (26,-4)--(26,7.5);
\draw (24.5,-4)--(24.5,2.5);

%vertical part 끝

\draw (30,-3) circle (1);
\draw (34,-3) node{$\cdots$};
\draw (38,-3) circle (1);

%\draw[blue] (20,-1.5)--(8,-1.5);
%\draw[blue] (21.5,-4.5)--(8,-4.5);
%\draw[blue] (8,-1.5) arc (90:270:1.5);
%\draw[blue] (21.5,-1.5) arc (90:-90:1.5);
\draw [decorate,decoration={brace,mirror,amplitude=5pt},xshift=0pt,yshift=0]
	(30,-6.5) -- (38,-6.5) node [black,midway,yshift=-15pt] 
	{\footnotesize $q$};

\end{scope}

\end{tikzpicture}
\caption{$Y_i$ in $\Sigma_{\Gamma_q}$ $(i=5,\dots, {q+4})$}
\label{i4}
\end{figure}
\begin{figure}[h]
\begin{tikzpicture}[scale=0.14]

\draw (7,0)--(20,0) (21.5,0)--(24.5,0) (26,0)--(39,0);
\draw (7,-6)--(39,-6);
\draw (7,0) arc (90:270:3);
\draw (39,0) arc (90:-90:3);
\draw (7,-3) circle (1);
\draw (10,-3) circle (1);
\draw (13,-3) circle (1);
\draw (16,-3) circle (1);

\draw[red] (21.5,-1)--(24.5,-1);
\draw[red] (20,-5)--(26,-5);
\draw[red] (26,-5) to [out=right, in=down] (27.5,-3)
to [out=up,in=right] (26,-1);
\draw[red] (20,-5) to [out=left, in=down] (18.5,-3)
to [out=up,in=left] (20,-1);
\draw[red] (16,0) node[above]{$Y_{q+5}$};

\draw (20,-2) arc (90:270:1);
\draw (21.5,-2) arc (90:-90:1);

%\draw[red] (20,-1.5) arc (1.5);
\draw (21.5,-2) arc (90:-90:1);
%vertical part
\draw (20,-4)--(20,7.5);
\draw (21.5,-4)--(21.5,2.5);
%\draw[red] (20.75,-4)--(20.75,8.5);
%\draw[red] (20.75,8.5) arc (180:0:2.25 and 1.75);
%\draw[red] (20.75,-4)--(23,-5.4)--(25.25,-4);
%\draw[red] (20.75,-4) arc (180:360:2.25 and 1.25);

\draw (21.5,2.5) arc (180:0:1.5);
\draw (20,7.5) arc (180:0:3);

\draw (23, 6.5) circle (1);

\draw (23,18) node[]{};
\draw (24+2,-2) arc (90:-90:1);
\draw (24.5,-2) arc (90:270:1);

\draw (26,-4)--(26,7.5);
\draw (24.5,-4)--(24.5,2.5);

%vertical part 끝

\draw (30,-3) circle (1);
\draw (34,-3) node{$\cdots$};
\draw (38,-3) circle (1);

%\draw[blue] (20,-1.5)--(8,-1.5);
%\draw[blue] (21.5,-4.5)--(8,-4.5);
%\draw[blue] (8,-1.5) arc (90:270:1.5);
%\draw[blue] (21.5,-1.5) arc (90:-90:1.5);
\draw [decorate,decoration={brace,mirror,amplitude=5pt},xshift=0pt,yshift=0]
	(30,-6.5) -- (38,-6.5) node [black,midway,yshift=-15pt] 
	{\footnotesize $q$};
\begin{scope}[shift={(39,0)}]

\draw (7,0)--(20,0) (21.5,0)--(24.5,0) (26,0)--(39,0);
\draw (7,-6)--(39,-6);
\draw (7,0) arc (90:270:3);
\draw (39,0) arc (90:-90:3);
\draw (7,-3) circle (1);
\draw (10,-3) circle (1);
\draw (13,-3) circle (1);
\draw (16,-3) circle (1);

\draw[red] (38,-5.5) to [out=right, in=down] (39.5,-3)
to [out=up,in=right] (38,-.5);
\draw[red] (7,-5.5) to [out=left, in=down] (5.5,-3)
to [out=up,in=left] (7,-.5);
\draw[red] (17,-5.5) to [out=right, in=down] (18.5,-3)
to [out=up, in=left] (20,-1)
;
\draw[red] (7,-.5)--(20,-.5) (17,-5.5)--(7,-5.5);
\draw[red] (13,0) node[above]{$Y_{q+6}$};

\draw[red] (21.5,-.5)--(24.5,-.5) (26,-.5)--(38,-.5);
\draw[red] (21.5,-1) to [out=right,in=up] (23,-3)
to [out=down, in=left] (24.5,-5.5)--(38,-5.5) ;
\draw (20,-2) arc (90:270:1);
\draw (21.5,-2) arc (90:-90:1);

%\draw[red] (20,-1.5) arc (1.5);
\draw (21.5,-2) arc (90:-90:1);
%vertical part
\draw (20,-4)--(20,7.5);
\draw (21.5,-4)--(21.5,2.5);
%\draw[red] (20.75,-4)--(20.75,8.5);
%\draw[red] (20.75,8.5) arc (180:0:2.25 and 1.75);
%\draw[red] (20.75,-4)--(23,-5.4)--(25.25,-4);
%\draw[red] (20.75,-4) arc (180:360:2.25 and 1.25);

\draw (21.5,2.5) arc (180:0:1.5);
\draw (20,7.5) arc (180:0:3);

\draw (23, 6.5) circle (1);

\draw (23,18) node[]{};
\draw (24+2,-2) arc (90:-90:1);
\draw (24.5,-2) arc (90:270:1);

\draw (26,-4)--(26,7.5);
\draw (24.5,-4)--(24.5,2.5);

%vertical part 끝

\draw (30,-3) circle (1);
\draw (34,-3) node{$\cdots$};
\draw (38,-3) circle (1);

%\draw[blue] (20,-1.5)--(8,-1.5);
%\draw[blue] (21.5,-4.5)--(8,-4.5);
%\draw[blue] (8,-1.5) arc (90:270:1.5);
%\draw[blue] (21.5,-1.5) arc (90:-90:1.5);
\draw [decorate,decoration={brace,mirror,amplitude=5pt},xshift=0pt,yshift=0]
	(30,-6.5) -- (38,-6.5) node [black,midway,yshift=-15pt] 
	{\footnotesize $q$};

\end{scope}

\end{tikzpicture}
\caption{$Y_{q+5}$ and $Y_{q+6}$ in $\Sigma_{\Gamma_q}$}
\label{i5}
\end{figure}
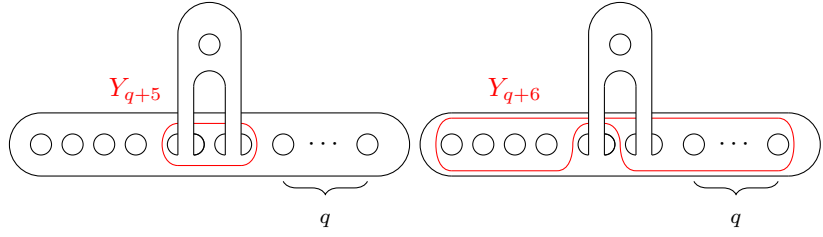
Since
\begin{eqnarray*}
&&
(t_{\alpha_6}^{-1}\cdot t_{\gamma_5}^{-2}\cdot t_{\alpha_5}^{-1})(\delta_2)
\cdot
(t_{\alpha_6}^{-1}\cdot t_{\gamma_5}^{-1})(\delta_2)
\\
&=&
\alpha_6^{-1}
\gamma_5^{-2}
\alpha_5^{-1}
\delta_2
\alpha_5
\gamma_5
\delta_2
\gamma_5
\alpha_6
\\
&=&
\alpha_6^{-1}
\gamma_5^{-2}
\alpha_5^{-1}
\delta_2
\alpha_5
\delta_2
\gamma_5
\delta_2
\alpha_6
\\
&=&
\alpha_6^{-1}
\gamma_5^{-2}
\alpha_5^{-1}
\alpha_5
\delta_2
\alpha_5
\gamma_5
\delta_2
\alpha_6
\\
&=&
\alpha_6^{-1}
\gamma_5^{-2}
\delta_2
\gamma_5
\alpha_5
\delta_2
\alpha_5
\alpha_5^{-1}
\alpha_6
\\
&=&
\alpha_6^{-1}
\gamma_5^{-2}
\delta_2
\gamma_5
\delta_2
\alpha_5
\delta_2
\alpha_5^{-1}
\alpha_6
\\
&=&
\alpha_6^{-1}
\gamma_5^{-2}
\gamma_5
\delta_2
\gamma_5
\alpha_5
\delta_2
\alpha_5^{-1}
\alpha_6
\\
&=&
\alpha_6^{-1}
\gamma_5^{-1}
\delta_2
\gamma_5
\alpha_6
\cdot
\alpha_6^{-1}
\alpha_5
\delta_2
\alpha_5^{-1}
\alpha_6
\\
&=&
(t_{\alpha_6}^{-1}\cdot t_{\gamma_5}^{-1})(\delta_2)
\cdot
(t_{\alpha_6}^{-1}\cdot t_{\alpha_5})(\delta_2)
\end{eqnarray*}
Note that
$\beta_1\cdot\beta_2\cdot Y_{q+5} \cdot (t_{\alpha_6}^{-1} \cdot t_{\alpha_5})(\delta_2)
=\delta_2\cdot Z\cdot W$ for some simple closed curves $W$ and $Z$ in $\Sigma_{\Gamma_q}$ due to the lantern relation. See Figure~\ref{i6} for corresponding curves in a planar surface. By performing a lantern substitution and cancelling $\delta_2$ with $\delta_2^{-1}$ after rearranging the word, we get a monodromy factorization $W_{\Gamma_q}'$ whose length is $b_1(\Sigma_{\Gamma_q})$.

\begin{eqnarray*}
&&
\beta_1
\beta_2
\delta_2^{-1}
Y_1
\cdots
Y_{q+5}
\cdot
Y_{q+6}
\cdot
(t_{\alpha_6}^{-1}\cdot t_{\gamma_5}^{-1})(\delta_2)
\cdot
(t_{\alpha_6}^{-1}\cdot t_{\alpha_5})(\delta_2)
\\
&=&
\beta_1
\beta_2
\delta_2^{-1}
Y_1
\cdots
Y_{q+4}
\cdot
t_{Y_{q+5}}(Y_{q+6})
\cdot
(t_{Y_{q+5}}\cdot t_{\alpha_6}^{-1}\cdot t_{\gamma_5}^{-1})(\delta_2)
\cdot
Y_{q+5}
\cdot
(t_{\alpha_6}^{-1}\cdot t_{\alpha_5})(\delta_2)
\\
&=&
(t_{\delta_2}^{-1}\cdot t_{\beta_1})(Y_1)
\cdot
(t_{\delta_2}^{-1}\cdot t_{\beta_1})(Y_2)
\cdots
(t_{\delta_2}^{-1}\cdot t_{\beta_1})(Y_{q+4})
\cdot
(t_{\delta_2}^{-1}\cdot t_{\beta_1} \cdot t_{Y_{q+5}})(Y_{q+6})
\\
&&
\cdot
(t_{\delta_2}^{-1} \cdot t_{\beta_1} \cdot t_{Y_{q+5}}\cdot t_{\alpha_6}^{-1}\cdot t_{\gamma_5}^{-1})(\delta_2)
\cdot
\delta_2^{-1}
\cdot
\beta_1
\cdot
\beta_2
\cdot
Y_{q+5}
\cdot
(t_{\alpha_6}^{-1} \cdot t_{\alpha_5})(\delta_2)
\\
&=&
(t_{\delta_2}^{-1}\cdot t_{\beta_1})(Y_1)
\cdot
(t_{\delta_2}^{-1}\cdot t_{\beta_1})(Y_2)
\cdots
(t_{\delta_2}^{-1}\cdot t_{\beta_1})(Y_{q+4})
\cdot
(t_{\delta_2}^{-1}\cdot t_{\beta_1} \cdot t_{Y_{q+5}})(Y_{q+6})
\\
&&
\cdot
(t_{\delta_2}^{-1} \cdot t_{\beta_1} \cdot t_{Y_{q+5}}\cdot t_{\alpha_6}^{-1}\cdot t_{\gamma_5}^{-1})(\delta_2)
\cdot
Z
\cdot
W
\end{eqnarray*}
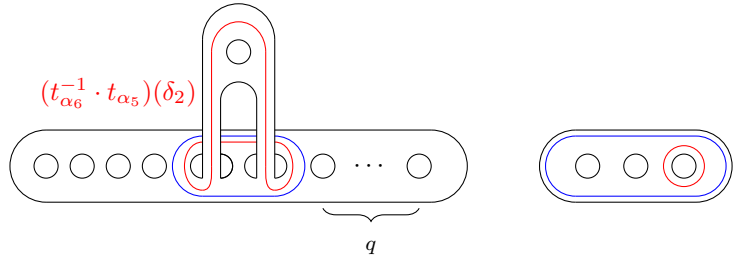
\begin{figure}[h]
\begin{tikzpicture}[scale=0.16]

\draw (7,0)--(20,0) (21.5,0)--(24.5,0) (26,0)--(39,0);
\draw (7,-6)--(39,-6);
\draw (7,0) arc (90:270:3);
\draw (39,0) arc (90:-90:3);
\draw (7,-3) circle (1);
\draw (10,-3) circle (1);
\draw (13,-3) circle (1);
\draw (16,-3) circle (1);

\draw[red] (21.5,-1)--(24.5,-1);
\draw[red] (20,-5) to[out=right,in=down] (20.75,-3);
\draw[red] (25.25,-3) to[out=down,in=left] (26,-5);
\draw[red] (25.25,-3)--(25.25,6.5) (20.75,-3)--(20.75,6.5);
\draw[red] (25.25,6.5) to [out=up,in=right] (23,9)
to [out=left,in=up](20.75,6.5);

\draw[red] (26,-5) to [out=right, in=down] (27.5,-3)
to [out=up,in=right] (26,-1);
\draw[red] (20,-5) to [out=left, in=down] (18.5,-3)
to [out=up,in=left] (20,-1);
\draw[red] (13,1) node[above]{$(t_{\alpha_6}^{-1} \cdot t_{\alpha_5})(\delta_2)$};

\draw[blue] (20,-.5) arc (90:270:2.5);
\draw[blue] (26,-.5) arc (90:-90:2.5);
\draw[blue] (21.5,-.5)--(24.5,-.5) (20,-5.5)--(26,-5.5);

\draw (20,-2) arc (90:270:1);
\draw (21.5,-2) arc (90:-90:1);

%\draw[red] (20,-1.5) arc (1.5);
\draw (21.5,-2) arc (90:-90:1);
%vertical part
\draw (20,-4)--(20,7.5);
\draw (21.5,-4)--(21.5,2.5);
%\draw[red] (20.75,-4)--(20.75,8.5);
%\draw[red] (20.75,8.5) arc (180:0:2.25 and 1.75);
%\draw[red] (20.75,-4)--(23,-5.4)--(25.25,-4);
%\draw[red] (20.75,-4) arc (180:360:2.25 and 1.25);

\draw (21.5,2.5) arc (180:0:1.5);
\draw (20,7.5) arc (180:0:3);

\draw (23, 6.5) circle (1);

\draw (23,18) node[]{};
\draw (24+2,-2) arc (90:-90:1);
\draw (24.5,-2) arc (90:270:1);

\draw (26,-4)--(26,7.5);
\draw (24.5,-4)--(24.5,2.5);

%vertical part 끝

\draw (30,-3) circle (1);
\draw (34,-3) node{$\cdots$};
\draw (38,-3) circle (1);

%\draw[blue] (20,-1.5)--(8,-1.5);
%\draw[blue] (21.5,-4.5)--(8,-4.5);
%\draw[blue] (8,-1.5) arc (90:270:1.5);
%\draw[blue] (21.5,-1.5) arc (90:-90:1.5);
\draw [decorate,decoration={brace,mirror,amplitude=5pt},xshift=0pt,yshift=0]
	(30,-6.5) -- (38,-6.5) node [black,midway,yshift=-15pt] 
	{\footnotesize $q$};
\begin{scope}[shift={(32,0)}]
\draw (19,4) node[]{};
\draw (19,0)--(29,0);
\draw (19,-6)--(29,-6);
\draw (19,0) arc (90:270:3);
\draw (29,0) arc (90:-90:3);

\draw (20,-3) circle (1);
\draw (24,-3) circle (1);
\draw (28,-3) circle (1);
\draw[red] (28,-3) circle (1.7);

\draw[blue] (19,-.5)--(29,-.5);
\draw[blue] (19,-5.5)--(29,-5.5);
\draw[blue] (19,-.5) arc (90:270:2.5);
\draw[blue] (29,-.5) arc (90:-90:2.5);

\end{scope}	

\end{tikzpicture}
\caption{Corresponding curves for the lantern relation}
\label{i6}
\end{figure}

 \end{document}